\documentclass[11pt,bibliography=totoc,reqno]{scrartcl}

\usepackage[a4paper,left=2cm,right=2cm,top=3cm,bottom=3cm]{geometry}

\usepackage[ansinew]{inputenc}
\usepackage[T1]{fontenc}
\usepackage[english]{babel}
\usepackage{marvosym}
\usepackage{stmaryrd}
\usepackage{caption}

\usepackage{units}
\usepackage{ifsym}
\usepackage[intlimits,leqno]{amsmath}
\usepackage{amsthm}
\usepackage{amssymb}
\usepackage{amsfonts}
\usepackage{dsfont}
\usepackage{exscale}
\usepackage{txfonts}
\usepackage{pxfonts}
\usepackage{bbold}
\usepackage{bbm}
\usepackage{wrapfig}
\usepackage{wasysym}
\usepackage{mathrsfs}
\usepackage{setspace}
\usepackage{todonotes}
\usepackage[all]{xy}
\usepackage{float}

\usepackage[numbers]{natbib}

\usepackage{multicol}

\newcommand{\e}{\operatorname{e}}
\newcommand{\rd}{\mathrm{d}}

\newcommand{\im}{\operatorname{im}}
\newcommand{\id}{\operatorname{id}}

\newcommand{\diag}{\operatorname{diag}}

\DeclareMathOperator{\Span}{span}
\DeclareMathOperator{\tr}{tr}

\newcommand{\Hom}{\operatorname{Hom}}
\newcommand{\sign}{\operatorname{sign}}
\newcommand{\ev}{\operatorname{ev}}
\newcommand{\ad}{\operatorname{ad}}

\newtheorem{theorem}{Theorem}%[section]
\newtheorem{remark}{Remark}%[section]
\newtheorem{lemma}[remark]{Lemma}%[section]
\newtheorem{corollary}[remark]{Corollary}%[section]
\theoremstyle{definition}
\newtheorem{definition}[remark]{Definition}%[section]
\newtheorem{example}[remark]{Beispiel}%[section]

\usepackage{scrpage2} %für Kopfzeile
\automark[subsection]{section}

\begin{document}
\pagenumbering{arabic} %%Seitennummerierung
%\chead{\headmark}%Kopfzeile - Titel ausschreiben

%--------------------------------------------------------------------------------------------------------------------------------

\title{Symplectic Lie algebras with degenerate center\footnote{funded by the Deutsche Forschungsgemeinschaft (DFG, German Research Foundation) - KA 1065/7-1}}
\author{M. Fischer}
\maketitle
\onehalfspacing %1,5 Zeilenabstand
%\tableofcontents %Inhaltsverzeichnis
%\newpage
%
%\pagestyle{scrheadings} %Kopfzeile
%\pagenumbering{arabic}%Seitennummerierung
%\onehalfspacing

\abstract{Every symplectic Lie algebra with degenerate (including non-abelian nilpotent symplectic Lie algebras) has the structure of a quadratic extension. We give a standard model and describe the equivalence classes on the level of corresponding quadratic cohomology sets. Finally, we give a scheme to classify the isomorphism classes of symplectic Lie algebras with degenerate center. We also give a classification scheme for nilpotent symplectic Lie algebras and compute a complete list of all $6$-dimensional nilpotent symplectic Lie algebras, which removes some little inaccuracies in an older list.}

\section{Introduction}
In this paper we concentrate on symplectic Lie algebras $(\mathfrak g,\langle\cdot,\cdot\rangle,\omega)$.
Here a symplectic Lie algebra $(\mathfrak g,\omega)$ is a Lie algebra $\mathfrak g$ admitting a closed non-degenerate $2$-form $\omega$ on $\mathfrak g$, which we call symplectic form.
Two symplectic Lie algebras $(\mathfrak g_i,\omega_i)$, $i=1,2$ are isomorphic if there is an isomorphism $\varphi:\mathfrak g_1\rightarrow\mathfrak g_2$ of Lie algebras which preserves the symplectic forms in the sense $\varphi^*\omega_2=\omega_1$.
Symplectic Lie algebras are in one-to-one correspondence with simply connected Lie groups with leftinvariant symplectic forms.
These symplectic Lie groups are interesting, since they define compact symplectic spaces, as long as they posses discrete cocompact subgroups, and thus are a rich source of non-Kähler symplectic manifolds (see for example \cite{Th76}, \cite{BG88}, \cite{BC06}).
An interesting subset of  symplectic Lie algebras are Kähler respective pseudo-Riemannian Kähler Lie algebras, which were considered for example in \cite{DM962} and \cite{CFU04}.
Symplectic Lie algebras are also called quasi-Frobenius Lie algebras, since Frobenius Lie Algebras, i. e. Lie algebras admitting an non-degenerate exact $2$-form, are examples of symplectic Lie algebras.
Moreover, symplectic Lie algebras are examples of Vinberg algebras, since they naturally carry an affine structure.

There are several classifications of symplectic Lie algebras in low dimension (\cite{Sa01}, \cite{GJK01}, \cite{Ov06}, \cite{CS09}) and for the nilpotent case (\cite{Bu06}, \cite{GKM04}, see also \cite{GB87}).

If the aim is to give informations for arbitrary dimension, then usually a reduction scheme is used. In this context we mention 
symplectic double extensions (\cite{MR91}, \cite{DM96}, \cite{DM962}) and oxidation \cite{BC13}.
The main idea is, that for every isotropic ideal $\mathfrak j$ of a symplectic Lie algebra $\mathfrak g$ the Lie algebra $\overline{\mathfrak g}=\mathfrak j^\bot/\mathfrak j$ inherits a symplectic form from $\mathfrak g$.
Conversely, they give a construction scheme taking a symplectic Lie algebra $\overline{\mathfrak{g}}$ and some additional structure and construct a higher dimensional symplectic Lie algebra, which can be reduced to $\overline{\mathfrak{g}}$ again.
Since the choice of the isotropic ideal is in general not canonical, it is hard to to give a general statement on the isomorphy of these Lie algebras with the help of these schemes. For instance, it is possible that extensions of two non-isomorphic low dimensional symplectic Lie algebras are isomorphic.
Moreover, the presentation of a symplectic Lie algebra as such an extension is not unique, since it depends on the choosen ideal.

The aim of this paper is to take this choices canonically for symplectic Lie algebras such that there is a certain standard model and the possibility to analyse the isomorphy systematically with this standard model.
However, we cannot use this scheme for every symplectic Lie algebra. We restrict our observations to symplectic Lie algebras with degenerate center. This includes for instance non-abelian nilpotent symplectic Lie algebras. 
\\\quad\\
In the following, we describe the main idea of the present paper.
Therefor, we consider symplectic Lie algebras $(\mathfrak g,\omega)$ with degenerate center $\mathfrak z$. Then $\mathfrak j:=\mathfrak z\cap\mathfrak z^\bot$ is an nontrivial ideal of $\mathfrak g$, which is isotropic with respect to $\omega$. Moreover, $\omega$ induces a symplectic form $\overline\omega$ on $\mathfrak a:=\mathfrak j^\bot/\mathfrak j$, whereby $\mathfrak a$ becomes a symplectic Lie algbera. We set $\mathfrak l:=\mathfrak g/\mathfrak j^\bot$.
Since $\mathfrak j$ is isomorphic to $\mathfrak l^*$, we can describe $\mathfrak g$ with the help of two not necessary abelian extensions of Lie algebras
\begin{align}\label{eq:2aE}
0\rightarrow\mathfrak l^*\rightarrow\mathfrak g\rightarrow\mathfrak h\rightarrow 0,\quad 0\rightarrow\mathfrak a\rightarrow\mathfrak h\rightarrow\mathfrak l\rightarrow 0.
\end{align}
Here $\mathfrak h=\mathfrak g/\mathfrak i$.
Conversely, two extensions given as in (\ref{eq:2aE}) define a symplectic Lie algebra for every abelian Lie algebra $\mathfrak l$ and symplectic Lie algebra $\mathfrak a$, unless the extensions satisfy certain compatibility conditions.
Then the image of $\mathfrak l^*$ in $\mathfrak g$ is usually not equal to the canonical isotropic ideal $\mathfrak j$.

We will introduce this construction scheme under the notion of quadratic extensions.
In general, not every quadratic extension of $\mathfrak l$ by $(\mathfrak a,\omega_\mathfrak a)$ is given by the canonical isotropic ideal $\mathfrak j$.
Thus, we also introduce balanced quadratic extensions. We call a quadratic extension balanced, if the image of $\mathfrak l^*$ in $\mathfrak g$ equals the canonical isotropic ideal $\mathfrak z\cap\mathfrak z^\bot$.
Every symplectic Lie algebra with degenerate center has the structure of a balanced quadratic extension in a canonical way.
There is a natural equivalence relation on the set of quadratic extensions of $\mathfrak l$ by $\mathfrak a$.
Moreover, we define a non-linear cohomology set $H^2(\mathfrak{l},\mathfrak{a},\omega_\mathfrak a)$ of the abelian Lie algebra $\mathfrak l$ with coefficients in the symplectic Lie algebra $(\mathfrak a,\omega_\mathfrak a)$. Then we prove that the equivalence classes of quadratic extensions of $\mathfrak l$ by $(\mathfrak a,\omega_\mathfrak a)$ are in bijection to the second cohomology set $H^2(\mathfrak{l},\mathfrak{a},\omega_\mathfrak a)$.
Moreover, we give the standard model $\mathfrak{d}_{\gamma,\epsilon,\xi}(\mathfrak l,\mathfrak a)$ of a symplectic Lie algebra with, which defines a quadratic extension of $\mathfrak{l}$ by $\mathfrak{a}$ for every $[\gamma,\epsilon,\xi]\in H^2(\mathfrak{l},\mathfrak{a},\omega_\mathfrak a)$.

The equivalence classes of balanced quadratic extensions of $\mathfrak l$ by $(\mathfrak a,\omega_\mathfrak a)$ are described by $H^2(\mathfrak{l},\mathfrak{a},\omega_\mathfrak a)_\sharp\subset H^2_{Q+}(\mathfrak{l},D_\mathfrak{l},\mathfrak{a})$ on the level of the cohomology classes. This set can be characterised explicitly.
We also call these cocycles balanced.
The automorphism group $G(\mathfrak{l},\mathfrak{a},\omega_\mathfrak a)$ of the pair $(\mathfrak{l},\mathfrak{a},\omega_\mathfrak a)$ acts on $H^2(\mathfrak{l},\mathfrak{a},\omega_\mathfrak a)_\sharp$ and we obtain a bijection between the isomorphism classes of $2n$-dimensional symplectic Lie algebras with degenerate center and the set 
\[\coprod_{m=1}^n\quad\coprod_{[\mathfrak a,\omega_{\mathfrak a}]\in \operatorname{Isom}(2n-2m)} H(\mathds R^m,\mathfrak a,\omega_{\mathfrak a})_\sharp/G.\]
Here $\operatorname{Isom}(2n-2m)$ denotes a system of representatives of the isomorphism classes of symplectic Lie algebras of dimension $2n-2m$.

The next step is to apply this classification scheme on nilpotent symplectic Lie algebras. Whereas, for a symplectic Lie algebra with degenerate center its corresponding symplectic Lie algebra $\mathfrak j^\bot/\mathfrak j$ can have a non-degenerate center and thus the classification of all symplectic Lie algebras with degenerate center uses the knowledge of all symplectic Lie algebras of lower dimension (including those with non-degenerate center), this problem vanishes for nilpotent symplectic Lie algebras. Because a nilpotent symplectic Lie algebra has a non-degenerate center, if and only if it is abelian.

We determine when cocycles of the second cohomology  $H^2(\mathfrak{l},\mathfrak{a},\omega_\mathfrak{a})_\sharp$ define nilpotent symplectic Lie algebras.
We denote this set of cohomology classes by $H^2(\mathfrak{l},\mathfrak{a},\omega_\mathfrak{a})_0$ and obtain a bijection between the isomorphism classes of non-abelian nilpotent symplectic Lie algebras and
\[\coprod_{m=1}^n\quad\coprod_{[\mathfrak a,\omega_{\mathfrak a}]\in \mathfrak A(2n-2m)} H^2(\mathds R^m,\mathfrak a,\omega_{\mathfrak a})_0/G,\] 
where $\mathfrak A(2n-2m)$ denotes a system of representatives of the isomorphism classes of $(2n-2m)$-dimensional nilpotent symplectic Lie algebras.
The classification scheme for nilpotent symplectic Lie algebras provides several advantages. For one thing nilpotent symplectic Lie algebras can be determined systematically up to isomorphism, since the scheme does not depend on choices (as in \cite{MR91}, \cite{DM96}, \cite{DM962}, \cite{BC13}).
For another thing the classification of nilpotent symplectic Lie algebras of a fixed dimension does not use any other classifications, as for example a classification of all nilpotent Lie algebras of a fixed dimension (as in \cite{GKM04} for the classification of $6$-dimensional nilpotent symplectic Lie algebras).
We only need a system of representatives of the isomorphism classes of nilpotent symplectic Lie algebras of lower dimension, which also can be computed with this scheme.
Moreover, this classification scheme is universally usable. There are no additional restrictions to the Lie algebras.

We derive exemplarily a system of representatives of the isomorphism classes of $6$-dimensional nilpotent symplectic Lie algebras (see \cite{GKM04}, also compare to \cite{GB87}) on a new way using this classification scheme.
Therefor, we determine the $G$-orbits of $H^2(\mathfrak{l},\mathfrak{a},\omega_\mathfrak{a})_0$ for 
\[(\mathfrak{l},\mathfrak{a})\in\{(\mathds R,\mathds R^4),
(\mathds R,\mathfrak h_3\times\mathds R),
(\mathds R,\mathfrak n_4),
(\mathds R^3,\{0\}),
(\mathds R^2,\mathds R^2)
\}\]
and corresponding symplectic form $\omega_\mathfrak a$.
Besides, we remove some little inaccuracies in the list in \cite{GKM04}.
\\\quad\\
The paper is organized as follows.
We introduce the notion of (balanced) quadratic extensions in section \ref{SLA:quadrErw} and show that every symplectic Lie algebra with degenerate center has the structure of a (balanced) quadratic extension in a canonical way.
Furthermore, we give an equivalence relation on the set of quadratic extensions.
In section \ref{SLA:Std} we define the standard model $\mathfrak{d}_{\gamma,\epsilon,\xi}(\mathfrak l,\mathfrak a)$ and show necessary and sufficient conditions on the level of the corresponding cocycles (which we also define in section \ref{SLA:Std}), when $\mathfrak{d}_{\gamma,\epsilon,\xi}(\mathfrak l,\mathfrak a)$ has the structure of a (balanced) quadratic extension.
We show that every quadratic extension is equivalent to a suitable standard model (section \ref{SLA:ÄquivStd}) and describe the equivalence of standard models (section \ref{SLA:Äquivkl}).
In the end of section \ref{SLA:Äquivkl} we obtain a bijection between $H^2_{Q+}(\mathfrak{l},D_\mathfrak{l},\mathfrak{a})$ and the equivalence classes of quadratic extensions of $\mathfrak{l}$ by $\mathfrak{a}$. Here the cohomology set $H^2(\mathfrak l,\mathfrak a,\omega_\mathfrak a)$ and also $H^2(\mathfrak l,\mathfrak a,\omega_\mathfrak a)_\sharp$ are defined in section \ref{SLA:Äquivkl}.
In section \ref{SLA:Isomkl} we describe the isomorphy of standard models on the level of the corresponding cohomology classes and obtain, finally, the classification scheme for symplectic Lie algebras.
Afterwards, we calculate in section \ref{SLA:Anw} when Standard models $\mathfrak{d}_{\gamma,\epsilon,\xi}(\mathfrak l,\mathfrak a)$ are nilpotent symplectic Lie algebras and obtain a classification scheme for nilpotent symplectic Lie algebras.
Then we classify all nilpotent symplectic Lie algebras of dimension less than eight (section \ref{NSLA:1} - \ref{NSLA:5}) on a new way.

\section{Symplectic Lie algebras with degenerate center}\label{SLA}

At first, wie define symplectic Lie algebras.

\begin{definition}
A symplectic Lie algebra $(\mathfrak{g},\omega)$ is a real Lie algebra with a skewsymmetric non-degenerate bilinear form $\omega:\mathfrak{g}\times\mathfrak{g}\rightarrow\mathds{R}$,
which is closed, i. e. 
\[\rd\omega(X,Y,Z)=-\omega([X,Y],Z)+\omega([X,Z],Y)-\omega([Y,Z],X)=0\]
for every $X,Y,Z\in\mathfrak{g}$.
We call $\omega$ a symplectic form of $\mathfrak g$.

An isomorphism between two symplectic Lie algebras $(\mathfrak{g}_1, \omega_1)$ and $(\mathfrak{g}_2,\omega_2)$ is an isomorphism $\varphi:\mathfrak{g}_1\rightarrow\mathfrak{g}_2$ of Lie algebras, which satisfies 
$\varphi^*\omega_2=\omega_1$.
\end{definition}

Every symplectic Lie algebra with non-degenerate center $\mathfrak z$ equals the sum $(\mathfrak g,\omega)=(\mathfrak z,\omega|_{\mathfrak z\times \mathfrak z})\oplus(\mathfrak z^\bot,\omega|_{\mathfrak z^\bot\times \mathfrak z^\bot})$ of an abelian symplectic Lie algebra and and symplectic Lie algebra with trivial center. Here $\mathfrak z^\bot$ denotes the set of all vectors in $\mathfrak g$ being orthogonal to the center $\mathfrak z$ with respect to $\omega$, which is an ideal of $\mathfrak g$ in this situation, because of $\rd \omega=0$.
Hence, a classification of symplectic Lie algebras with non-degenerate center reduces to the classification of symplectic Lie algebras with trivial center.
In this paper we concentrate on symplectic Lie algebras with degenerate center,
which, for instance, includes non-abelian nilpotent symplectic Lie algebras.

\subsection{Quadratic extensions of symplectic Lie algebras}\label{SLA:quadrErw}

We introduce the notion of a quadratic extension of symplectic Lie algebras and show that every symplectic Lie algebra with degenerate center naturally has such a structure. Moreover, we define a equivalence relation on the set of quadratic extensions.
\\\quad\\
For every central isotropic ideal $\mathfrak j\subset \mathfrak z$ and an symplectic Lie algebra $(\mathfrak g,\omega_\mathfrak g)$ the set $\mathfrak j^\bot$ is also an ideal of $\mathfrak g$, since $\rd\omega_\mathfrak g=0$.
Moreover, $\omega_\mathfrak g$ induces a symplectic form $\overline\omega$ of the Lie algebra $\mathfrak j^\bot/\mathfrak j$.
This motivates the following definition.

\begin{definition}
A quadratic extension of an abelian Lie algebra $\mathfrak{l}$ by a symplectic one $(\mathfrak a,[\cdot,\cdot]_\mathfrak{a},\omega_\mathfrak{a})$ is an tuple $(\mathfrak{g},\mathfrak{j},i,p)$, where
\begin{itemize}
\item $\mathfrak g= (\mathfrak{g},[\cdot,\cdot]_\mathfrak{g},\omega_\mathfrak{g})$ is a symplectic Lie algebra,
\item $\mathfrak j\subset \mathfrak z$ is an ideal of $\omega_{\mathfrak g}$, which is isotropic with respect to $\omega_\mathfrak g$ and
\item $i$ and $p$ are homomorphisms of Lie algebras forming a short exact sequence of Lie algebras 
\[0\rightarrow \mathfrak a\xrightarrow{i}\mathfrak g/\mathfrak j\xrightarrow{p}\mathfrak l\rightarrow 0,\]
where $\im i=\mathfrak j^\bot/\mathfrak j$ and 
$i:\mathfrak{a}\rightarrow \mathfrak j^\bot/\mathfrak j$ is an isomorphism of symplectic Lie algebras.
\end{itemize}
We call a quadratic extension balanced, if $\mathfrak j=\mathfrak z\cap\mathfrak z^\bot$.
\end{definition}

\begin{theorem}\label{Th:1}
Every symplectic Lie algebra $(\mathfrak g, [\cdot,\cdot]_{\mathfrak g}, \omega_{\mathfrak g})$ has the structure of a balanced quadratic extension in a canonical way.
\end{theorem}

\begin{proof}
Let $\mathfrak j$ denote the isotropic ideal $\mathfrak j:=\mathfrak z\cap\mathfrak z^\bot$.
Then
\[0\rightarrow \mathfrak j\xrightarrow{\iota}\mathfrak g\xrightarrow{\pi}\mathfrak g/\mathfrak j\rightarrow 0\] is a short exact sequence, where $\iota$ and $\pi$ denote the canonical embedding and projection.
Since $\mathfrak j^\bot$ is an ideal of $\mathfrak g$, also $\mathfrak j^\bot/\mathfrak j$ is one in $\mathfrak g/\mathfrak j$ and
\[0\rightarrow \mathfrak j^\bot/\mathfrak j\xrightarrow{i}\mathfrak g/\mathfrak j\xrightarrow{p}\mathfrak g/\mathfrak j^\bot\rightarrow 0\]
is a short exact sequence, where $i$ and $p$ are the canonical embedding and projection.
Now, we set
\[\mathfrak a:=\mathfrak j^\bot/\mathfrak j, \quad \mathfrak l:=\mathfrak g/\mathfrak j^\bot.\]
It is clear that $\mathfrak a$ is a symplectic Lie algebra.
Moreover, from \[\omega_g([\mathfrak g,\mathfrak g],\mathfrak j)\subset\omega_g([\mathfrak g,\mathfrak j],\mathfrak g)=0\] we obtain that $\mathfrak l$ is abelian.
\end{proof}

\begin{remark}
Symplectic Lie algebras $\mathfrak{g}$ with non-degenerate center have the structure of a trivial balanced quadratic extension of $\{0\}$ by $\mathfrak{g}$. In particular, every symplectic Lie algebra with degenerate center obtains the structure of a nontrivial balanced quadratic extension.
Conversely, the corresponding symplectic Lie algebra $\mathfrak g$ of a nontrivial balanced quadratic extension $(\mathfrak g,\mathfrak j,i,p)$ has a degenerate center.
\end{remark}

\begin{definition}
Two quadratic extensions $(\mathfrak g_k,\mathfrak j_k,i_k,p_k)$, $k=1,2$, of $\mathfrak l$ by $\mathfrak a$ are equivalent, if there is an isomorphism $\Phi:\mathfrak g_1\rightarrow\mathfrak g_2$ of the corresponding symplectic Lie algebras, which maps $\mathfrak j_1$ to $\mathfrak j_2$ and satisfies
\[\overline\Phi\circ i_1=i_2,\quad p_2\circ\overline\Phi=p_1.\]

Here $\overline\Phi:\mathfrak g_1/\mathfrak j_1\rightarrow \mathfrak g_2/\mathfrak j_2$ denotes the map on the quotient induced by $\Phi$.
\end{definition}

\subsection{Factor system}\label{SLA:Faktorsys}

In this section we give the definition of a factor system.
These objects describe extensions of non-abelian Lie algebras \cite{Ne06} (see for instance \cite{MM53}, \cite{Ho54}, \cite{Sh66}),
which we will use for symplectic Lie algebras with degenerate center.

\begin{definition}
Let $\mathfrak l$ be a Lie algebra and $U,V,W$ $\mathfrak l$-modules.
Furthermore, let $\operatorname{m}:U\times V\rightarrow W$ be a $\mathfrak l$-equivariant bilinear mapping, i. e. $L\cdot \operatorname{m}(u,v)=\operatorname{m}(L\cdot u,v)+\operatorname{m}(u,L\cdot v)$ for all $L\in\mathfrak l$, $u\in U$ and $v\in V$.
There is a naturlal product \[\operatorname{m}(\cdot\wedge\cdot):C^p(\mathfrak l,U)\times C^q(\mathfrak l,V)\rightarrow C^{p+q}(\mathfrak l,W),\quad (\tau_1,\tau_2)\mapsto \operatorname{m}(\tau_1\wedge\tau_2),\]
which is given by 
\begin{align*}
\operatorname{m}(&\alpha\wedge\tau)(L_1,\dots,L_{p+q})\\
&=\sum_{[\sigma]\in\mathcal{S}_{p+q}/\mathcal{S}_{p}\times\mathcal{S}_{q}}sgn(\sigma)\operatorname{m}(\alpha(L_{\sigma(1)},\dots,L_{\sigma(p)}),\tau(L_{\sigma(p+1)},\dots,L_{\sigma(p+q)})). 
\end{align*}
Here $\mathcal S_k$ denotes the symmetric group of $k$ letters.
\end{definition}

\begin{example}
Let $\mathfrak l$ be a Lie algebra, $(\mathfrak a,[\cdot,\cdot]_{\mathfrak a},\omega_{\mathfrak a})$ a symplectic one and a trivial $\mathfrak l$-module. Furthermore, let $\mathfrak h$ denote a trivial $\mathfrak l$-module and we think of $C^1(\mathfrak a,\mathfrak h)$ as a trivial $\mathfrak l$-module.
We consider the bilinear map
\[\ev:C^1(\mathfrak a,\mathfrak h)\times\mathfrak a\rightarrow \mathfrak h,\quad  \ev(f,A):=f(A).\]
Then we have the multiplications
\begin{itemize}
\item $\ev(\cdot\wedge\cdot):C^p(\mathfrak l,C^1(\mathfrak a,\mathfrak h))\times C^q(\mathfrak l,\mathfrak a)\rightarrow C^{p+q}(\mathfrak l,\mathfrak h)$,
\item $\omega_{\mathfrak a}(\cdot\wedge\cdot):C^p(\mathfrak l,\mathfrak a)\times C^q(\mathfrak l,\mathfrak a)\rightarrow C^{p+q}(\mathfrak l,\mathds R)$ and 
\item $[\cdot\wedge\cdot]_{\mathfrak a}:C^p(\mathfrak l,\mathfrak a)\times C^q(\mathfrak l,\mathfrak a)\rightarrow C^{p+q}(\mathfrak l,\mathfrak a)$.
\end{itemize}
\end{example}

\begin{definition}
Let $\xi:\mathfrak l\rightarrow \Hom(\mathfrak a)$ be linear, $\mathfrak a$ a trivial $\mathfrak l$-module and denote $(C^*(\mathfrak{l},\mathfrak{a}),\rd)$ the corresponding cochain complex.
We consider
$\xi_\wedge:C^p(\mathfrak l,\mathfrak a)\rightarrow C^{p+1}(\mathfrak l,\mathfrak a)$, defined by $\tau\rightarrow\ev(\xi\wedge\tau)$.
We define the covariant differential $\rd_\xi:C^p(\mathfrak l,\mathfrak a)\rightarrow C^{p+1}(\mathfrak l,\mathfrak a)$ by
\[\rd_\xi\tau:=\xi_\wedge+\rd:C^p(\mathfrak l,\mathfrak a)\rightarrow C^{p+1}(\mathfrak l,\mathfrak a).\]
\end{definition}

For the covariant differential we have
\begin{align*}
(\rd_\xi\tau)(L_1,\dots ,L_{p+1}) =&\sum_{i=1}^{p+1}(-1)^{i+1}\xi(L_i)\tau(L_1,\dots,\hat{L_i},\dots,L_{p+1})\\
&+\sum_{i<j}(-1)^{i+j}\tau([L_i,L_j],L_1,\dots,\hat{L_i},\dots,\hat{L_j},\dots,L_{p+1})
\end{align*}
(see \cite{Ne06}).

\begin{remark}
The map $\rd_\xi$ is not a differential in the typical sense. 
This was already shown in \cite{Ne06}. 
Therein the map satisfies $\rd_\xi\circ\rd_\xi=0$, if and only if $\xi$ is an homomorphism of Lie algebras and hence a representation.
\end{remark}

Here we mention that every linear map $\xi:\mathfrak{l}\rightarrow\Hom(\mathfrak{a})$ defines a linear map $\hat\xi:\mathfrak l\rightarrow \mathfrak{gl}(C^1(\mathfrak a,\mathfrak l^*))$ by \[(\hat\xi(L))f=f\circ\xi(L).\]

\begin{definition}
In the following, let us assume that $\mathfrak a$ is a Lie algebra.
We call $(\xi,\alpha)\in C^1(\mathfrak l,\mathfrak{gl}(\mathfrak a))\oplus C^2(\mathfrak l,\mathfrak a)$ a factor system, if $\im\xi\subset\mathfrak{der}(\mathfrak a)$, $\rd\xi+\frac{1}{2}[\xi\wedge\xi]=\ad\circ \alpha$ and $\rd_\xi\alpha=0$.
Here $\mathfrak{der}(\mathfrak a)$ denotes the space of derivations of $\mathfrak a$.
\end{definition} 

\subsection{The standard model}\label{SLA:Std}

In this section we define the standard model of a quadratic extension of an abelian Lie algebra $\mathfrak l$ by a symplectic one $(\mathfrak a,\omega_\mathfrak a)$.
We give necessary and sufficient conditions for the existence of the standard model on the level of cocycles and decide when cocycles describe balanced quadratic extensions.
\\\quad\\
For a linear map $f:\mathfrak a\rightarrow\mathfrak l^*$ let $f^*:\mathfrak l\rightarrow\mathfrak a$ denote its dual map given by $(f(A))(L)=\omega_{\mathfrak a}(f^*(L),A)$.

\begin{definition}
Let $(\mathfrak a,[\cdot,\cdot]_{\mathfrak a},\omega_{\mathfrak a})$ be a symplectic Lie algebra and $\mathfrak l$ an abelian one. Moreover, consider the dual space $\mathfrak l^*$ of $\mathfrak l$ as an abelian Lie algebra.
Assume a linear map $\gamma:\mathfrak l\rightarrow\Hom(\mathfrak a,\mathfrak l^*)$, a $2$-form $\epsilon$ on $\mathfrak l$ with values in $\mathfrak l^*$ and a linear $\xi:\mathfrak l\rightarrow \mathfrak{gl}(\mathfrak a)$.
In the following, let $\beta$ denote the $2$-Form on $\mathfrak a$ with coefficients in $\mathfrak l^*$, which is defined by 
\begin{align}\label{def:beta}
\beta(A_1,A_2)=-\omega_{\mathfrak a}(\xi(\cdot)A_1,A_2)-\omega_{\mathfrak a}(A_1,\xi(\cdot)A_2)
\end{align}
for all $A_1,A_2\in\mathfrak a$.
Moreover, we set
\begin{align}
\beta_0(A)=\beta(A,\cdot)
\end{align}
for every $A\in\mathfrak a$. This map $\beta_0$ equals $\beta$ considered as an linear map from $\mathfrak a$ to $\Hom(\mathfrak a,\mathfrak l^*)$.
By $\alpha\in C^2(\mathfrak l,\mathfrak a)$ we denote, in the following, the map
\begin{align}\label{def:alpha}
\alpha(L_1,L_2)=(\gamma(L_1))^*L_2-(\gamma(L_2))^*L_1.
\end{align}
Now, we consider the vector space $\mathfrak l^*\oplus\mathfrak a\oplus\mathfrak l$ and define a skewsymmetric bilinear map
$[\cdot,\cdot]:=[\cdot,\cdot]_{\gamma,\epsilon,\xi}:(\mathfrak l^*\oplus\mathfrak a\oplus\mathfrak l)^2\rightarrow \mathfrak l^*\oplus\mathfrak a\oplus\mathfrak l$ by
\begin{align*}
[\mathfrak l^*,\mathfrak l^*\oplus \mathfrak a\oplus \mathfrak l]:&=0,\\
[A_1,A_2]:&=\beta(A_1,A_2)+[A_1,A_2]_{\mathfrak a}\in\mathfrak l^*\oplus\mathfrak a,\\
[L,A]:&=\gamma(L)A+\xi(L)(A)\in\mathfrak l^*\oplus\mathfrak a,\\
[L_1,L_2]:&=\epsilon(L_1,L_2)+\alpha(L_1,L_2)\in\mathfrak l^*\oplus\mathfrak a,
\end{align*}
for all $A,A_1,A_2\in\mathfrak a$ and $L,L_1,L_2\in\mathfrak l$.
Moreover, we define a skewsymmetric bilinear map $\omega:(\mathfrak l^*\oplus\mathfrak a\oplus\mathfrak l)^2\rightarrow\mathds R$, which is given by
\begin{align*}
\omega(Z_1+A_1+L_1,Z_2+A_2+L_2)=Z_1(L_2)+\omega_\mathfrak a(A_1,A_2)-Z_2(L_1)
\end{align*}
for every $Z_1,Z_2\in\mathfrak l^*$, $A_1,A_2\in\mathfrak a$ and $L_1,L_2\in\mathfrak l$.

For further considerations we give some more notation. We define $\xi_0:\mathfrak a\rightarrow \Hom(\mathfrak l,\mathfrak a)$ and $\gamma_0:\mathfrak a\rightarrow \Hom(\mathfrak l,\mathfrak l^*)$ by 
\begin{align}
\xi_0(A)L=\xi(L)A,\quad\gamma_0(A)L=\gamma(L)A
\end{align}
for all $A\in\mathfrak a$ and $L\in\mathfrak l$.
We omit brackets, if it is clear from the context, for instance $\xi(L)A$ instead of $(\xi(L))(A)$.
Moreover, we write $\gamma(L_1,A,L_2)$ for $((\gamma(L_1))(A))(L_2)$ as a simplification.
By $\mathfrak a^\mathfrak{l}_\xi$ we mean the space $\mathfrak a^\mathfrak l_\xi:=\ker \xi_0$.
If $\xi$ is a representation, then $\mathfrak a^\mathfrak l_\xi$ equals the space of invariants $\mathfrak a^\mathfrak l$.
\end{definition}

\begin{definition}
Let $\mathfrak l$ be an abelian Lie algebra and $(\mathfrak a,[\cdot,\cdot]_{\mathfrak a},\omega_{\mathfrak a})$ a symplectic Lie algebra, considered as a trivial $\mathfrak l$-module.
Let $Z^2(\mathfrak l,\mathfrak a,\omega_{\mathfrak a})$ denote the set of triples $(\gamma,\epsilon,\xi)$ of linear maps $\gamma:\mathfrak l\rightarrow\Hom(\mathfrak a,\mathfrak l^*)$, $\epsilon\in C^2(\mathfrak l,\mathfrak l^*)$ and $\xi:\mathfrak l\rightarrow\Hom(\mathfrak a)$, which satisfy the following conditions for every $L,L_1,L_2,L_3\in\mathfrak l$ and $A_1,A_2\in\mathfrak a$:
\begin{align}
&(\xi,\alpha)\text{ is a factor system},\label{eq:lemma1}\\
&\beta(\xi(L)A_1,A_2)+\beta(A_1,\xi(L)A_2)=(\gamma(L))[A_1,A_2]_{\mathfrak a},\label{eq:lemma2}\\
&\rd_{\hat\xi}\gamma+\beta_0\circ\alpha=0,\label{eq:lemma3}\\
&\ev(\gamma\wedge\alpha)=0,\label{eq:lemma4}\\
&\epsilon(L_1,L_2)L_3+\epsilon(L_2,L_3)L_1+\epsilon(L_3,L_1)L_2=0.\label{eq:lemma5}
\end{align}
Here $\alpha$ and $\beta$ are defined as in equation (\ref{def:alpha}) and (\ref{def:beta}).
\end{definition}

We call the elements of $Z^2(\mathfrak{l},\mathfrak{a},\omega_\mathfrak{a})$ quadratic cocycles.

\begin{theorem} \label{Thm:Z2}
The triple $\mathfrak d_{\gamma,\epsilon,\xi}(\mathfrak l,\mathfrak a):=(\mathfrak l^*\oplus\mathfrak a\oplus\mathfrak l,[\cdot,\cdot],\omega)$ is a symplectic Lie algebra if and only if $(\gamma,\epsilon,\xi)\in Z^2(\mathfrak l,\mathfrak a,\omega_{\mathfrak a})$.
\end{theorem}

\begin{proof}
The vector space $\mathfrak l^*\oplus\mathfrak a\oplus\mathfrak l$ becomes a Lie algebra, if and only if $[\cdot,\cdot]$ satisfies the Jacobi identity.
We define
\[\operatorname{Jac}(X,Y,Z):=[[X,Y],Z]+[[Y,Z],X]+[[Z,X],Y]\]
for $X,Y,Z\in\mathfrak l^*\oplus\mathfrak a\oplus\mathfrak l$.
The map $[\cdot,\cdot]$, defined above, satisfies obviously 
$\operatorname{Jac}(Z,X_1,X_2)=0$ for every $Z\in\mathfrak l^*, X_1,X_2\in\mathfrak l^*\oplus\mathfrak a\oplus\mathfrak l$.
Now, we show that the other Jacobi identities are equivalent to the conditions  (\ref{eq:lemma1})-(\ref{eq:lemma4}).
Since 
\begin{align*}
\operatorname{Jac}(A_1,A_2,L)=&\,[[A_1,A_2]_{\mathfrak a},L]-[\xi(L)A_2,A_1]+[\xi(L)A_1,A_2]\\
=&\, -\xi(L)[A_1,A_2]_{\mathfrak a}-(\gamma(L))[A_1,A_2]_{\mathfrak a}-[\xi(L)A_2,A_1]_{\mathfrak a}-\beta(\xi(L)A_2,A_1)\\
&+[\xi(L)A_1,A_2]_{\mathfrak a}+\beta(\xi(L)A_1,A_2)
\end{align*}
for all $A_1,A_2\in\mathfrak a$, $L\in\mathfrak l$, the equation $\operatorname{Jac}(A_1,A_2,L)=0$ is equivalent to condition (\ref{eq:lemma2}) and $\xi(L)\in\mathfrak{der}(\mathfrak a)$.
Moreover, $\operatorname{Jac}(A_1,A_2,A_3)=0$ holds for all $A_1,A_2,A_3\in\mathfrak a$, if $\xi(L)$ is a derivation on $\mathfrak a$ for all $L\in\mathfrak l$, because of
\begin{align*}
\operatorname{Jac}(A_1,A_2,A_3)&=[[A_1,A_2]_{\mathfrak a},A_3]+[[A_2,A_3]_{\mathfrak a},A_1]+[[A_2,A_3]_{\mathfrak a},A_1]\\
&=\beta([A_1,A_2]_{\mathfrak a},A_3)+\beta([A_2,A_3]_{\mathfrak a},A_1)+\beta([A_3,A_1]_{\mathfrak a},A_2)\\
&=\sum_{\operatorname{cycl}(A_1,A_2,A_3)}\big(\omega_{\mathfrak a}(\xi(\cdot)[A_1,A_2]_{\mathfrak a},A_3)+\omega_{\mathfrak a}([A_1,A_2]_{\mathfrak a},\xi(\cdot)A_3)\big)
\end{align*}
and
\begin{align*}
0&=-\rd\omega_{\mathfrak a}(\xi(\cdot)A_1,A_2,A_3)-\rd\omega_{\mathfrak a}(A_1,\xi(\cdot)A_2,A_3)-\rd\omega_{\mathfrak a}(A_1,A_2,\xi(\cdot)A_3)\\
&=\sum_{\operatorname{cycl}(A_1,A_2,A_3)}\big(\omega_{\mathfrak a}([\xi(\cdot)A_1,A_2]_{\mathfrak a},A_3)+\omega_{\mathfrak a}([A_1,\xi(\cdot)A_2]_{\mathfrak a},A_3)+\omega_{\mathfrak a}([A_1,A_2]_{\mathfrak a},\xi(\cdot)A_3)\big).
\end{align*}
Using
\begin{align*}
[[A,&L_1],L_2]+[[L_1,L_2],A]+[[L_2,A],L_1]=-[\xi(L_1)A,L_2]+[\alpha(L_1,L_2),A]+[\xi(L_2)A,L_1]\\
=&\xi(L_2)\xi(L_1)A+(\gamma(L_2))(\xi(L_1)A)+[\alpha(L_1,L_2),A]_{\mathfrak a}\\
&+\beta(\alpha(L_1,L_2),A)-\xi(L_1)\xi(L_2)A-(\gamma(L_1))(\xi(L_2)A)
\end{align*}
we obtain that $\operatorname{Jac}(A,L_1,L_2)=0$ holds for all $A\in\mathfrak a$, $L_1,L_2\in\mathfrak l$ if and only if $\frac{1}{2}[\xi\wedge\xi]=\ad\circ\alpha$ and condition (\ref{eq:lemma3}) is satisfied.
Finally, $\operatorname{Jac}(L_1,L_2,L_3)=0$ for every $L_1,L_2,L_3\in \mathfrak l$ is equivalent to $\rd_\xi\alpha=0$ and condition (\ref{eq:lemma4}) because of 
\begin{align*}
\operatorname{Jac}(L_1,L_2,L_3)&=\sum_{\operatorname{cycl}(L_1,L_2,L_3)}[\alpha(L_1,L_2),L_3]\\
&=\sum_{\operatorname{cycl}(L_1,L_2,L_3)}\big(-\gamma(L_3)\alpha(L_1,L_2)-\big(\xi(L_3)\big)\big(\alpha(L_1,L_2)\big)\big).
\end{align*}
Altogether, the triple $(\mathfrak l^*\oplus\mathfrak a\oplus\mathfrak l,[\cdot,\cdot])$ is a Lie algebra if and only if the equations  (\ref{eq:lemma1}) - (\ref{eq:lemma4}) hold.

Now, the Lie-Algebra $(\mathfrak l^*\oplus\mathfrak a\oplus\mathfrak l,[\cdot,\cdot])$ with the skewsymmetric bilinear form $\omega$ is symplectic if and only if $\omega$ is closed.
Here $\omega$, as defined above, satisfies obviously
\[\rd\omega(\mathfrak l^*,\mathfrak l^*\oplus\mathfrak a\oplus\mathfrak l,\mathfrak l^*\oplus\mathfrak a\oplus\mathfrak l)=\omega([\mathfrak l^*\oplus\mathfrak a\oplus\mathfrak l,\mathfrak l^*\oplus\mathfrak a\oplus\mathfrak l],\mathfrak l^*)\subset\omega(\mathfrak l^*\oplus\mathfrak a,\mathfrak l^*)=0\]
and $\rd\omega(\mathfrak a,\mathfrak a,\mathfrak a)=\rd\omega_a(\mathfrak a,\mathfrak a,\mathfrak a)=0$. 
Moreover, $\rd\omega(\mathfrak a,\mathfrak a,\mathfrak l)=0$ and $\rd\omega(\mathfrak a,\mathfrak l,\mathfrak l)=0$ hold, since
\begin{align*}
-&\rd \omega(A_1,A_2,L)=\omega([A_1,A_2],L)+\omega([L,A_1],A_2)+\omega([A_2,L],A_1)\\
&=\omega([A_1,A_2]_{\mathfrak a}+\beta(A_1,A_2),L)+\omega(\xi(L)A_1+(\gamma(L))A_1,A_2)-\omega(\xi(L)A_2+(\gamma(L))A_2,A_1)\\
&=\beta(A_1,A_2)L+\omega_{\mathfrak a}(\xi(L)A_1,A_2)+\omega_{\mathfrak a}(A_1,\xi(L)A_2)=0
\end{align*}
and 
\begin{align*}
-&\rd \omega(A,L_1,L_2)=\omega([A,L_1],L_2)+\omega([L_1,L_2],A)+\omega([L_2,A],L_1)\\
&=-\omega(\xi(L_1)A+(\gamma(L_1))A,L_2)+\omega(\alpha(L_1,L_2)+\epsilon(L_1,L_2),A)+\omega(\xi(L_2)A+(\gamma(L_2))A,L_1)\\
&=-\gamma(L_1,A,L_2)+\gamma(L_2,A,L_1)+\omega_{\mathfrak a}(\alpha(L_1,L_2),A)=0
\end{align*}
for all $A,A_1,A_2\in\mathfrak a$ und $L,L_1,L_2\in\mathfrak l$.
Thus, it remains to show that $\rd \omega(L_1,L_2,L_3)=0$ for all $L_1,L_2,L_3\in\mathfrak l$ is equivalent to $\sum_{\operatorname{cycl}(L_1,L_2,L_3)}\epsilon(L_1,L_2)L_3=0$ and this follows directly from
\[-\rd\omega(L_1,L_2,L_3)=\sum_{\operatorname{cycl}(L_1,L_2,L_3)}\big(\omega(\alpha(L_1,L_2)+\epsilon(L_1,L_2),L_3)\big)=\sum_{\operatorname{cycl}(L_1,L_2,L_3)}\epsilon(L_1,L_2)L_3.\]
\end{proof}

We identify $\mathfrak d_{\gamma,\epsilon,\xi}(\mathfrak l,\mathfrak a)/\mathfrak{l}^*$ with $\mathfrak a\oplus\mathfrak l$. Moreover, let $i:\mathfrak
a\rightarrow\mathfrak a\oplus\mathfrak l$ and $p:\mathfrak
a\oplus\mathfrak l\rightarrow\mathfrak l$ denote the canonical embedding and projection.

Using Theorem \ref{Thm:Z2} we get directly that $(\mathfrak d_{\gamma,\epsilon,\xi}(\mathfrak l,\mathfrak a),\mathfrak l^*,i,p)$ is a quadratic extension of $\mathfrak l$ by $(\mathfrak a,\omega_\mathfrak a)$ if and only if $(\gamma,\epsilon,\xi)$ is an element of $Z^2(\mathfrak l,\mathfrak a,\omega_{\mathfrak a})$.
In the following, we simply write $\mathfrak d_{\gamma,\epsilon,\xi}(\mathfrak l,\mathfrak a)$ for the quadratic extension $(\mathfrak d_{\gamma,\epsilon,\xi}(\mathfrak l,\mathfrak a),\mathfrak l^*,i,p)$.

\begin{definition}\label{SLA:def:balanced}
We call a triple $(\gamma,\epsilon,\xi)$ in $Z^2(\mathfrak l,\mathfrak a,\omega_{\mathfrak a})$ balanced, if it satisfies the following two conditions:
\begin{itemize}\label{eq:aundb}
\item[(a)] 
The set $\mathds L$ of all elements $L\in\mathfrak l$, which satisfy all the equations
\[\ad_{\mathfrak a}(A)+\xi(L)=0\quad (\text{a}.1), \quad\gamma(L)+\beta_0(A)=0\quad (\text{a}.2),\] \[\alpha(L,\cdot)=\xi_0(A)\quad (\text{a}.3)\quad \text{ and }\quad \epsilon(L,\cdot)=\gamma_0(A)\quad (\text{a}.4)\]
for a suitable $A\in\mathfrak a$ at ones, contains only the zero vector.
\item[(b)] The subspace $\mathfrak z(\mathfrak a)\cap \ker \beta_0\cap \mathfrak a^\mathfrak l_\xi\cap\ker\gamma_0$ of $\mathfrak a$ is non-degenerate with respect to $\omega_\mathfrak a$.
\end{itemize}
Let $Z^2(\mathfrak l,\mathfrak a,\omega_{\mathfrak a})_\sharp$ denote the set of all balanced triple of $Z^2(\mathfrak l,\mathfrak a,\omega_{\mathfrak a})$.
\end{definition}

It follows that balanced triple define balanced quadratic extensions.
\begin{theorem}\label{Thm:xa}
The quadratic extension $\mathfrak d_{\gamma,\epsilon,\xi}(\mathfrak l,\mathfrak a)$ is balanced if and only if $(\gamma,\epsilon,\xi)$ is an element in $Z^2(\mathfrak l,\mathfrak a,\omega_{\mathfrak a})_\sharp$.
\end{theorem}

\begin{proof}
It remains to show that $\mathfrak l^*=\mathfrak z\cap\mathfrak z^\bot$ holds if and only if the conditions in Definition \ref{SLA:def:balanced} are satisfied.
But before, we show that the first condition is equivalent to $\mathfrak z\subset\mathfrak l^*\oplus\mathfrak a$:
The center $\mathfrak z$ is a subset of $\mathfrak l^*\oplus\mathfrak a$, if and only if the equations $[A+L,A']=0$ and $[A+L,L']=0$ imply $L=0$ for all $A+L,A'+L'\in\mathfrak a\oplus\mathfrak l$. 
This implication is equivalent to the first condition.

The second condition holds if and only if $\mathfrak z\cap\mathfrak a$ is non-degenerate with respect to $\omega_\mathfrak a$. In addition, it is not hard to see that $\mathfrak z\subset \mathfrak l^*\oplus\mathfrak a$ implies the equality \[(\mathfrak z\cap\mathfrak a)\cap(\mathfrak z\cap\mathfrak a)^\bot=\mathfrak z\cap\mathfrak a\cap\mathfrak z^\bot.\]

Now, we prove the theorem.
Assume $\mathfrak z\cap\mathfrak z^\bot=\mathfrak l^*$.
Then $\mathfrak l^*$ is a subset of $\mathfrak z^\bot$ and thus $\mathfrak z$ a subset of $\mathfrak l^*\oplus\mathfrak a$.
Hence, the first condition holds.
In addition, $\mathfrak z\cap\mathfrak z^\bot\cap\mathfrak a=\{0\}$ holds and therby the second condition.
Conversely, suppose that condition (a) and (b) hold. Then $\mathfrak z\subset\mathfrak l^*\oplus\mathfrak a$ and hence $\mathfrak l^*\subset\mathfrak z\cap(\mathfrak l^*\oplus\mathfrak a)^\bot\subset\mathfrak z\cap\mathfrak z^\bot$.
Moreover, $\mathfrak z\cap\mathfrak z^\bot\cap\mathfrak a=(\mathfrak z\cap\mathfrak a)\cap(\mathfrak z\cap\mathfrak a)^\bot=\{0\}$ and we obtain $\mathfrak z\cap\mathfrak z^\bot\subset\mathfrak l^*$. This yields the assumption.
\end{proof}

\subsection{Equivalence to the standard model}\label{SLA:ÄquivStd}

The term standard model for $\mathfrak{d}_{\gamma,\epsilon,\xi}$ is motivated by the fact that every quadratic extension of symplectic Lie algebras is equivalent to $\mathfrak{d}_{\gamma,\epsilon,\xi}$ for some suitable $(\gamma,\epsilon,\xi)\in Z^2(\mathfrak{l},\mathfrak{a},\omega_\mathfrak a)$.
Showing this is the task of this section.

\begin{theorem}\label{Th:2}
For every quadratic extension $(\mathfrak g,\mathfrak j,i,p)$ of $\mathfrak l$ by $\mathfrak a$ there is a triple $(\gamma,\epsilon,\xi)\in Z^2(\mathfrak l,\mathfrak a,\omega_{\mathfrak a})$ such that $(\mathfrak g,\mathfrak j,i,p)$ is equivalent to $\mathfrak d_{\gamma,\epsilon,\xi}(\mathfrak l,\mathfrak a)$.
Moreover, for balanced quadratic extensions $(\mathfrak g,\mathfrak j,i,p)$ the corresponding cocycle $(\gamma,\epsilon,\xi)$ is balanced.
\end{theorem}

\begin{proof}
Let $V_{\mathfrak l}$ be an isotropic complement of $\mathfrak j^\bot$ and $V_{\mathfrak a}$ an orthogonal one of $\mathfrak j\oplus V_{\mathfrak l}$.
Then $\mathfrak j^\bot=\mathfrak j\oplus V_\mathfrak a$.
We choose an isomorphism $s:\mathfrak l\rightarrow V_{\mathfrak l}$ of vector spaces with $\tilde p\circ s=\id$,
where $\tilde p:\mathfrak g\rightarrow \mathfrak l$ denotes the composition of the natural projection $\pi:\mathfrak g\rightarrow\mathfrak g/\mathfrak j$ with $p$.
Moreover, there is a linear map $t:\mathfrak a\rightarrow V_{\mathfrak a}$, defined by \[i(A)=\pi(t(A))=t(A)+\mathfrak j\in\mathfrak j^\bot/\mathfrak j.\]

Since $\tilde p(X)=0$ for all $X\in\mathfrak j^\bot$, we define
the dual map $p^*:\mathfrak l^*\rightarrow \mathfrak g$ for $\tilde p:\mathfrak g\rightarrow \mathfrak l$ by
\[\omega_\mathfrak{g}(p^* Z,s(L))=Z((\tilde p\circ s)(L))=Z(L)\]
for all $Z\in\mathfrak l^*$, $L\in\mathfrak l$.
This map is an isomorphism from $\mathfrak{l}^*$ to $\mathfrak j$.

We define $\xi$, $\gamma$ and $\epsilon$ by
\begin{align}
i(\xi(L)A)&:=[s(L),t(A)]_{\mathfrak g}+\mathfrak j,\label{def:xi}\\
p^{*}((\gamma(L))A)&:=[s(L),t(A)]_{\mathfrak g}-t(\xi(L)A),\label{def:gamma}\\
\epsilon(L_1,L_2)&:=\omega_\mathfrak g([s(L_1),s(L_2)]_{\mathfrak g},s(\cdot)).\label{def:epsilon}
\end{align}
It is easy to see that this maps are well-defined.

The next step is to give some useful facts about $\beta$ and $\alpha$: 
Therefor, we consider 
\begin{align*}
p^{*}&\big(\beta(A_1,A_2)\big)=-p^{*}\big(\omega_{\mathfrak a}(\xi(\cdot)A_1,A_2)\big)-p^{*}\big(\omega_{\mathfrak a}(A_1,\xi(\cdot)A_2)\big)\\
=&-p^{*}\big(\omega_{\mathfrak g}(t(\xi(\cdot)A_1),t(A_2))\big)-p^{*}\big(\omega_{\mathfrak g}(t(A_1),t(\xi(\cdot)A_2))\big)\\
=&p^{*}\big(\omega_{\mathfrak g}( p^{*}((\gamma(\cdot))A_1)-[s(\cdot),t(A_1)]_{\mathfrak g},t(A_2))\big)\\
&+p^{*}\big(\omega_{\mathfrak g}(t(A_1),p^{*}((\gamma(\cdot))A_2)-[s(\cdot),t(A_2)]_{\mathfrak g})\big)\\
=&-p^{*}\big(\omega_{\mathfrak g}([s(\cdot),t(A_1)]_{\mathfrak g},t(A_2))\big)+p^{*}\big(\omega_{\mathfrak g}([s(\cdot),t(A_2)]_{\mathfrak g},t(A_1))\big)\\
=&p^{*}\big(\omega_{\mathfrak g}([t(A_1),t(A_2)]_{\mathfrak g},s(\cdot))\big)
\end{align*}
fr all $A_1,A_2\in\mathfrak a$.
Here we used 
\[\omega_\mathfrak a(A_1,A_2)=\omega_{\mathfrak j^\bot/\mathfrak j}(i(A_1),i(A_2))=\omega_\mathfrak g(t(A_1),t(A_2))\]
for the second equality.
Using 
\begin{align*}
\pi\big([t(A_1),t(A_2)]_\mathfrak g-t([A_1,A_2]_\mathfrak a)\big)&=[\pi t(A_1),\pi t(A_2)]_{\mathfrak g/\mathfrak j}-[i(A_1),i(A_2)]_{\mathfrak g/\mathfrak j}\\
&=[i(A_1),i(A_2)]_{\mathfrak g/\mathfrak j}-[i(A_1),i(A_2)]_{\mathfrak g/\mathfrak j}=0
\end{align*}
we get $[t(A_1),t(A_2)]_\mathfrak g-t([A_1,A_2]_\mathfrak a)\in\mathfrak j$. Moreover, $t([A_1,A_2]_{\mathfrak a})\in V_{\mathfrak a}$ and thus orthogonal to $V_{\mathfrak l}$.
We obtain 
\begin{align}\label{eq:pbeta}
p^*(\beta(A_1,A_2))=[t(A_1),t(A_2)]_{\mathfrak g}-t([A_1,A_2]_{\mathfrak a})
\end{align}
for all $A_1,A_2\in\mathfrak a$.
Now, we consider
\begin{align*}
\omega_{\mathfrak g}(t(\alpha(L_1,L_2)),t(A))=&\,\omega_{\mathfrak a}(\alpha(L_1,L_2),A)\\
=&\,\gamma(L_1,A,L_2)-\gamma(L_2,A,L_1)\\
=&\,\omega_{\mathfrak g}([s(L_1),t(A)]_{\mathfrak g}-t(\xi(L_1)A),s(L_2))\\
&\,-\omega_{\mathfrak g}([s(L_2),t(A)]_{\mathfrak g}-t(\xi(L_2)A),s(L_1))\\
=&\,\omega_{\mathfrak g}([s(L_1),t(A)]_{\mathfrak g},s(L_2))-\omega_{\mathfrak g}([s(L_2),t(A)]_{\mathfrak g},s(L_1))\\
=&\,\omega_{\mathfrak g}([s(L_1),s(L_2)]_{\mathfrak g},t(A))
\end{align*}
for all $A\in \mathfrak a$ and $L_1,L_2\in\mathfrak l$.
Since $p^{*}(\epsilon(L_1,L_2))$ lies in $\mathfrak j$, it is orthogonal to $V_{\mathfrak a}$. Moreover, $[s(L_1),s(L_2)]_{\mathfrak g}-p^{*}(\epsilon(L_1,L_2))\in V_{\mathfrak a}$, since on the one hand
\[\omega_\mathfrak g([s(L_1),s(L_2)]_\mathfrak g-p^*(\epsilon(L_1,L_2)),s(L))=\omega_\mathfrak g([s(L_1),s(L_2)]_\mathfrak g,s(L))-\epsilon(L_1,L_2)(L)=0\]
for all $L_1,L_2,L\in\mathfrak l$ and thus 
$[s(L_1),s(L_2)]_{\mathfrak g}-p^{*}(\epsilon(L_1,L_2))$ orthogonal to $V_\mathfrak l$.
On the other hand $[s(L_1),s(L_2)]_{\mathfrak g}-p^{*}(\epsilon(L_1,L_2))$ is also orthogonal to $\mathfrak j$, because $\mathfrak j$ is isotropic and $\omega_g([\mathfrak g,\mathfrak g],\mathfrak j)=0$.
We obtain 
\begin{align}\label{eq:talpha}
t(\alpha(L_1,L_2))=[s(L_1),s(L_2)]_{\mathfrak g}-p^{*}(\epsilon(L_1,L_2))
\end{align}
for all $L_1,L_2\in\mathfrak l$.

\begin{lemma}
The linear map \[\Phi:(\mathfrak l^*\oplus\mathfrak a\oplus\mathfrak l,[\cdot,\cdot]_{\gamma,\epsilon,\xi},\omega)\rightarrow(\mathfrak g,[\cdot,\cdot]_{\mathfrak g},\omega_\mathfrak g), \Phi:Z+A+L\mapsto p^{*}(Z)+t(A)+s(L)\] is an isomorphism of symplectic Lie algebras.
\end{lemma}

\begin{proof}[Proof of the lemma]
Clearly $\Phi$ is bijective.
Since $\mathfrak l^*$ and $\mathfrak j$ lie in the center, we have \begin{align*}
\Phi([\mathfrak l^*,\mathfrak l^*\oplus\mathfrak a\oplus\mathfrak l])=0=[\Phi(\mathfrak l^*),\Phi(\mathfrak l^*\oplus\mathfrak a\oplus\mathfrak l)]_{\mathfrak g}.
\end{align*}
Using the formulas (\ref{eq:pbeta}) and (\ref{eq:talpha}) for $\beta$ and $\alpha$ we obtain that $\Phi$ respects the commutator.

At this moment we have seen that $(\mathfrak g,[\cdot,\cdot]_{\mathfrak g},\omega_{\mathfrak g})$ and $(\mathfrak l^*\oplus\mathfrak a\oplus\mathfrak l,[\cdot,\cdot],\Phi^*\omega_{\mathfrak g})$ are isomorphic.
Moreover, it is easy to see that $\Phi^*\omega_{\mathfrak g}=\omega$. Thus $\Phi$ is an isomorphism of symplectic Lie algebras.
\end{proof}

Since $(\mathfrak l^*\oplus\mathfrak a\oplus\mathfrak l,[\cdot,\cdot]_{\gamma,\epsilon,\xi},\omega)$ is a symplectic Lie algebra, the triple $(\gamma,\epsilon,\xi)$ lies in $Z^2(\mathfrak l,\mathfrak a,\omega_\mathfrak a).$
In addition, we get that $\Phi:\mathfrak d_{\gamma,\epsilon,\xi}(\mathfrak l,\mathfrak a)\rightarrow\mathfrak g$ is an equivalence of the quadratic extensions $\mathfrak g$ and $\mathfrak d_{\gamma,\epsilon,\xi}(\mathfrak l,\mathfrak a)$, because on the one hand we have
$i(A)=\pi(t(A))=\overline\Phi(A)$ and on the other hand $p(\Phi(L))=p(s(L))=L$ for all $A\in\mathfrak a$ and $L\in\mathfrak l$.

For $\mathfrak{j}=\mathfrak{z}\cap\mathfrak{z}^\bot$ the isomorphism $\Phi$ is an equivalence from $\mathfrak{d}_{\gamma,\epsilon,\xi}(\mathfrak l,\mathfrak a)$ to $(\mathfrak{g},\mathfrak{j},i,p)$ with $\Phi(\mathfrak{l}^*)=\mathfrak{j}$.
Thus $(\gamma,\epsilon,\xi)$ is balanced.
\end{proof}

\subsection{Equivalence classes}\label{SLA:Äquivkl}
In this section we define the second quadratic cohomology set of abelian Lie algebras with values in symplectic ones. We show that the set of equivalence classes of quadratic extensions of an abelian Lie algebra $\mathfrak l$ by a symplectic Lie algebra $(\mathfrak a,\omega_\mathfrak a)$ is in one-to-one correspondence with the quadratic cohomology set $H^2(\mathfrak l,\mathfrak a,\omega_{\mathfrak a})$.

\begin{definition}
We call $\bar{\sigma}:\mathfrak l\rightarrow \mathfrak l^*$ selfadjoint, if $\overline\sigma=\overline\sigma^*$, i. e. 
$(\bar{\sigma} (L_1))(L_2)=(\bar{\sigma} (L_2))(L_1)$ for all $L_1,L_2\in\mathfrak l$.
Recall that the dual map $\tau^*$ for $\tau:\mathfrak l\rightarrow\mathfrak a$ is given by
\[\tau^*:\mathfrak a\rightarrow\mathfrak l^*,\quad \tau^*(A)L=\omega_{\mathfrak a}(\tau L,A).\]
\end{definition}

\begin{definition}
For a $\tau\in C^1(\mathfrak l,\mathfrak a)$ and a triple
$(\gamma,\epsilon,\xi)$, consisting of linear maps $\gamma:\mathfrak l\rightarrow\Hom(\mathfrak a,\mathfrak l^*)$, $\epsilon\in C^2(\mathfrak l,\mathfrak l^*)$ and $\xi:\mathfrak l\rightarrow\Hom(\mathfrak a)$, we define
\begin{align}\label{eq:gex1}
(\gamma,\epsilon,\xi)\tau=(&\gamma+\tilde\tau^*\circ\xi-\tilde\tau^*\circ \ad_{\mathfrak a}\circ\tau-\beta_0\circ\tau,\\
&\epsilon+\tau^*\alpha-\ev(\gamma\wedge\tau)-\tau^*\circ \rd_\xi\tau+\frac{1}{2}\beta(\tau\wedge\tau),\,\xi-\ad_{\mathfrak a}\circ\tau).
\end{align}
Here the linear map $\tilde\tau^*:\Hom(\mathfrak a)\rightarrow \Hom(\mathfrak a,\mathfrak l^*)$ is given by$(\tilde\tau^* \varphi)(A)=\tau^*(\varphi(A))$ for all $\varphi\in \Hom(\mathfrak a)$ and $A\in\mathfrak a$.
\end{definition}

\begin{lemma} \label{lem:Eq:1}
The quadratic extensions $\mathfrak d_{\gamma,\epsilon,\xi}(\mathfrak l,\mathfrak a)$ and $\mathfrak d_{\gamma',\epsilon'\xi'}(\mathfrak l,\mathfrak a)$ are equivalent, if and only if there is an $\tau\in C^1(\mathfrak{l},\mathfrak{a})$ such that $(\gamma,\epsilon,\xi)\tau=(\gamma',\epsilon',\xi')$.
\end{lemma}

\begin{proof}
The quadratic extensions $\mathfrak d_{\gamma,\epsilon,\xi}(\mathfrak l,\mathfrak a)$ and $\mathfrak d_{\gamma',\epsilon'\xi'}(\mathfrak l,\mathfrak a)$ are equivalent, if and only if there is an isomorphism $\varphi$ of Lie algebras, which has the form (with respect to $\mathfrak l^*\oplus\mathfrak a\oplus\mathfrak l$)
\[\varphi=\begin{pmatrix}\psi&c&b\\0&\id&\tau\\0&0&\id\end{pmatrix}\] for some linear $\tau:\mathfrak l\rightarrow \mathfrak a$, $b:\mathfrak l\rightarrow \mathfrak l^*$ and $c:\mathfrak a\rightarrow \mathfrak l^*$, and satisfies
$\varphi^*\omega=\omega$.

Here $\varphi^*\omega=\omega$ is equivalent to $\psi=id$, $c=\tau^*$ and $b=\bar\sigma+\frac{1}{2}\tau^*\tau$ for some selfadjoint $\bar\sigma$ because of
\begin{align*}
Z(L)&=\omega(\varphi(Z),\varphi(L))=(\psi(Z))(L),\\
0=\omega(\varphi(A),\varphi(L))&=\omega(A+cA,L+\tau L+bL)=-(\tau^*(A))(L)+(c(A))(L)
\end{align*}
and
\begin{align*}
0=\omega(\varphi(L_1),\varphi(L_2))=(b(L_1))(L_2)-(b(L_2))(L_1)-(\tau^*\tau(L_1))L_2.
\end{align*}

Since an equivalence $\varphi$ from $\mathfrak d_{\gamma,\epsilon,\xi}(\mathfrak l,\mathfrak a)$ to $\mathfrak d_{\gamma',\epsilon',\xi'}(\mathfrak l,\mathfrak a)$ respects the commutator, we have
\begin{align*}
[A_1,A_2]_{\mathfrak a}+\beta'(A_1,A_2)&=[\varphi(A_1),\varphi(A_2)]_{\gamma',\epsilon',\xi'}=\varphi([A_1,A_2]_{\gamma,\epsilon,\xi})\\
&=[A_1,A_2]_{\mathfrak a}+\tau^*([A_1,A_2]_{\mathfrak a})+\beta(A_1,A_2)
\end{align*}
for every $A_1,A_2\in\mathfrak a$.
In addition, 
\begin{align*}
\xi'(L)A&+\gamma'(L)A+[\tau(L),A]_{\mathfrak a}+\beta'(\tau(L),A)=[\varphi(L),\varphi(A)]_{\gamma',\epsilon',\xi'}\\
&=\varphi([L,A]_{\gamma,\epsilon,\xi})=\xi(L)A+\tau^*(\xi(L)A)+\gamma(L)A
\end{align*}
and
\begin{align*}
\alpha'(L_1,L_2)&+\epsilon'(L_1,L_2)+(\xi'(L_1))(\tau L_2)+(\gamma'(L_1))(\tau L_2)-(\xi'(L_2))(\tau L_1)\\
-(\gamma'(L_2)&(\tau L_1))+[\tau L_1,\tau L_2]_{\mathfrak a}+\beta'(\tau L_1,\tau L_2)\displaybreak\\
&=\,[\varphi(L_1),\varphi(L_2)]_{\gamma',\epsilon',\xi'}=\varphi([L_1,L_2]_{\gamma,\epsilon,\xi})\\
&=\,\alpha(L_1,L_2)+\tau^*(\alpha(L_1,L_2))+\epsilon(L_1,L_2).
\end{align*}
Finally equating the coefficients yields
\begin{align*}
\xi'(L)A&=\xi(L)A-[\tau(L),A]_{\mathfrak a},\\
\gamma'(L)A&=\gamma(L)A+\tau^*(\xi(L)A)-\beta'(\tau(L),A)\\
           &=\gamma(L)A+\tau^*(\xi(L)A)-\tau^*(((\ad_{\mathfrak a}\circ\tau)(L))(A))-((\beta_0\circ \tau)L)(A)
\end{align*}
and 
\begin{align*}
\epsilon'(L_1,L_2)=&-\gamma'(L_1)(\tau L_2)+\gamma'(L_2)(\tau L_1)-\beta'(\tau L_1,\tau L_2)+\tau^*(\alpha(L_1,L_2))+\epsilon(L_1,L_2)\\
=&-\gamma(L_1)(\tau L_2)-\omega_{\mathfrak a}(\tau(\cdot),\xi(L_1)\tau L_2)+\tau^*([\tau L_1,\tau L_2]_{\mathfrak a})+\beta(\tau L_1,\tau L_2)\\
&+\gamma(L_2)(\tau L_1)+\omega_{\mathfrak a}(\tau(\cdot),\xi(L_2)\tau L_1)+\tau^*(\alpha(L_1,L_2))+\epsilon(L_1,L_2).
\end{align*}
Hence $(\gamma,\epsilon,\xi)\tau=(\gamma',\epsilon',\xi')$.

Conversely, it is straight forward that  \[\varphi_{\tau,\bar\sigma}:=\begin{pmatrix}id&\tau^*&\bar\sigma+\frac{1}{2}\tau^*\tau\\0&id&\tau\\0&0&id\end{pmatrix}\] respects the commutator for $(\gamma,\epsilon,\xi)\tau=(\gamma',\epsilon',\xi')$ and for every selfadjoint $\bar\sigma:\mathfrak{l}\rightarrow\mathfrak{l}^*$.
\end{proof}

\begin{remark}
The map $\varphi_{\tau,\bar\sigma}$ is an equivalence from $\mathfrak{d}_{\gamma,\epsilon,\xi}$ to $\mathfrak{d}_{\gamma',\epsilon',\xi'}$, if and only if $(\gamma,\epsilon,\xi)\tau=(\gamma',\epsilon',\xi')$.
Since, moreover, $\varphi_{\tau_2,0}\circ\varphi_{\tau_1,0}=\varphi_{\tau_1+\tau_2,\bar\sigma}$ holds for the selfadjoint $\bar\sigma=\frac{1}{2}\tau_2^*\tau_1-\frac{1}{2}\tau_1^*\tau_2$, we have
$((\gamma,\epsilon,\xi)\tau_1)\tau_2=(\gamma,\epsilon,\xi)(\tau_1+\tau_2)$.
This shows that equation (\ref{eq:gex1}) defines a group action of $C^1(\mathfrak{l},\mathfrak{a})$ on the set of all triples, consisting of linear maps $\gamma:\mathfrak l\rightarrow\Hom(\mathfrak a,\mathfrak l^*)$, $\epsilon\in C^2(\mathfrak l,\mathfrak l^*)$ and $\xi:\mathfrak l\rightarrow\Hom(\mathfrak a)$.
Furthermore, using Theorem \ref{Thm:Z2} and Theorem \ref{Thm:xa} this group action of $C^1(\mathfrak{l},\mathfrak{a})$ leaves the sets $Z^2(\mathfrak{l},\mathfrak{a},\omega_\mathfrak{a})$ and $Z^2(\mathfrak{l},\mathfrak{a},\omega_\mathfrak{a})_\sharp$
invariant.
\end{remark}

We define 
\[H^2(\mathfrak l,\mathfrak a,\omega_{\mathfrak a}):=Z^2(\mathfrak l,\mathfrak a,\omega_{\mathfrak a})/C^1(\mathfrak l,\mathfrak a)\quad\text{and}\quad H^2(\mathfrak l,\mathfrak a,\omega_{\mathfrak a})_\sharp:=Z^2(\mathfrak l,\mathfrak a,\omega_{\mathfrak a})_\sharp/C^1(\mathfrak l,\mathfrak a)\]
We call $H^2(\mathfrak l,\mathfrak a,\omega_{\mathfrak a})$
the quadratic cohomology of the abelian Lie algebra $\mathfrak l$ with values in the symplectic Lie algebra $(\mathfrak a,\omega_\mathfrak a)$.
For a given cocycle $(\gamma,\epsilon,\xi)$ let $[\gamma,\epsilon,\xi]$ denote itscohomology class.

\begin{corollary}
The cohomology class of the triple $(\gamma,\epsilon,\xi)$ defined by equations (\ref{def:xi})-(\ref{def:epsilon}) does not depend on the choice of the section $s$.
\end{corollary}

\begin{proof}
For two sections $s_1,s_2:\mathfrak l\rightarrow V_\mathfrak l$ let 
$(\gamma_1,\epsilon_1,\xi_1),(\gamma_2,\epsilon_2,\xi_2)\in Z^2(\mathfrak l,\mathfrak a,\omega_\mathfrak a)$ denote the corresponding cocycles, which are given by equations (\ref{def:xi})-(\ref{def:epsilon}).
Then $\mathfrak d_{\gamma_1,\epsilon_1,\xi_1}(\mathfrak l,\mathfrak a)$ and $\mathfrak d_{\gamma_2,\epsilon_2,\xi_2}(\mathfrak l,\mathfrak a)$ are equivalent, since both quadratic extensions are equivalent to $(\mathfrak g,\mathfrak j,i,p)$.
Using the previous lemma we get $[\gamma_1,\epsilon_1,\xi_1]=[\gamma_2,\epsilon_2,\xi_2]$.
\end{proof}

As the final result of this section we have the following theorem.
\begin{theorem}
The equivalence classes of quadratic extensions of an abelian Lie algebra $\mathfrak l$ by an symplectic Lie algebra $(\mathfrak a,\omega_\mathfrak a)$ is in bijection to the elements of 
$H^2(\mathfrak l,\mathfrak a,\omega_{\mathfrak a})$.
Moreover, the equivalence classes of balanced quadratic extensions is in bijection to  
$H^2(\mathfrak l,\mathfrak a,\omega_{\mathfrak a})_\sharp$.
\end{theorem}

\subsection{Isomorphism classes of symplectic Lie algebras}\label{SLA:Isomkl}

The aim of this section is to describe the isomorphism classes of symplectic Lie algebras using a suitable classification scheme. Therefor, we determine the isomorphy of symplectic Lie algebras.
\\\quad\\
Therefor, we consider pairs $(\mathfrak l,\mathfrak a)$ of abelian Lie-Algebras $\mathfrak l$ and symplectic Lie algebras $(\mathfrak a,\omega_\mathfrak a)$. These pairs form a category, whose morphisms are pairs $(S,U):(\mathfrak l_1,\mathfrak a_1)\rightarrow(\mathfrak l_2,\mathfrak a_2)$ containing a homomorphism $S:\mathfrak l_1\rightarrow\mathfrak l_2$ of vector spaces and a homomorphism $U:\mathfrak a_2\rightarrow\mathfrak a_1$ of symplectic Lie algebras. We will denote this morphisms by morphisms of pairs.

Suppose that $(S,U):(\mathfrak l_1,\mathfrak a_1)\rightarrow(\mathfrak l_2,\mathfrak a_2)$ is an isomorphism of pairs.
We have pull back maps
$(S,U)^*:C^p(\mathfrak l_2,\mathfrak a_2)\rightarrow C^p(\mathfrak l_1,\mathfrak a_1)$ and $(S,U)^*:C^p(\mathfrak l_2)\rightarrow C^p(\mathfrak l_1)$ given by
\begin{align*}
(S,U)^*\alpha(L_1,\dots,L_p)=U\circ\alpha(SL_1,\dots,SL_p),\quad
(S,U)^*\gamma(L_1,\dots,L_p)=\gamma(SL_1,\dots,SL_p).
\end{align*}
Now assume that $S:\mathfrak l_1\rightarrow \mathfrak l_2$ and $U:\mathfrak a_2\rightarrow \mathfrak a_1$ are isomorphisms of Lie algebras satisfying $U^*\omega_{\mathfrak a_1}=\omega_{\mathfrak a_2}$.
We define a pull back map $(S,U)^*$ from $\Hom(\mathfrak l_2,\Hom(\mathfrak a_2,\mathfrak l_2^*))\oplus C^2(\mathfrak l_2,\mathfrak l_2^*)\oplus \Hom(\mathfrak l_2,\Hom(\mathfrak a_2))$
to 
$\Hom(\mathfrak l_1,\Hom(\mathfrak a_1,\mathfrak l_1^*))\oplus C^2(\mathfrak l_1,\mathfrak l_1^*)\oplus \Hom(\mathfrak l_1,\Hom(\mathfrak a_1))$ by
\[(S,U)^*(\gamma,\epsilon,\xi)=((S,U)^*\gamma,(S,U)^*\epsilon,(S,U)^*\xi)\]
where
\begin{itemize}
\item $((S,U)^*\gamma(L))(A)=S^*(\gamma(SL)(U^{-1}A))$,
\item $(S,U)^*\epsilon(L_1,L_2)=S^*(\epsilon(SL_1,SL_2))$,
\item $(S,U)^*\xi(L)A=(U\xi(SL))(U^{-1}A)$.
\end{itemize}

Let $(S,U)$ be an isomorphism of pairs and $(\gamma,\epsilon,\xi)\in Z^2(\mathfrak l,\mathfrak a,\omega_\mathfrak a)$.
Then
\begin{align*}
&\omega_\mathfrak a( ((S,U)^*\gamma(L_1))^*(L_2)-((S,U)^*\gamma(L_2))^*(L_1) ,A)\\
&=(((S,U)^*\gamma(L_1))(A))(L_2)-(((S,U)^*\gamma(L_2))(A))(L_1)\\
&=((\gamma(SL_1))(U^{-1}A))(SL_2)-((\gamma(SL_2))(U^{-1}A))(SL_1)\\
&=\omega_\mathfrak a( (\gamma(SL_1))^*(SL_2)-(\gamma(SL_2))^*(SL_1) ,U^{-1}A)\\
&=\omega_\mathfrak a( U\circ\alpha(SL_1,SL_2),A)
\end{align*}
and
\begin{align*}
-&\omega_\mathfrak a(U\xi(S\cdot)U^{-1}A_1,A_2)-\omega_\mathfrak a(A_1,U\xi(S\cdot)U^{-1}A_2)\\
&=-\omega_\mathfrak a(\xi(S\cdot)U^{-1}A_1,U^{-1}A_2)-\omega_\mathfrak a(U{-1}A_1,\xi(S\cdot)U^{-1}A_2)\\
&=S^*\beta(U^{-1}A_1,U^{-1}A_2)
\end{align*}
for all $L_1,L_2\in\mathfrak l_1$ and $A,A_1,A_2\in\mathfrak a_1$.
This shows that the mappings $\alpha'$ and $\beta'$ corresponding to the triple $(\gamma',\epsilon',\xi')=((S,U)^*\gamma,(S,U)^*\epsilon,(S,U)^*\xi)$ are given by
\begin{align}\label{eq:astrich}
\alpha'(L_1,L_2)=U\circ\alpha(SL_1,SL_2)
\end{align}
and 
\begin{align}\label{eq:bstrich}
\beta'(A_1,A_2)(L_1)=\beta(U^{-1}A_1,U^{-1}A_2)(SL_1).
\end{align}
So, it is easy to show that equation (\ref{eq:lemma2}) holds for $(\gamma',\epsilon',\xi')$:
\begin{align*}
\beta'(\xi'(L_1)&A_1,A_2)(L_2)+\beta'(A_1,\xi'(L_1)A_2)(L_2)\\
&=\beta(U^{-1}\xi'(L_1)A_1,U^{-1}A_2)(SL_2)+\beta(U^{-1}A_1,U^{-1}\xi'(L_1)A_2)(SL_2)\\
&=\gamma(SL_1,[U^{-1}A_1,U^{-1}A_2]_\mathfrak a,SL_2)\\
&=\gamma'(L_1,[A_1,A_2]_\mathfrak a,L_2).
\end{align*}
The other equations (\ref{eq:lemma1}) and (\ref{eq:lemma3})-(\ref{eq:lemma5}) can be verified analogous with equation (\ref{eq:astrich}) and (\ref{eq:bstrich}).
So, $(S,U)^*$ maps the space of cocycles $Z^2(\mathfrak l_2,\mathfrak a_2,\omega_{\mathfrak a_2})$ to $Z^2(\mathfrak l_1,\mathfrak a_1,\omega_{\mathfrak a_1})$.
Since the pull back map respects the group action in the following sense
\[(S,U)^*((\gamma,\epsilon,\xi)\tau)=((S,U)^*(\gamma,\epsilon,\xi))((S,U)^*\tau),\]
we also have a pull back map from $H^2(\mathfrak l_2,\mathfrak a_2,\omega_{\mathfrak a_2})$ to $H^2(\mathfrak l_1,\mathfrak a_1,\omega_{\mathfrak a_1})$ given by \[(S,U)^*[\gamma,\epsilon,\xi]=[(S,U)^*(\gamma,\epsilon,\xi)].\]

\begin{lemma}
Let $(\mathfrak g_k,\mathfrak j_k,i_k,p_k)$ be balanced quadratic extensions of $\mathfrak l_k$ by $(\mathfrak a_k,\omega_{\mathfrak a_k})$ for $k=1,2$.
The symplectic Lie algebras $\mathfrak g_1$ and $\mathfrak g_2$ are isomorphic if and only if there is an isomorphism of pairs $(S,U):(\mathfrak l_1,\mathfrak a_1)\rightarrow(\mathfrak l_2,\mathfrak a_2)$ such that the quadratic extensions $(\mathfrak g_1,\mathfrak j_1,i_1\circ U,S\circ p_1)$ and $(\mathfrak g_2,\mathfrak j_2,i_2,p_2)$ are equivalent.
\end{lemma}

\begin{proof}
One direction is trivial, since every equivalence of quadratic extensions is by definition an isomorphism of the corresponding symplectic Lie algebras.
So, assume that $F:\mathfrak g_1\rightarrow\mathfrak g_2$ is an isomorphism of symplectic Lie algebras.
Because of $F(\mathfrak j_1)=\mathfrak j_2$ and $F(\mathfrak j_1^\bot)=\mathfrak j_2^\bot$ the isomorphism $F$ induces an isomorphism $\overline F:\mathfrak g_1/\mathfrak j_1\rightarrow\mathfrak g_2/\mathfrak j_2$, which satisfies $\ker p_2=\overline F(\ker p_1)$.
Now, we define $S:\mathfrak l_1\rightarrow\mathfrak l_2$ and $U:\mathfrak a_2\rightarrow\mathfrak a_1$ by 
\[S(L)=(p_2\circ\overline F)(\tilde L),\quad i_1\circ U=\overline F^{-1}\circ i_2\]
for $p_1(\tilde L)=L$.
It is not hard to see, that $S$ is an isomorphism of vector spaces and $U$ an isomorphism of symplectic Lie algebras. Thus $F$ becomes an equivalence between the quadratic extensions $(\mathfrak g_1,\mathfrak j_1,i_1\circ U,S\circ p_1)$ and $(\mathfrak g_2,\mathfrak j_2,i_2,p_2)$ of $\mathfrak l_2$ by $(\mathfrak a_2,\omega_{\mathfrak a_2})$.
\end{proof}

\begin{theorem}\label{thm:15}
Assume $(\gamma_i,\epsilon_i,\xi_i)\in Z^2(\mathfrak l_i,\mathfrak a_i,\omega_{\mathfrak a_i})_\sharp$ for $i=1,2$.
The symplectic Lie algebras 
$\mathfrak d_{\gamma_1,\epsilon_1,\xi_1}(\mathfrak l_1,\mathfrak a_1)$ and $\mathfrak d_{\gamma_2,\epsilon_2,\xi_2}(\mathfrak l_2,\mathfrak a_2)$ are isomorphic if and only if there is an isomorphism $S:\mathfrak l_1\rightarrow\mathfrak l_2$ of vector spaces and an isomorphism $U:(\mathfrak a_2,\omega_{\mathfrak a_2})\rightarrow(\mathfrak a_1,\omega_{\mathfrak a_1})$ such that
\[(S,U)^*[\gamma_2,\epsilon_2,\xi_2]=[\gamma_1,\epsilon_1,\xi_1].\]
\end{theorem}

\begin{proof}
From the previous lemma, we already know that the symplectic Lie algebras \linebreak $\mathfrak d_{\gamma_1,\epsilon_1,\xi_1}(\mathfrak l_1,\mathfrak a_1)$ and $\mathfrak d_{\gamma_2,\epsilon_2,\xi_2}(\mathfrak l_2,\mathfrak a_2)$ are isomorphic if and only if there is an isomorphism $S:\mathfrak l_1\rightarrow\mathfrak l_2$ of vector spaces and an isomorphism $U:\mathfrak a_2\rightarrow\mathfrak a_1$ of symplectic Lie algebras such that the quadratic extensions $(\mathfrak d_{\gamma_1,\epsilon_1,\xi_1}(\mathfrak l_1,\mathfrak a_1),\mathfrak l^*_1,i_1\circ U,S\circ p_1)$ and $(\mathfrak d_{\gamma_2,\epsilon_2,\xi_2}(\mathfrak l_2,\mathfrak a_2),\mathfrak l^*_2,i_2,p_2)$ are equivalent.

We define a section $\tilde s:\mathfrak l_2\rightarrow\mathfrak d_{\gamma_1,\epsilon_1,\xi_1}(\mathfrak l_1,\mathfrak a_1)$ of $S\circ \tilde p_1$ as the composition of $S^{-1}$  and the canonical embedding $s:\mathfrak l_1\rightarrow\mathfrak d_{\gamma_1,\epsilon_1,\xi_1}(\mathfrak l_1,\mathfrak a_1)$. Here $\tilde p_1$ denotes the canonical projection from $\mathfrak d_{\gamma_1,\epsilon_1,\xi_1}(\mathfrak l_1,\mathfrak a_1)$ to $\mathfrak l_1$.
Now, the cocyle associated to this quadratic extension $(\mathfrak d_{\gamma_1,\epsilon_1,\xi_1}(\mathfrak l_1,\mathfrak a_1),\mathfrak l^*_1,i_1\circ U,S\circ p_1)$ of $\mathfrak l_2$ by $\mathfrak a_2$ and this section $\tilde s$ is given by
\[(S^{-1},U^{-1})^*(\gamma_1,\epsilon_1,\xi_1)\]
(see (\ref{def:xi}) - (\ref{def:epsilon}) ).
Using Lemma \ref{lem:Eq:1} yields the assumption.
\end{proof}

\begin{remark}
Using Theorem \ref{Thm:xa} and Theorem \ref{thm:15} we obtain that isomorphisms of pairs pull back balanced cocycles in $Z^2(\mathfrak{l}_2,\mathfrak{a}_2,\omega_{\mathfrak{a}_2})$ to balanced cocycles in $Z^2(\mathfrak{l}_1,\mathfrak{a}_1,\omega_{\mathfrak{a}_1})$.
\end{remark}

Let $\mathfrak l$ be an abelian Lie algebra and $(\mathfrak a,\omega_\mathfrak a)$ a symplectic one. Let $G(\mathfrak l,\mathfrak a,\omega_\mathfrak a)$ denote the set of all pairs of automorphisms $S$ of the vector space $\mathfrak l$ and automorphisms $U$ of the symplectic Lie algebra $(\mathfrak a,\omega_\mathfrak a)$.
If it is clear from the context which $(\mathfrak l,\mathfrak a,\omega_\mathfrak a)$ is considered, then we simply write $G$ instead of $G(\mathfrak l,\mathfrak a,\omega_\mathfrak a)$.

We conclude this section with the following classification theorem. It follows immediately from Theorem \ref{Th:1}, Theorem \ref{Th:2} and Theorem \ref{thm:15}.

\begin{theorem}\label{Th:Klassifik}
The set of isomorphism classes of symplectic Lie algebras,
\begin{itemize}
\item whose ideal $\mathfrak j=\mathfrak z\cap\mathfrak z^\bot$ is $n$-dimensional and
\item associated symplectic Lie algebra $(\nicefrac{\mathfrak j^\bot}{\mathfrak j},\overline\omega)$ is isomorphic to $(\mathfrak a,\omega_{\mathfrak a})$
\end{itemize}
is in one-to-one correspondence to the orbit space
\[H^2(\mathds R^n,\mathfrak a,\omega_{\mathfrak a})_\sharp/G.\]

Let $\operatorname{Isom}(2n)$ denote a system of representatives of the isomorphism classes of symplectic Lie algebras of dimension $2n$. 
There is a bijection between the set of all isomorphism classes of symplectic Lie algebras of dimension $2n$ with degenerate center and the set  
\[\coprod_{m=1}^n\quad\coprod_{[\mathfrak a,\omega_{\mathfrak a}]\in \operatorname{Isom}(2n-2m)} H^2(\mathds R^m,\mathfrak a,\omega_{\mathfrak a})_\sharp/G.\] 
\end{theorem}

\subsection{Nilpotent symplectic Lie algebras}\label{SLA:Anw}

While the classification of all symplectic Lie algebras with degenerate center of a fixed dimension with quadratic extensions uses the knowledge of the isomorphism classes of all symplectic Lie algebras of lower dimension and thus needs the classification of symplectic Lie algebras with non-degenrated center, nilpotent symplectic Lie algebren can be classified without such a knowledge.

\begin{lemma}
Every nilpotent symplectic Lie algebra $(\mathfrak g,\omega)$ with non-degenerate center is abelian.
\end{lemma}

\begin{proof}
If the center $\mathfrak z$ of a symplectic Lie algebra $\mathfrak g$ is non-degenerate, then $\mathfrak g$ is the direct sum of an abelian symplectic Lie algebra $\mathfrak z$ and the symplectic Lie algebra $\mathfrak z^\bot$ with trivial center. Since $\mathfrak g$ is nilpotent, $\mathfrak z^\bot$ is also nilpotent and thus $\mathfrak z^\bot=\{0\}$. Hence, $\mathfrak g$ is abelian.
\end{proof}

\begin{lemma} \label{Lemma:Nilpotenz}
Assume $(\gamma,\epsilon,\xi)\in Z^2(\mathfrak l,\mathfrak a,\omega_{\mathfrak a})_\sharp$. The sympectic Lie algebra $\mathfrak d_{\gamma,\epsilon,\xi}(\mathfrak l,\mathfrak a,\omega_a)$ is nilpotent if and only if $\mathfrak{a}$ is nilpotent and
\begin{align}\label{eq:Nilpotenz}
\xi(L_1)\cdot\dots\cdot\xi(L_n)=0
\end{align}
for some $n\in\mathds N$ and all $L_1,\dots,L_n\in\mathfrak l$.
\end{lemma}

\begin{proof}
Let $\mathfrak d:=\mathfrak d_{\gamma,\epsilon,\xi}(\mathfrak l,\mathfrak a,\omega_a)$ be nilpotent.
Since $\mathfrak a$ is isomorphic to the nilpotent Lie-Algebra $\mathfrak l^{*^{\bot}}/\mathfrak l^*$, we obtain that $\mathfrak a$ is nilpotent.
Moreover, $\mathfrak d^{n+1}=0$ for some $n>0$ implies
\[ \xi(L_1)\cdots\xi(L_n)A=[L_1,[\dots[L_n,A]\dots]]=0 \] for all $L_1,\dots,L_n\in\mathfrak l$ and all $A\in\mathfrak a$. Thus $\xi(L_1)\cdots\xi(L_n)=0\in\mathfrak{gl}(\mathfrak a)$.
Conversely, assume $\mathfrak a^{m}=\{0\}$ for some $m>0$ and equation (\ref{eq:Nilpotenz}). It is not hard to see that $\mathfrak d^{{nm+2}}\in\mathfrak l^*$ and thus $\mathfrak d^{{nm+3}}=\{0\}$. This shows that $\mathfrak d$ is nilpotent.
\end{proof}

\begin{definition}
For an abelian Lie algebra $\mathfrak l$ and a nilpotent symplectic Lie algebra $\mathfrak a$ let $Z^2(\mathfrak l,\mathfrak a,\omega_{\mathfrak a})_0\subset Z^2(\mathfrak l,\mathfrak a,\omega_\mathfrak a)_\sharp$ denote the set of cocycles $(\gamma,\epsilon,\xi)$, which satisfy equation (\ref{eq:Nilpotenz}). We call these cocycles nilpotent cocycles. Moreover, we define $H^2(\mathfrak l,\mathfrak a,\omega_{\mathfrak a})_0\subset H^2(\mathfrak l,\mathfrak a,\omega_\mathfrak a)_\sharp$ as the set of cohomology classes of nilpotent cocycles.
\end{definition}

\begin{theorem}
The set of isomorphism classes of non-abelian nilpotent $2n$-dimensional symplectic Lie algebras is in one-to-one correspondence to the set
\[\coprod_{m=1}^n\quad\coprod_{[\mathfrak a,\omega_{\mathfrak a}]\in \mathfrak A(2n-2m)} H^2(\mathds R^m,\mathfrak a,\omega_{\mathfrak a})_0/G,\] 
where $\mathfrak A(2n-2m)$ denotes a system of representatives of the isomorphism classes of $(2n-2m)$-dimensional nilpotent symplectic Lie algebras.
\end{theorem}

\section{Nilpotent symplectic Lie algebras of dimension six}\label{NSLA}

In \cite{GKM04}, a list of all $6$-dimensional nilpotent symplectic Lie algebras (up to isomorphism) is given.
We compute this list on new way with the help of our new classification scheme and correct some little mistakes in the list in \cite{GKM04}.
\\\quad\\
In the following, as long as nothing else is mentioned, let $\{e_1,\dots,e_k\}$ denote the standard basis of the $k-$dimensional abelian Lie algebra $\mathfrak l=\mathds R^k$ and $\{\sigma^1,\dots,\sigma^k\}$ the corresponding dual basis of $\mathfrak l^*$.

We call a basis $\{v_1,w_1,\dots,v_{n},w_{n}\}$ of the $2n$-dimensional symplectic Lie algebra $(\mathfrak{a},\omega_\mathfrak{a})$  Darboux basis,
if $\omega_\mathfrak{a}(v_{i},w_{i})=1$ for all $i=1,\dots,n$ and $\omega_\mathfrak{a}(v_i,v_j)=\omega_\mathfrak{a}(w_i,w_j)=0$ for all $i,j=1,\dots,n$.

We will describe a Lie algebra by giving a basis and all non-vanishing brackets of basis vectors. If all basis vectors appear in the bracket relations, then we do not mention the basis explicitly.
The Lie algebra $\mathds{R}^2=\Span\{a_1,a_2\}$ becomes a symplectic Lie algebra with $\omega_\mathfrak a=\alpha^1\wedge\alpha^2$,
where $\{\alpha^1,\alpha^2\}$ denotes the dual basis of $\{a_1,a_2\}$.

Moreover, we identify $\Hom(U,V)$ with $U^*\otimes V$ for vector spaces $U$ and $V$.
\\\quad\\
Recall the following well-known theorem.
\begin{theorem}\label{Th:Klass4dim}
Every $4-$dimensional nilpotent symplectic Lie algebra is isomorphic to exactly one of the following one:
\begin{itemize}
\item $\mathds R^4=\Span\{a_1,\dots,a_4\}$ where $\omega=\alpha^1\wedge \alpha^2+\alpha^3\wedge \alpha^4$,
\item $\mathfrak n_4:=\{[a_4,a_1]=a_2,[a_4,a_2]=a_3\}$ where $\omega=\alpha^1\wedge \alpha^2+\alpha^3\wedge \alpha^4$ or
\item $\mathfrak h_3\oplus\mathds R:=\{a_1,\dots,a_4 \mid [a_1,a_2]=a_3\}$ where  $\omega=\alpha^1\wedge\alpha^4+\alpha^2\wedge\alpha^3$.
\end{itemize}
Here $\{\alpha^1,\dots,\alpha^4\}$ denotes the dual basis of $\{a_1,\dots,a_4\}$.
\end{theorem}

In the following, we determine all $6$-dimensional nilpotent symplectic Lie algebras up to isomorphism. From our previous observations, we know that it suffices to calculate the $G$-orbits of 
\[H^2(\mathds R^3,\{0\},0)_0,\quad H^2(\mathds R^2,\mathds R^2,\alpha^1\wedge \alpha^2)_0,\]
\[H^2(\mathds R,\mathds R^4,\omega_\mathfrak a)_0,\quad H^2(\mathds R,\mathfrak h_3\oplus\mathds R,\omega_\mathfrak a)_0\quad\text{ and }\quad H^2(\mathds R,\mathfrak n_4,\omega_\mathfrak a)_0,\]
where $\omega_\mathfrak a$ is given as in Theorem \ref{Th:Klass4dim}.

\subsection{Case $\mathfrak l=\mathds R^3$, $\mathfrak a=\{0\}$}\label{NSLA:1}

Let $\{e_1,e_2,e_3\}$ denote the standard basis of $\mathfrak l=\mathds R^3$. Assume $\mathfrak a=\{0\}$ and let $\{\sigma^1,\sigma^2,\sigma^3\}$ denote the dual basis of $\mathfrak l^*$. Moreover, $\{v_1=e_2\wedge e_3,v_2=e_3\wedge e_1, v_3=e_1\wedge e_2\}$ is a basis of $\mathfrak l\wedge \mathfrak l$.
For $\epsilon\in C^2(\mathfrak l,\mathfrak l^*)$ we define $M_\epsilon\in Mat(3,\mathds R)$ by $(M_\epsilon)_{ij}=\epsilon(v_j)(e_i)$.

\begin{lemma}
Every element in $H^2(\mathds{R}^3,0,0)_0/G$ can be represented by exactly one $(0,\epsilon,0)$ where 
\[M_\epsilon\in\left\{\begin{pmatrix}1&0&0\\0&b&0\\0&0&-1-b\end{pmatrix},\begin{pmatrix}1&1&0\\0&1&0\\0&0&-2\end{pmatrix},\begin{pmatrix}1&z&0\\-z&1&0\\0&0&-2\end{pmatrix},\begin{pmatrix}0&1&0\\-1&0&0\\0&0&0\end{pmatrix},\begin{pmatrix}0&1&0\\0&0&1\\0&0&0\end{pmatrix}\,\Big|\, b\in[0,1], z>0\right\}.\]
\end{lemma}

\begin{proof}
We prove the lemma by determining $H^2(\mathds R^3,0)_0$ and describing the action of $G$ on that cohomology set. 

Since $\mathfrak a=\{0\}$ the set of cocycles $Z^2(\mathds R^3,0)_0$ consists of all triples $(0,\epsilon,0)$ satisfying condition (\ref{eq:lemma5}) and condition (a) on page \pageref{eq:aundb}.
Since $\mathfrak l$ is three-dimensional, condition (\ref{eq:lemma5}) is equivalent to
\begin{align}\label{eq:R3Leer1}
\epsilon(e_1,e_2)(e_3)+\epsilon(e_2,e_3)(e_1)+\epsilon(e_3,e_1)(e_2)=\tr(M_\epsilon)=0.
\end{align}
Moreover, condition (a) is satisfied if and only if $L=0$ is the only element in $\mathfrak l$, which satisfies
\begin{align}\label{eq:R3Leer2}
\epsilon(L,\tilde L)=0\in\mathfrak l^*
\end{align}
for all $\tilde L\in\mathfrak l$.
Equation (\ref{eq:R3Leer2}) holds for $L=\lambda_1e_1+\lambda_2e_2+\lambda_3e_3$ where $\lambda_1,\lambda_2,\lambda_3\in\mathds R$ if and only if $\epsilon(L,e_i)=0\in\mathfrak l^*$ for all $i=1,\dots,3$. In addition this condition is equivalent to
\begin{align*}
-\lambda_2\epsilon(v_3)+\lambda_3\epsilon(v_2)=0\in\mathfrak l^*,\quad
\lambda_1\epsilon(v_3)-\lambda_3\epsilon(v_1)=0\in\mathfrak l^*,\quad
-\lambda_1\epsilon(v_2)+\lambda_2\epsilon(v_1)=0\in\mathfrak l^*.
\end{align*}
This shows that condition (a) holds if and only if $M_\epsilon$ has at least $2$ linearly independent columns.

Two triples $(0,\epsilon_1,0)$ and $(0,\epsilon_2,0)$ are equivalent in $Z^2(\mathds R^3,0)_0$ if and only if they are equal.
The next step is to describe the action of $G$ with the help of $M_\epsilon$.
Here we mention that $G=\operatorname{Aut}(\mathfrak l)$. Thus we use the notation $S$ for elements in $G$ instead of $(S,U)$ in the following.
\begin{lemma}\label{lemma:R3Leer}
Two cohomology classes $[0,\epsilon,0]$ and $[0,\epsilon',0]$ in $H^2(\mathds R^3,0)_0$ lie in the same $G$-orbit if and only if there is a $c\neq 0\in\mathds R$ such that the matrices $M_{\epsilon'}$ and $cM_\epsilon$ are conjugate.
\end{lemma}

\begin{proof}
Assume $\epsilon,\epsilon'\in C^2(\mathfrak l,\mathfrak l^*)$.
Let $S:\mathfrak l\rightarrow\mathfrak l$ be a bijective linear map given by $S=(s_{ij})_{i,j=1,2,3}$ with respect to the basis $\{e_1,e_2,e_3\}$.
Then 
\begin{align*}
\epsilon(Se_i,Se_j)(Se_k)&=\epsilon(s_{1i}e_1+\dots +s_{3i}e_3,s_{1j}e_1+\dots+s_{3j}e_3)(Se_k)\\
&=\sum_{1\leq r<t\leq 3}(s_{ri}s_{tj}-s_{rj}s_{ti})\epsilon(e_r,e_t)(Se_k)\displaybreak\\
&=\sum_{u=1}^3\sum_{1\leq r<t\leq 3}s_{uk}(s_{ri}s_{tj}-s_{rj}s_{ti})\epsilon(e_r,e_t)(e_u)
\end{align*}
and using the Cramer's rule yields that $(0,\epsilon',0)=S^*(0,\epsilon,0)$  holds if and only if
\[M_{\epsilon'}=\det(S^T)S^TM_\epsilon {S^{-1}}^T.\]
Since the dimension of $\mathfrak l$ is odd we obtain our assertion.
\end{proof} 

Finally the observations from above and Lemma \ref{lemma:R3Leer} show that $H^2(\mathfrak l,\mathfrak a)_0/G$ can be described by all trace free $3\times 3$-matrices with an rang of at least two modulo conjugation and modulo multiplication with real numbers unequal to zero. This yields the system of representatives given in the lemma.
\end{proof}

\subsection{Case $\mathfrak l=\mathds R^n$, $\mathfrak a=\mathfrak n_4$}

Denote $\mathfrak a=\mathfrak n_4:=\{[a_4,a_1]=a_2,[a_4,a_2]=a_3\}$ with $\omega_\mathfrak a=\alpha^1\wedge\alpha^2+\alpha^3\wedge\alpha^4$, where $\{\alpha^1,\dots,\alpha^4\}$ denotes the dual basis of $\mathfrak a^*$.

\begin{lemma}
The set $Z^2(\mathds R^n,\mathfrak n_4,\omega_\mathfrak a)$ does not consist nilpotent cocycles.
\end{lemma}

\begin{proof}
At first we compute some useful properties of nilpotent derivtions of $\mathfrak n_4$. 
Nilpotent derivations $D$ map the center $\mathfrak z(\mathfrak n_4)=\Span\{a_3\}$ to zero, since it is one-dimensional.
Moreover, using \[Da_2=D[a_4,a_1]=[Da_4,a_1]+[a_4,Da_1]\in \Span\{a_2,a_3\}\] and the nilpotency of $D$ we obtain $Da_2\in\Span\{a_3\}$.
Thus $Da_1\in\Span\{a_2,a_3\}$ because of
\[0=[Da_1,a_2]+[a_1,Da_2]=[Da_1,a_2].\]
Hence the image of a nilpotent derivation $D$ lies in $\Span\{a_2,a_3\}$.

Now, assume $(\gamma,0,\xi)\in Z^2(\mathds R^n,\mathfrak n_4,\omega_\mathfrak{a})$ where $\xi(L)$ is nilpotent for all $L\in\mathds R^n$.
Then $\beta_0(a_3)=0$.
Using condition (\ref{eq:lemma2}) yields
\begin{align*}
0&=\beta(\xi(L)a_2,a_4)+\beta(a_2,\xi(L)a_4)-\gamma(L)[a_2,a_4]\\
&=-\omega_\mathfrak a(\xi(\cdot)\xi(L)a_2,a_4)-\omega_\mathfrak a(\xi(L)a_2,\xi(\cdot)a_4)-\omega_\mathfrak a(\xi(\cdot)a_2,\xi(L)a_4)-\omega_\mathfrak a(a_2,\xi(\cdot)\xi(L)a_4)+\gamma(L)a_3\\
&=\gamma(L)a_3
\end{align*}
for all $L\in\mathds R^n$. Thus $a_3\in\ker\gamma_0$.
Now, this shows that $(\gamma,0,\xi)$ is not balanced, since 
\[\mathfrak z(\mathfrak a)\cap\mathfrak a^\mathfrak l_\xi\cap\ker\beta_0\cap\ker\gamma_0=\Span\{a_3\}\] is degenerate and hence condition (b) is not satisfied.
\end{proof}

\subsection{Case $\mathfrak l=\mathds R$, $\mathfrak a=\mathds R^4$}

Let $\{e_1\}$ denote the standard basis of $\mathfrak l=\mathds R$ and $\{a_1,a_2,a_3,a_4\}$ the standard basis of $\mathfrak a=\mathds R^4$. Moreover, let $\{\sigma^1\}$ and $\{\alpha^1,\alpha^2,\alpha^3,\alpha^4\}$ denote the dual basis of $\mathfrak l^*$ and $\mathfrak a^*$ respectively. Let $\omega_{\mathfrak a}$ denote the symplectic form $\omega_\mathfrak a=\alpha^1\wedge\alpha^2+\alpha^3\wedge\alpha^4$ of $\mathds R^4.$

In this section we compute a system of representatives for $H^2(\mathds R,\mathds R^4,\omega_{\mathfrak a})_0/G$.
Therefor, we set
$\gamma^0:\mathfrak l\rightarrow \Hom(\mathfrak a,\mathfrak l^*)$ as
\begin{align}\label{eq:brauch1}
\gamma^0=\sigma^1\otimes\alpha^1\otimes\sigma^1.
\end{align}
Moreover, we define linear maps $\xi_1,\xi_\pm,\xi^\kappa:\mathfrak l\rightarrow\Hom(\mathfrak a)$ by 
\begin{align}
\xi_1&=\sigma^1\otimes\alpha^2\otimes a_1,\label{eq:brauch2}\\
\xi_\pm&=\sigma^1\otimes\alpha^3\otimes a_1\pm\sigma^1\otimes\alpha^2\otimes a_3\text{ und}\label{eq:brauch3}\\ 
\xi^\kappa&=\sigma^1\otimes\big(\alpha^4\otimes a_1+\alpha^2\otimes a_3+\alpha^3\otimes a_4+\alpha^3\otimes\kappa a_1\big)\label{eq:brauch4}
\end{align}
for $\kappa\in\mathds R$.

\begin{lemma}
Every element of $H^2(\mathds R,\mathds R^4,\omega_\mathfrak a)_0/G$ can be represented by the cohomology class of exactly one of the following cocycles:
\[(\gamma^0,0,\xi_1),(\gamma^0,0,\xi_\pm),(\gamma^0,0,\xi^\kappa),\quad\kappa\in\mathds R.\]
\end{lemma}

\begin{proof}
At first it is not hard to see that the given cocycles are nilpotent ones.
Now, we show that the cohomology class of an arbitrary nilpotent cocylce lies in the $G$-orbit of a cohomology class given in the lemma.
So, let $(\gamma,\epsilon,\xi)$ be a nilpotent cocycle.
Since $\mathfrak l$ is one-dimensional, we have $\epsilon=0$ and $\alpha=0$.
It holds $\xi(e_1)\neq 0$: Conversely, for $\xi=0$ we get $\beta=0$ and thus, condition (b) implies that
$\mathfrak z(\mathfrak a)\cap\mathfrak a^\mathfrak l_\xi\cap\ker\gamma_0\cap\ker\beta_0=\ker\gamma_0$ is non-degenerate. Hence $\gamma=0$.
But for $\xi=0$ and $\gamma=0$ it is easy to see that condition (a) is not satisfied.\\
Moreover, $\ker\xi(e_1)\cap \im\xi(e_1)\neq \{0\}$ is isotropic. 
To see this, we choose vectors $v_1,v_2\in\ker\xi(e_1)\cap \im\xi(e_1)$. Let $\overline v_i$ denote inverse images of $v_i$ under $\xi(e_1)$ for $i=1,2$.
Then 
\begin{align*}
\omega_\mathfrak a(v_1,v_2)&=\omega_\mathfrak a(\xi(e_1)\overline v_1,\xi(e_1)\overline v_2)=-\frac{1}{2}\beta(\xi(e_1)\overline v_1,\overline v_2)(e_1)-\frac{1}{2}\beta(\overline v_1,\xi(e_1)\overline v_2)(e_1)=0,
\end{align*}
because of condition (\ref{eq:lemma2}).
Especially 
\begin{align}\label{eq:isotr}
\im\xi(e_1)\neq\ker\xi(e_1),
\end{align}
since $\mathfrak a^\mathfrak l_\xi=\ker\xi(e_1)\subset\ker\beta_0$ and thus 
\[\mathfrak z(\mathfrak a)\cap\mathfrak a^\mathfrak l_\xi\cap\ker\gamma_0\cap\ker\beta_0=\ker\xi(e_1)\cap\ker\gamma(e_1)\neq \{0\}\]
is isotropic otherwise.

Now, let $\mathfrak a^\mathfrak l_\xi$ be three-dimensional. This means that the image of $\xi(e_1)$ is one-dimensional.
At first we show that $\beta=0$.
Therefor, we note that the at least $2$-dimensional subspace $\im(\xi(e_1))^\bot\cap\ker\xi(e_1)$ is contained in $\ker\beta_0$.
Since, especially, $\im\xi(e_1)\subset \im(\xi(e_1))^\bot\cap\ker\xi(e_1)$ lies in the kernel of $\beta_0$, the subspace $\ker\beta_0$ contains a $2$-dimensional isotropic subspace.
Because of the dimensions the subspace  $\ker\xi(e_1)\cap\ker\gamma(e_1)\cap\ker\beta_0$ is not trivial. This space is moreover non-degenerate because of condition (a) and hence $\ker\beta_0$ is not $2$-dimensional, thus $\beta= 0$. Especially, using condition (b) we see that $\gamma\neq 0$ and 
$\ker\xi(e_1)\cap\ker\gamma(e_1)$ is a $2-$dimensional non-degenerate subspace of $\mathfrak a=\mathds R^4$.
So, we can choose $v_3,v_4\in \ker\xi(e_1)\cap\ker\gamma(e_1)$ such that $\omega_{\mathfrak a}(v_3,v_4)=1$.
Since $\beta=0$ the derivation $\xi(e_1)$ is skewsymmetric with respect to $\omega_\mathfrak a$.
Thus the image of $\xi(e_1)$ is orthogonal to $\Span\{v_3,v_4\}\subset\ker\xi(e_1)$ and, especially, it holds $\im\xi(e_1)\nsubseteq\ker\gamma(e_1)$ because of (b).
We choose $0\neq v_1\in\im\xi(e_1)$.
Then $v_1\notin\ker\gamma(e_1)$.
Moreover, we choose $v_2\in\ker\gamma(e_1)$ such that $v_2$ is orthogonal to the non-degenerate subspace $\Span\{v_3,v_4\}=\ker\xi(e_1)\cap\ker\gamma(e_1)$ and $\omega_\mathfrak a(v_1,v_2)=1$ holds.
We have $\xi(e_1)v_2=\lambda v_1$ for some $\lambda\neq 0$.
Now, $\gamma$, $\epsilon$ and $\xi$ are given by 
\[ \gamma(\tilde L)=x\nu^1\otimes Z_1,\quad \epsilon=0,\quad \xi(\tilde L)=\nu^2\otimes v_1\]
with respect to the basis $\{\tilde L:=e_1/\lambda\}$ of $\mathfrak l$ and the Darboux basis $\{v_1,v_2,v_3,v_4\}$ of $\mathfrak a$ for some $x\neq 0$.
Here $\{\nu^1,\nu^2,\nu^3,\nu^4\}$ denotes the dual basis of $\mathfrak a^*$ for $\{v_1,v_2,v_3,v_4\}$ and $\{Z_1\}$ the dual basis for $\{\tilde L\}$.
Thus, for every $(\gamma,\epsilon,\xi)\in Z^2(\mathfrak l,\mathfrak a,\omega_\mathfrak a)_0$ with three-dimensional $\mathfrak a^\mathfrak l_\xi$ there is an $(S,U)\in G$ such that $(S,U)^*(\gamma,\epsilon,\xi)=(x\gamma^0,0,\xi_1)$ for some $x\neq 0$.
Moreover, $[x\gamma^0,0,\xi_1]$ and $[\gamma^0,0,\xi_1]$ lie in the same $G$-orbit. This can be seen directly by using $(S_2,U_2)=(x^{-2/5},\diag(x^{1/5},x^{-1/5},1,1))$.

Now, let $(\gamma,0,\xi)$ be a nilpotent cocylce with two-dimensional $\mathfrak a^\mathfrak l_\xi$.
Condition (\ref{eq:isotr}) yields $\xi(e_1)^2\neq 0$. Moreover, the subspaces $\mathfrak a^\mathfrak l_\xi=\ker\xi(e_1)$ and $\im\xi(e_1)$ are isotropic because of condition (\ref{eq:lemma2}):
Therefor, assume $v_1\in\ker\xi(e_1)\cap \im\xi(e_1)=\im(\xi(e_1)^2)$, $v_4\in\ker\xi(e_1)$ and choose an inverse image $v_2\in \im\xi(e_1)$ of $v_1$ under $\xi(e_1)$. 
Then
\[\omega_\mathfrak{a}(v_1,v_4)=\omega_\mathfrak{a}(v_1,v_4)+2\omega_\mathfrak{a}(v_2,\xi(e_1)v_4)+\omega_\mathfrak{a}(v_3,\xi(e_1)^2v_4)=0\]
and
\[-\omega_\mathfrak{a}(v_1,v_2)=\omega_\mathfrak{a}(v_1,v_2)+2\omega_\mathfrak{a}(v_2,v_1)+\omega_\mathfrak{a}(v_3,\xi(e_1)v_1)=0,\]
where $v_3$ denotes an inverse image of $v_2$ under $\xi(e_1)$.
Especially, we have
\begin{align}\label{eq:ausg}
\mathfrak a_\xi^{\mathfrak l}\cap\ker\gamma_0\cap\ker\beta_0=0.
\end{align}
So, let us fix a $v_1\neq 0$ in the image of $\xi(e_1)^2$.
Then $v_1\notin \ker\gamma_0$ because of condition (\ref{eq:ausg}) and $\im(\xi(e_1)^2)=\ker\xi(e_1)\cap \im\xi(e_1)\subset \ker\beta_0$.
Moreover, we fix a $v_2\in\ker\gamma(e_1)\cap\im\xi(e_1)$ with $\xi(e_1)v_2=v_1$. 
Since $\omega_\mathfrak a$ is non-degenerate, the three-dimensional subspace of $\mathfrak a$ spanned by $\ker\xi(e_1)$ and $\im\xi(e_1)$ is not isotropic.
Thus there is a vector in the kernel of $\xi(e_1)$ which is not orthogonal to $v_2$.
Then there is also a $v_3\in\ker\gamma(e_1)$ with $\xi(e_1)v_3=v_2$ and $\omega_{\mathfrak a}(v_2,v_3)=0$. 
Finally, we choose $v_4\in \Span\{v_1,v_3\}^\bot\cap\ker\xi(e_1)$ such that $\omega_{\mathfrak a}(v_2,v_4)=1$.
Now, $\gamma$, $\epsilon$ and $\xi$ are given by
\[\gamma(e_1)=x\nu^1\otimes \sigma^1+\tilde y\nu^4\otimes Z_1,\quad \epsilon=0,\quad \xi(e_1)=\pm\nu^2\otimes\tilde v_3+\nu^3\otimes\tilde v_1\]
with respect to the Darboux basis 
\[\tilde v_1:=\frac{\sign(b)}{\sqrt{|b|}}v_1,\quad\tilde v_2:=\frac{1}{\sqrt{|b|}}v_3,\quad\tilde v_3:=\frac{\sign(b)}{\sqrt{|b|}}v_2,\quad\tilde v_4:=\frac{\sqrt{|b|}}{\sign(b)}v_4\}\] of $\mathfrak a$,
where $b:=\omega_\mathfrak{a}(v_1,v_3)\neq 0$, and with respect to the basis $\{e_1\}$ of $\mathfrak l=\mathds R$ for a suitable $x\neq 0$ and $\tilde y\in\mathds R$.
Here $\{\nu^1,\dots,\nu^4\}$ denotes the dual basis of $\{\tilde v_1,\dots,\tilde v_4\}$.

As in the previous case, it is possible to rescale the bases.
We set $S=1/\sqrt[3]{x}$ and $U=\diag(\sqrt[3]{x},1/\sqrt[3]{x},1,1)$. Then 
\[(S,U)^*(\sigma^1\otimes(x\alpha^1+\tilde y\alpha^4)\otimes\sigma^1,0,\xi_\pm)=(\sigma^1\otimes(\alpha^1+y\alpha^4)\otimes\sigma^1,0,\xi_\pm)\]
for some $y\in\mathds R$.
Finally  
$[\sigma^1\otimes(\alpha^1+y\alpha^4)\otimes\sigma^1,0,\xi_\pm]$ and $[\gamma^0,0,\xi_\pm]$ lie in the same $G$-orbit.
Therefor, we compute on the one hand that
\[\big(\gamma^0,0,\xi_\pm\big)\big(\sigma^1\otimes(\pm\frac{1}{3}ya_2\pm\frac{2}{3}y^2a_3)\big)=\big(\sigma^1\otimes(\alpha^1\pm\frac{2}{3}y^2\alpha^2\mp\frac{2}{3}y\alpha^3+\frac{1}{3}y\alpha^4)\otimes\sigma^1,0,\xi_\pm\big):\]
 Assume $\tau=\sigma^1\otimes(\tau_2a_2+\tau_3a_3)$ with $\tau_2=\pm\frac{1}{3}y$ and $\tau_3=\pm\frac{2}{3}y^2$. Then 
\begin{align*}
\tau^*=(-\tau_2\alpha^1+\tau_3\alpha^4)\otimes\sigma^1.
\end{align*}
In addition, we have
\begin{align*}
\tau^*(\xi_\pm(L)A)&=\sigma^1(L)(-\tau_2\alpha^1+\tau_3\alpha^4)(\alpha^3(A)a_1\pm\alpha^2(A)a_3)\sigma^1\\
&=-\sigma^1(L)(\tau_2\alpha^3(A))\sigma^1=\big((-\sigma^1\otimes\tau_2\alpha^3\otimes\sigma^1)(L)\big)(A)
\end{align*}
and
\begin{align*}
((\beta_0\circ\tau)(L))(A)&=\beta(\tau L,A)=\sigma^1(L)(\tau_2\alpha^3\mp\tau_2\alpha^4-\tau_3\alpha^2)(A)\sigma^1\\
&=\big((\sigma^1\otimes(\tau_2\alpha^3\mp\tau_2\alpha^4-\tau_3\alpha^2)\otimes\sigma^1)(L)\big)(A).
\end{align*}
Thus
\begin{align*}
\gamma^0+\tilde\tau^*\circ\xi_\pm-\tilde\tau^*\circ ad_\mathfrak{a}\circ \tau-\beta_0\circ\tau&=\sigma^1\otimes(\alpha^1-2\tau_2\alpha^3+\tau_3\alpha^2\pm\tau_2\alpha^4)\otimes\sigma^1\\
&=\sigma^1\otimes(\alpha^1\pm\frac{2}{3}y^2\alpha^2\mp\frac{2}{3}y\alpha^3+\frac{1}{3}y\alpha^4)\otimes\sigma^1.
\end{align*}
On the other side, for $S=1$ and 
\[U=\begin{pmatrix}1&0&\pm\frac{2}{3}y&\frac{2}{3}y\\0&1&0&0\\0&\frac{2}{3}y&1&0\\0&\mp\frac{2}{3}y&0&1\end{pmatrix}\] we have
\[(S,U)^*(\sigma^1\otimes(\alpha^1+y\alpha^4)\otimes\sigma^1,0,\xi_\pm)=(\sigma^1\otimes(\alpha^1\pm\frac{2}{3}y^2\alpha^2\mp\frac{2}{3}y\alpha^3+\frac{1}{3}y\alpha^4)\otimes\sigma^1,0,\xi_\pm).\]
Hence 
$[\sigma^1\otimes(\alpha^1+y\alpha^4)\otimes\sigma^1,0,\xi_\pm]$ and $[\gamma^0,0,\xi_\pm]$ lie in the same $G$-orbit.

As the last case, let $\mathfrak a^\mathfrak l_\xi$ be one-dimensional. Then $\xi(e_1)^3\neq 0$. We choose $v_1,\dots,v_4\in\mathfrak a$ with $\xi(e_1)v_i=v_{i-1}$. Since 
\[\mathfrak z(\mathfrak a)\cap\mathfrak a^\mathfrak l_\xi\cap\ker\gamma_0\cap\ker\beta_0=\ker\gamma(e_1)\cap \Span\{v_1\}\] 
is non-degenerate we obtain $v_1\notin\ker\gamma(e_1)$.
Because of condition (\ref{eq:lemma2}) it holds
\begin{align}
\omega_\mathfrak a(v_1,v_2)&=-\beta(\xi(e_1)v_2,v_3)(e_1)-\beta(v_2,\xi(e_1)v_3)(e_1)=0,\label{eq:bre1}\\
\omega_\mathfrak a(v_1,v_3)&=-\frac{1}{2}\beta(\xi(e_1)v_2,v_4)(e_1)-\frac{1}{2}\beta(v_2,\xi(e_1)v_4)(e_1)=0\label{eq:bre2}.
\end{align}
Hence $\im(\xi(e_1))^3=\ker\xi(e_1)\cap\im(e_1)$ is orthogonal to the image of $\xi(L)$.
Moreover, using condition (\ref{eq:lemma2}) yields
\begin{align}\label{eq:7}
 \omega_\mathfrak a(v_2,v_3)=-\omega_\mathfrak a(v_1,v_4)\neq 0.
\end{align}
Especially, we can assume w.l.o.g that $\omega_{\mathfrak a}(v_1,v_4)=\pm 1$.
Now, we choose $\{\tilde v_1:=v_1,\tilde v_2:=\pm v_4-\frac{f}{2} v_2,\tilde v_3:=v_3\mp\frac{f}{2} v_1,\tilde v_4:=\pm v_2-ev_1\}$
as a Darboux basis of $\mathfrak{a}$ and $\{\tilde L=\pm e_1\}$ as a basis of $\mathfrak{l}$, where $e=\omega_\mathfrak a(v_2,v_4)$ and $f=\omega_\mathfrak a(v_3,v_4)$.
Then 
\[\xi(\tilde L)=\nu^2\otimes\tilde v_3+\nu^3\otimes(\tilde\kappa\tilde v_1+\tilde v_4)+\nu^4\otimes\tilde v_1\quad\text{und}\quad\gamma(\tilde L,\tilde v_1,\tilde L)\neq 0\]
%\[\xi(\tilde L)=\begin{pmatrix}&&\tilde\kappa&1\\&&&\\&1&&\\&&1&\end{pmatrix}\]
for some $\tilde\kappa\in\mathds R$. Here $\{\nu^1,\dots,\nu^4\}$ denotes the dual basis of $\mathfrak a^*$ for $\{\tilde v_1,\dots,\tilde v_4\}$.
This shows that $[\gamma,0,\xi]$ lies in the $G$-orbit of $[x\gamma^0,0,\xi^{\tilde\kappa}]$, since the equivalence class of $(x\gamma^0,0,\xi^{\tilde\kappa})$ consists of all triples $(\gamma,0,\xi^{\tilde\kappa})\in Z^2(\mathfrak l,\mathfrak a,\omega_\mathfrak a)_\sharp$ which satisfy $\gamma(e_1)(a_1)=x\sigma^1$:
Therefor, assume $\tau=\sigma^1\otimes(\tau_1a_1+\cdots+\tau_4a_4)$. Then
\begin{align}
\tau^*&=(\tau_1\alpha^2-\tau_2\alpha^1+\tau_3\alpha^4-\tau_4\alpha^3)\otimes\sigma^1
\end{align}
and
\begin{align}
\tau^*(\xi^{\tilde\kappa}(L)A)&=\sigma^1(L)(\tau_1\alpha^2-\tau_2\alpha^1+\tau_3\alpha^4-\tau_4\alpha^3)(\alpha^2(A)a_3+\alpha^3(A)a_4+\alpha^3(A)\tilde\kappa a_1+\alpha^4(A)a_1)\sigma^1\label{eq:5}\nonumber\\
&=\sigma^1(L)(-\tau_4\alpha^2+\tau_3\alpha^3-\tau_2(\tilde\kappa\alpha^3+\alpha^4))(A)\sigma^1\\
&=\big((\sigma^1\otimes(-\tau_4\alpha^2+(\tau_3-\tau_2\tilde\kappa)\alpha^3-\tau_2\alpha^4)\otimes\sigma^1)(L)\big)(A)\nonumber
\end{align}
Moreover, we have
\begin{align}
((\beta_0\circ\tau)(L))(A)&=\beta(\tau L,A)=\sigma^1(L)(\tilde\kappa\tau_2\alpha^3-\tilde\kappa\tau_3\alpha^2)(A)\sigma^1\label{eq:6}\\
&=\big((\sigma^1\otimes(\tilde\kappa\tau_2\alpha^3-\tilde\kappa\tau_3\alpha^2)\otimes\sigma^1)(L)\big)(A)\nonumber
\end{align}
and
\begin{align*}
x\gamma^0+\tilde\tau^*\circ\xi^{\tilde\kappa}-\tilde\tau^*\circ ad_\mathfrak{a}\circ \tau-\beta_0\circ\tau=\\
\sigma^1\otimes(x\alpha^1+(-\tau_4+\tau_3\tilde\kappa)\alpha^2+(\tau_3-2\tau_2\tilde\kappa)\alpha^3-\tau_2\alpha^4)\otimes\sigma^1.
\end{align*}
Finally 
\[(k^2,\diag(k^{-3},k^3,k,k^{-1}))^*(x\gamma^0,0,\xi^{\tilde\kappa})=(\gamma^0,0,\xi^\kappa)\]
for $\kappa=\tilde\kappa k^{-2}$ and $k:=\sqrt[7]{x}$.\\
\quad\\
In the last step, we want to show that no cohomology classes of two distinct cocycles given in the set of the lemma have the same $G$-orbit.
Since the equivalence of two cocycles $(\gamma^0,0,\xi),(\gamma^0,0,\tilde\xi)\in Z^2(\mathds R,\mathds R^4,\omega_\mathfrak a)$ implie the equality $\xi=\tilde\xi$, it suffices to show that on the one hand $[\gamma^0,0,\xi_+]$ and $[\gamma^0,0,\xi_-]$ lie in distinct $G$-orbits and on the other hand that $[\gamma^0,0,\xi^\kappa]$ and $[\gamma^0,0,\xi^{\tilde\kappa}]$ lie in distinct $G$-orbits for $\kappa\neq\tilde\kappa$.

For $(\gamma,0,\xi)\in Z^2(\mathfrak l,\mathfrak a,\omega_\mathfrak{a})_0$ with two-dimensional $\mathfrak a^\mathfrak l_\xi$ we choose a vector $v\in\mathfrak a$ with $\xi^2(e_1)v\neq 0$ and a $L\neq 0$ in $\mathfrak l$ and consider the sign of $\omega_{\mathfrak a}(\xi^2(L)v,v)\neq 0$.
This sign is independent from the choice of $v$, since the image of $\xi(e_1)$ is orthogonal to $\ker\xi(e_1)\cap \im\xi(e_1)$ and hence $\omega_\mathfrak{a}(\xi^2(L)\tilde v,\tilde v)$ is a positive multiple of $\omega_\mathfrak{a}(\xi^2(L)v,v)$ for every $\tilde v\in \mathfrak a$ satisfying $\xi^2(e_1)\tilde v\neq 0$.
It is clear that the sign is independent from the choice of $L\neq 0$ and invariant under equivalence. Hence the sign does only depend on the structure of the Lie algebra and the symplectic form and is a well-defined invariant of the $G$-orbits.
Since $\omega_\mathfrak a(\xi_\pm^2(e_1)a_2,a_2)=\pm 1$, the cohomology classes $[\gamma^0,0,\xi_+]$ und $[\gamma^0,0,\xi_-]$ lie in different $G$-orbits.

Now, assume $(\gamma,0,\xi)\in Z^2(\mathfrak l,\mathfrak a,\omega_\mathfrak a)_0$ and $\xi^3(e_1)\neq 0$. We choose a $L\neq 0\in\mathds R$ and a $v\in\mathfrak a$ such that $\xi^3(e_1)v\neq 0$ and consider
\begin{align}\label{eq:KappaNeqIsom}
\frac{\omega_\mathfrak a(\xi^2(L)v,v)\gamma(L,(\xi(L))^3(v),L)^{2/7}} {\omega_\mathfrak a(\xi^3(L)v,v)^{8/7}}.
\end{align} As before, we can show that the value of (\ref{eq:KappaNeqIsom}) is independent from the choice $v$.
Moreover, it is easy to prove that the value is also independent of the choice of $L\neq 0$.
Thus the value of (\ref{eq:KappaNeqIsom}) is only given by the structure of the Lie algebra and the symplectic form (hence, especially, invariant under the $G$-action).
Furthermore, the term (\ref{eq:KappaNeqIsom}) is invariant under equivalence, because elements in the image of $\xi(e_1)^3$ lie in the kernel of $\beta_0$ (compare equations (\ref{eq:bre1}) and (\ref{eq:bre2})).
Since the value of the term (\ref{eq:KappaNeqIsom}) equals $\kappa$ for $(\gamma^0,0,\xi^\kappa)$, we obtain that $\kappa$ is an invariant of the  $G$-orbits and hence the set given in the lemma forms a system of representatives.
\end{proof}

\subsection{Case $\mathfrak l=\mathds R$, $\mathfrak a=\mathfrak h_3\oplus\mathds R$}

As before, let $\{e_1\}$ denote the standard basis of $\mathds R$ and $\{\sigma^1\}$ its dual basis. Moreover, let 
$\{a_1,\dots,a_4\}$ a basis of $\mathfrak h_3\oplus\mathds R$ satisfying $[a_1,a_2]=a_3$ and $\omega_\mathfrak a=\alpha^1\wedge\alpha^4+\alpha^2\wedge\alpha^3$, where $\{\alpha^1,\dots,\alpha^4\}$ denotes the corresponding dual basis.

In this section we determine $H^2(\mathds R,\mathfrak h_3\oplus\mathds R,\omega_\mathfrak a)_0/G$.
Basically, we argue as in the case $\mathfrak l=\mathds R$ and $\mathfrak a=\mathds R^4$ excluding that the choice of the basis of $\mathfrak h_3\oplus\mathds R$ is much more limited. So, the basis has to preserve the symplectic form and the fixed structure of the commutator.
Thus it is not easy to find a basis directly such that a triple $(\gamma,\epsilon,\xi)$ has a nice form with respect to that basis.
On the other hand nilpotent derivations of $\mathfrak h_3\oplus\mathds R$ are also of special form.

For $\kappa_1,\kappa_2\in\{\pm 1\}$, $b,c,i,j,m,n,l\in\mathds R$ we define
\[\xi_{\kappa_1}^{\kappa_2}(b,c,i,j,m,n,l):=
\begin{pmatrix}
 \kappa_1bc  & \kappa_2b^2 & 0& 0 \\
-\kappa_2c^2 & -\kappa_1bc & 0& 0 \\
      i      &     j       & 0& l \\
      m      &     n       & 0& 0
\end{pmatrix}.\]

\begin{lemma}
A linear map $\xi:\mathds R\rightarrow \mathfrak{der}(\mathfrak a)$ satisfies the condition (\ref{eq:Nilpotenz}) in Lemma \ref{Lemma:Nilpotenz}, if and only if $\xi(e_1)=\xi_{\kappa_1}^{\kappa_2}(b,c,i,j,m,n,l)$ for $\kappa_1,\kappa_2\in\{\pm 1\}$, $b,c,i,j,m,n,l\in\mathds R$.
\end{lemma}

\begin{proof}
Because of $\mathfrak l=\mathds R$ the condition (\ref{eq:Nilpotenz}) is satisfied if and only if $\xi(e_1)$ is a nilpotent derivation.
So, let $\xi(e_1)$ be a nilpotent derivation. 
Since derivations map the derived Lie algebra and the center to itself, the subspaces $\Span\{e_3\}$ and $\Span\{e_3,e_4\}$ are invariant under $\xi(e_1)$.
All eigenvalues and the trace vanishes for a nilpotent $\xi(e_1)$.
Thus
\[\xi(e_1)=\begin{pmatrix}a&d&0&0\\e&-a&0&0\\ *&*&0&* \\ *&*&0&0\end{pmatrix},\] where $a^2=-de\geq 0$.

Conversely, direct computations show that $\xi(e_1)=\xi_{\kappa_1}^{\kappa_2}(b,c,i,j,m,n,l)$ is a nilpotent derivation for $\kappa_1,\kappa_2\in\{\pm 1\}$, $b,c,i,j,m,n,l\in\mathds R$.
\end{proof}

Now, we can describe when a triple $(\gamma,\epsilon,\xi)$ is a nilpotent cocycle.

\begin{lemma}\label{lemma:Z2RH3xR0}
A triple $(\gamma,\epsilon,\xi)$ is a nilpotent cocycle if and only if $\gamma=\sigma^1\otimes(x_1\alpha^1+x_2\alpha^2+x_3\alpha^3+x_4\alpha^4)\otimes\sigma^1$, $\epsilon=0$ and $\xi(e_1)=\xi_{\kappa_1}^{\kappa_2}(b,c,i,j,m,n,l)$, where the parameters satisfy one of the following conditions:
\begin{itemize}
\item $l\neq 0$, $c=0$, $x_3=m(l+\kappa_2b^2)\neq 0$,
\item $l=0$, $(b,c)\neq (0,0)$, $x_3=-\kappa_1bc(i+n)+\kappa_2b^2m+\kappa_2c^2j$, $\kappa_1bx_3+\kappa_2cx_4\neq 0$.
\end{itemize}
\end{lemma}

\begin{proof}
At first, a triple $(\gamma,\epsilon,\xi)$ where $\xi(e_1)=\xi_{\kappa_1}^{\kappa_2}(b,c,i,j,m,n,l)$ is a cocycle if and only if equation (\ref{eq:lemma2}) holds, since the remaining conditions (\ref{eq:lemma1}) and (\ref{eq:lemma3}) - (\ref{eq:lemma5}) are trivial for $\mathfrak l=\mathds R$.
Especially, $\epsilon=0$.

Since $\xi^2(e_1)\mathfrak a\subset\Span\{a_3,a_4\}$ is orthogonal to $a_3\in\ker\xi(e_1)\cap\mathfrak z(\mathfrak a)$,
condition (\ref{eq:lemma2}) is equivalent to
\begin{align}
\beta(\xi(e_1)a_1,a_2)+\beta(a_1,\xi(e_1)a_2)=\gamma(e_1)a_3,\label{eq:oben1}\\
\beta(\xi(e_1)A,a_4)+\beta(A,\xi(e_1)a_4)=0\label{eq:oben2}
\end{align}
for all $A\in\Span\{a_1,a_2\}$.
Since, moreover, $\xi^2(e_1)A$ is othogonal to $a_4$ for all $A\in\Span\{a_1,a_2\}$ and $\xi^2(e_1)a_4=0$ holds, the equation (\ref{eq:oben2}) is again equivalent to 
\begin{align}
2\omega_{\mathfrak a}(\xi(e_1)A,\xi(e_1)a_4)=0\label{eq:oben2b}
\end{align} for all $A\in\Span\{a_1,a_2\}$.
Finally we obtain 
\[x_3=-\kappa_1bc(i+n)+m(l+\kappa_2b^2)+\kappa_2c^2j\]
from equation (\ref{eq:oben1}) and $cl=0$ from equation (\ref{eq:oben2b}).

For a $(\gamma,\epsilon,\xi)\in Z^2(\mathds R,\mathfrak h_3\oplus\mathds R,\omega_\mathfrak a)$ satisfying $\xi(e_1)=\xi_{\kappa_1}^{\kappa_2}(b,c,i,j,m,n,l)$ we have 
\begin{align}
\mathfrak z(\mathfrak a)\cap\mathfrak a^\mathfrak l_\xi=
\begin{cases}\Span\{a_3\},&\text{ für } l\neq 0\\
             \Span\{a_3,a_4\},&\text{ für } l=0.
\end{cases}
\end{align}

Assume $l\neq 0$. Then $c=0$ and thus $a_3\in\ker\beta_0$. Hence condition (b) on page \pageref{eq:aundb} is satisfied if and only if the first equation of the lemma holds.
Let us now consider the case $l=0$.
Because of $\beta_0(a_3)=(\kappa_2c^2\alpha^1+\kappa_1bc\alpha^2)\otimes\sigma^1$ and $\beta_0(a_4)=(-\kappa_1bc\alpha^1-\kappa_2b^2\alpha^2)\otimes\sigma^1$ for $l=0$ we have
\begin{align}
\Span\{a_3,a_4\}\cap\ker\beta_0=\begin{cases}
\Span\{a_3,a_4\},\quad&\text{for }l=0, b=c=0,\\
\Span\{\kappa_1bx_3+\kappa_2cx_4\},\quad&\text{für }l=0, (b,c)\neq(0,0).
\end{cases}
\end{align}
Thus $\mathfrak z(\mathfrak a)\cap\mathfrak a^\mathfrak l_\xi\cap\ker\gamma_0\cap\ker\beta_0$ is non-degenerate if and only if the second condition of the lemma holds.
Finally condition (b) (see page \pageref{eq:aundb}) implies condition (a), because $L=0$ follows from (a.1). So, we have proved the lemma.
\end{proof}

It is easy to see that the linear map
\[U(x,p,q,k):\mathfrak a\rightarrow \mathfrak a,\quad
U(x,p,q,k)=\begin{pmatrix}1&x&0&0\\0&1&0&0\\p&q&1&-x\\k&p+kx&0&1\end{pmatrix}\]
is an isomorphism of $(\mathfrak a,\omega_\mathfrak a)$ for all $x,k,p,q\in\mathds R$.
\\\quad\\
Let $(\gamma,0,\xi)$ be a nilpotent cocycle with $\xi(e_1)=0$ on $\mathfrak a/\mathfrak z(\mathfrak a)$. This property is invariant under equivalence and $G$-action. Moreover, we have $\xi(e_1)\neq 0$ on $\mathfrak z(\mathfrak a)$ since the cocycle is balanced.

For every $x\in\mathds R$ and $f\neq 0$, the set $\{f^{-2}a_1,fa_2+xa_1,f^{-1}a_3,f^2a_4-fxa_3\}$ forms a basis of $\mathfrak a$ such that the commutator and symplectic form is given by $[f^{-2}a_1,fa_2+xa_1]=f^{-1}a_3$ and $\omega_\mathfrak a=\nu^1\wedge\nu^4+\nu^2\wedge\nu^3$,
where $\{\nu^1,\dots,\nu^4\}$ denotes the corresponding dual basis.
Thus we can assume that $\xi(e_1)$ is given by 
\[\begin{pmatrix}
0&0&0&0\\
0&0&0&0\\
*&*&0&m\\
m&0&0&0
\end{pmatrix}\]
for a suitable basis and a suitable $m\neq 0$.

This motivates to define 
\begin{align}\label{eq:brauch6}
\xi_{i,j}^1=\sigma^1\otimes\alpha^1\otimes(ia_3+a_4)+\sigma^1\otimes\alpha^2\otimes ja_3+\sigma^1\otimes\alpha^4\otimes a_3
\end{align}
for $i,j\in\mathds{R}$.
If $i=j=0$, we simply write $\xi^1$ instead of $\xi^1_{0,0}$.

We obtain, that a given nilpotent cocycle with $\xi(e_1)=0$ on $\mathfrak a/\mathfrak z(\mathfrak a)$ lies in the $G$-orbit of
\[ [\sigma^1\otimes(x_1\alpha^1+x_2\alpha^2+\alpha^3+x_4\alpha^4)\otimes\sigma^1,0,\xi_{i,j}^1] \]
for an $x_1,x_2,x_4,i,j\in\mathds R$.

\begin{lemma}\label{lm:66a}
The cohomology classes of all nilpotent cocycles $(\gamma,0,\xi)$ satisfying $\xi(e_1)=0$ on $\mathfrak a/\mathfrak z(\mathfrak a)$ lie in the $G$-orbit of $[\sigma^1\otimes\alpha^3\otimes\sigma^1,0,\xi^1]$.
\end{lemma}

\begin{proof}
At first, we have directly from Lemma \ref{lemma:Z2RH3xR0} that $(\sigma^1\otimes\alpha^3\otimes\sigma^1,0,\xi^1)$ is a nilpotent cocycle.

Before giving an isomorphism of pairs which maps the cohomology class of $[\sigma^1\otimes(x_1\alpha^1+x_2\alpha^2+\alpha^3+x_4\alpha^4)\otimes\sigma^1,0,\xi_{i,j}^1]$ to $[\sigma^1\otimes\alpha^3\otimes\sigma^1,0,\xi^1]$, we determine the equivalence class of $(\sigma^1\otimes\alpha^3\otimes\sigma^1,0,\xi^1)$.
For $\tau=\sigma^1\otimes(\tau_1a_1+\dots+\tau_4a_4)$ it holds 
\begin{align}\label{eq:aequivH3xR_1}
\tau^*=(\tau_1\alpha^4+\tau_2\alpha^3-\tau_3\alpha^2-\tau_4\alpha^1)\otimes\sigma^1.
\end{align}
Moreover, we have
\begin{align}
(\ad_{\mathfrak a}\circ\tau)(L)&=(\sigma^1\otimes(\tau_1\alpha^2-\tau_2\alpha^1)\otimes a_3)(L) \label{eq:aequivH3xR_2}
\end{align}
and  
\begin{align}
\tau^*([\tau L,A]_{\mathfrak a})&=((\sigma^1\otimes(\tau_1\tau_2\alpha^2-\tau_2^2\alpha^1)\otimes\sigma^1)(L))(A). \label{eq:aequivH3xR_3}\\
\end{align}
Furthermore,
\begin{align}
\tau^*(\xi(L)A)&=\sigma^1(L)\big((\tau_1\alpha^4+\tau_2\alpha^3-\tau_3\alpha^2-\tau_4\alpha^1)\otimes\sigma^1\big)(\alpha^1(A)a_4+\alpha^4(A)a_3) \label{eq:aequivH3xR_4}\\
&=\sigma^1(L)(\tau_1\alpha^1(A)+\tau_2\alpha^4(A))\sigma^1\nonumber\\
&=\big((\sigma^1\otimes(\tau_1\alpha^1+\tau_2\alpha^4)\otimes\sigma^1)(L)\big)(A)\nonumber
\end{align}
and using $\beta=-(\alpha^2\wedge\alpha^4)\otimes\sigma^1$ yields
\begin{align}\label{eq:aequivH3xR_5}
\beta_0\circ\tau=\sigma^1\otimes(\tau_4 \alpha^2-\tau_2 \alpha^4)\otimes\sigma^1.
\end{align}
Thus, we obtain
\begin{align*}
(\sigma^1\otimes\alpha^3\otimes\sigma^1,0,\xi^1)\tau=
(\sigma^1\otimes((\tau_1+\tau_2^2)\alpha^1-(\tau_1\tau_2+\tau_4)\alpha^2+\alpha^3+2\tau_2\alpha^4)\otimes\sigma^1,0,\xi^1_{\tau_2,-\tau_1}).
\end{align*}
So, the equivalence classes of $(\sigma^1\otimes\alpha^3\otimes\sigma^1,0,\xi^1)$ consists of all cocycles $(\tilde\gamma,0,\xi^1_{\tau_2,-\tau_1})$ satisfying $\tilde\gamma(e_1)a_1=(\tau_1+\tau_2^2)\sigma^1$ and $\tilde\gamma(e_1)a_4=2\tau_2\sigma^1$.
Now, assume \[(\gamma,0,\xi)=(\sigma^1\otimes(x_1\alpha^1+x_2\alpha^2+\alpha^3+x_4\alpha^4)\otimes\sigma^1,0,\xi_{i,j}^1).\]
We set $S=1$ and $U=U(0,p,0,k)$ where  
$k=i-\frac{1}{2}x_4$ and $p=\frac{1}{2}(x_1+j-x_4k-\frac{1}{4}x_4^2)$.
Then
\[(S,U)^*(\gamma,0,\xi)=(\tilde\gamma,0,\xi^1_{\tau_2,-\tau_1})\] where $\tilde\gamma(e_1)a_1=(\tau_1+\tau_2^2)\sigma^1$ and $\tilde\gamma(e_1)a_4=2\tau_2\sigma^1$ for $\tau_2=\frac{1}{2}x_4$ and $\tau_1=\frac{1}{2}(x_1-j-x_4k-\tau_2^2)$.
Finally, we have \[(S,U)^*[\gamma,0,\xi]=[\sigma^1\otimes\alpha^3\otimes\sigma^1,0,\xi^1].\]
\end{proof}

Now, let us assume that $(\gamma,0,\xi)\in Z^2(\mathds R,\mathfrak h_3\oplus\mathds R,\omega_\mathfrak a)_0$ satisfies $\xi(e_1)\neq 0$ on $\mathfrak a/\mathfrak z(\mathfrak a)$ and that $\im(\xi(e_1))$ is orthogonal to $[\mathfrak a,\mathfrak a]$. The set of this cocycles is invariant under equivalence and $G$-action.
Moreover, this case can be devided into two separat cases, namely $\xi(e_1)=0$ on $\mathfrak z(\mathfrak a)$ and $\xi(e_1)\neq 0$ on $\mathfrak z(\mathfrak a)$.
Both cases are also invariant under equivalence and $G$-action.

Because of $\im(\xi(e_1))\bot[\mathfrak a,\mathfrak a]$ the vector $\xi(e_1)a_1$ lies in the center of $\mathfrak a$. Thus $\xi(e_1)a_2\neq 0$ on $\mathfrak a/\mathfrak z(\mathfrak a)$ and moreover $\xi(e_1)a_1\notin[\mathfrak a,\mathfrak a]$ since the cocycle is balanced.
For all $k\in\mathds R$ and $f\neq 0$, the set $\{f^{-2}a_1+ka_4,fa_2,f^{-1}a_3,f^2a_4\}$ is a basis of $\mathfrak a$,
such that the commutator and symplectic form is given by $[f^{-2}a_1+ka_4,fa_2]=f^{-1}a_3$ and $\omega_\mathfrak a=\nu^1\wedge\nu^4+\nu^2\wedge\nu^3$,
where $\{\nu^1,\dots,\nu^4\}$ denotes the corresponding dual basis.
Thus, we can find a basis of $\mathfrak a$ such that
\[\xi(e_1)=\begin{pmatrix}
0 & m & 0 & 0 \\
0 & 0 & 0 & 0 \\
* & * & 0 & * \\
m & 0 & 0 & 0
\end{pmatrix}\]
for a suitable $m\neq 0$.
This shows that the nilpotent $[\gamma,0,\xi]$ lies in the $G$-orbit of
$[\sigma^1\otimes(x_1\alpha^1+x_2\alpha^2+(l+1)\alpha^3+x_4\alpha^4)\otimes\sigma^1,0,\xi^1_{i,j}(l)]$
for a suitable $l\neq -1$ and $x_1,x_2,x_4,i,j\in\mathds R$.
Here $\xi^1_{i,j}(l)$ is defined by
\begin{align}\label{eq:brauch5}
\xi^1_{i,j}(l)=\sigma^1\otimes\alpha^1\otimes(ia_3+a_4)+\sigma^1\otimes\alpha^2\otimes(a_1+ja_3)+\sigma^1\otimes\alpha^4\otimes la_3.
\end{align}
If $i=j=0$ we again omit the indices and simply write $\xi^1(l)$ instead of $\xi^1_{0,0}(l)$.

\begin{lemma}\label{lm:67}
The set of all triples $[\gamma,0,\xi]\in H^2(\mathds R,\mathfrak h_3\oplus\mathds R,\omega_\mathfrak a)_0/G$ satisfying $\xi(e_1)\neq 0$ on $\mathfrak a/\mathfrak z(\mathfrak a)$ and $\im(\xi(e_1))\bot[\mathfrak a,\mathfrak a]$ can be represented by the following cocycles
\begin{align}
(\sigma^1\otimes(l+1)\alpha^3\otimes\sigma^1,0,\xi^1(l)),\quad(\sigma^1\otimes(y_1\alpha^1+\alpha^3)\otimes\sigma^1,0,\xi^1(0)),\quad l\notin\{0,-1\},y_1\in\mathds R
\end{align}
\end{lemma}

\begin{proof}
Again using Lemma \ref{lemma:Z2RH3xR0} yields that the given cocycles are nilpotent.
Now, we begin to compute the equivalence class of the cocycle $(\sigma^1\otimes(l+1)\alpha^3\otimes\sigma^1,0,\xi^1(l))$. Assume therefor $\tau=\sigma^1\otimes(\tau_1a_1+\dots+\tau_4a_4)$.
Then 
 \begin{align}
\tau^*(\xi^1(l))&=\sigma^1\otimes(\tau_1\alpha^1-\tau_4\alpha^2+\tau_2l\alpha^4)\otimes\sigma^1\label{eq:tautau1}\\
\beta_0\circ\tau&=\sigma^1\otimes(\tau_4(l+1)\alpha^2-\tau_2(l+1)\alpha^4)\otimes\sigma^1,\label{eq:tautau2}
\end{align} analogous to the equations (\ref{eq:aequivH3xR_4}) and (\ref{eq:aequivH3xR_5}).
Using equations 
 (\ref{eq:aequivH3xR_1})-(\ref{eq:aequivH3xR_3}) and (\ref{eq:tautau1})-(\ref{eq:tautau2})
we obtain
\begin{align*}
(\sigma^1&\otimes(l+1)\alpha^3\otimes\sigma^1,0,\xi^1(l))\tau\\
=&\big( \sigma^1\otimes( (\tau_1+\tau_2^2)\alpha^1-(\tau_4(l+2)+\tau_1\tau_2)\alpha^2+(l+1)\alpha^3+\tau_2(2l+1)\alpha^4 )\otimes\sigma^1,0,\xi^1_{\tau_2,-\tau_1}(l) \big)
\end{align*}
for $l\notin\{0,1\}$.
Completely analogous we have
\begin{align*}
(\sigma^1\otimes(&y_1\alpha^1+\alpha^3)\otimes\sigma^1,0,\xi^1(0))\tau=\\
&\big( \sigma^1\otimes( (y_1+\tau_1+\tau_2^2)\alpha^1-(2\tau_4+\tau_1\tau_2)\alpha^2+\alpha^3+\tau_2\alpha^4 )\otimes\sigma^1,0,\xi^1_{\tau_2,-\tau_1}(0) \big).
\end{align*}
Thus, we also see that
$[\sigma^1\otimes(x_1\alpha^1+x_2\alpha^2+(l+1)\alpha^3+x_4\alpha^4)\otimes\sigma^1,0,\xi^1_{i,j}(l)]$ for $l\notin\{0,-1\}$ lies in the $G$-orbit of $[\sigma^1\otimes(l+1)\alpha^3\otimes\sigma^1,0,\xi^1(l)]$.
Therefor, we set $S=1$ and $U=U(x,p,q,x)$ where
\begin{align}
x&=\frac{1}{2(l+1)^2}(i-x_4)\label{eq:8}\\
p&=\frac{1}{2l}\big(-(i-x)x-(i-lx-x)^2+j-x_4x-x_1-(l+1)x^2\big)\\
q&=\frac{1}{l+1}\big((i-(l+1)x)(ix-x^2-p+pl-j)-x_1x+x_2-x_4p\big).\label{eq:9}
\end{align}
Then
\[(S,U)^*(\gamma,0,\xi)=\big( \sigma^1\otimes( (\tau_1+\tau_2^2)\alpha^1-\tau_1\tau_2\alpha^2+(l+1)\alpha^3+\tau_2(2l+1)\alpha^4 )\otimes\sigma^1,0,\xi^1_{\tau_2,-\tau_1}(l) \big)\]
for
\begin{align}
\tau_1&=(i-x)x-p+pl-j\label{eq:10}\\
\tau_2&=i-(l+1)x.\label{eq:11}
\end{align}
For $l=0$ we set instead $p=0$ and  the parameters $x$ and $q$ as before (equations (\ref{eq:8}) and (\ref{eq:9})). So, we get analogous
\begin{align*}
(S,U)^*(\sigma^1\otimes(x_1\alpha^1+x_2\alpha^2+\alpha^3+x_4\alpha^4)\otimes\sigma^1,0,\xi^1_{i,j}(0))=\\
(\sigma^1\otimes(y_1\alpha^1+\alpha^3)\otimes\sigma^1,0,\xi^1(0))(\tau_1\alpha^1+\tau_2\alpha^2)\otimes\sigma^1
\end{align*}
for $\tau_1$ and $\tau_2$ (see no. (\ref{eq:10}) and (\ref{eq:11})) and $y_1=-(i-x)x-(i-x)^2+j-x_4x+x_1-x^2\in\mathds R$.

As a consequence we obtain that every $[\gamma,0,\xi]\in H^2(\mathfrak l,\mathfrak a,\omega_\mathfrak a)_0$ satisfying $\xi(e_1)\neq 0$ on $\mathfrak a/\mathfrak z(\mathfrak a)$ and $\im(\xi(e_1))\bot[\mathfrak a,\mathfrak a]$ lies in the $G$-orbit of the cohomology class of one of the in the lemma given cocycles.

It remains to show that this cocycle is uniquely determined.
Therefor we choose for $\xi(e_1)\neq 0$ on $\mathfrak z(\mathfrak a)$ a $L\in\mathfrak l$, $L\neq 0$ and a $v\in\mathfrak a$ with $\xi(e_1)^3v\neq 0$ and consider
\begin{align}\label{gl:h3xR_l}
\omega_\mathfrak a(v,\xi(L)^3v)\omega_\mathfrak a(\xi(L)v,\xi(L)^2v)^{-1}.
\end{align}
Obviously, this value is independent from the choice of $L\neq 0$.
It is also independent from the choice of $v$:
Since $\xi^3(e_1)=[\mathfrak{a},\mathfrak{a}]$ is one-dimensional, every $\tilde v\in\mathfrak a$ satifying $\xi^3(e_1)\tilde v\neq 0$ can be written as the sum of a nontrivial multiple of $v$ and an element in the image of $\xi(L)$.
Since the image of $\xi(L)$ is orthogonal to $[\mathfrak a,\mathfrak a]$, we have
\[\omega_\mathfrak a(\tilde v,\xi^3(L)\tilde v)=\mu^2\omega_\mathfrak a(v,\xi^3(L)v)\]
for some $\mu\neq 0$. Moreover, the center is isotropic and it follows that
\[\omega_\mathfrak a(\xi(L)\tilde v,\xi^2(L)\tilde v)=\mu^2\omega_\mathfrak a(\xi(L)v,\xi^2(L)v).\]
Thus the value of the term (\ref{gl:h3xR_l}) is only given by the structure of the Lie algebra and the symplectic form (hence invariant under the $G$-action).
Furthermore, it is invariant under equivalence and defines hence an invariant of the $G$-orbits.
Finally, the term (\ref{gl:h3xR_l}) equals $l$ for $\xi=\xi^1(l)$, $L=e_1$ and $v=a_2$ and we obtain, that $(\sigma^1\otimes(l_1+1)\alpha^3\otimes\sigma^1,0,\xi^1(l_1))$ and $(\sigma^1\otimes(l_2+1)\alpha^3\otimes\sigma^1,0,\xi^1(l_2))$ define isomorphic Lie algebras if and only if $l_1=l_2$.

It remains to show that $[\sigma^1\otimes(x_1\alpha^1+\alpha^3)\otimes\sigma^1,0,\xi^1(0)]$ and $[\sigma^1\otimes(y_1\alpha^1+\alpha^3)\otimes\sigma^1,0,\xi^1(0)]$ lie in the same $G$-orbit if and only if $x_1=y_1$ holds.
Since the invariant can only be given by a confusing formular, we show directly that there is a suitable $(S,U)\in G$ if and only if $x_1=y_1$.
So, let $(S,U)$ be an isomorphism of pairs satisfying 
\[(S,U)^*[\sigma^1\otimes(x_1\alpha^1+\alpha^3)\otimes\sigma^1,0,\xi^1(0)]=[\sigma^1\otimes(y_1\alpha^1+\alpha^3)\otimes\sigma^1,0,\xi^1(0)].\]
This means that $(S,U)^*\xi^1(0)=\xi^1_{\tau_2,-\tau_1}(0)$ and 
\[(S,U)^*\gamma=\sigma^1\otimes((y_1+\tau_1+\tau_2^2)\alpha^1-(2\tau_4+\tau_1\tau_2)\alpha^2+\alpha^3+\tau_2\alpha^4)\otimes\sigma^1=:\tilde\gamma,\]
where $\gamma:=\sigma^1\otimes(x_1\alpha^1+\alpha^3)\otimes\sigma^1$.
In particular, assume $U^{-1}a_2=fa_2+xa_1+z$ for $f,x\in\mathds R$, $z\in\mathfrak z(\mathfrak a)$ and assume $Se_1=\lambda e_1$ for $\lambda\neq 0$.
We consider
\begin{align*}
1&=\omega_\mathfrak a(a_1,a_4)=\omega_\mathfrak a(\xi^1_{\tau_2,-\tau_1}(0)(e_1)a_2,(\xi^1_{\tau_2,-\tau_1}(0)(e_1))^2a_2)\\
&=\omega_\mathfrak a((S,U)^*\xi^1(0)(e_1)a_2,((S,U)^*\xi^1(0)(e_1))^2a_2)\\
&=\omega_\mathfrak a(\xi^1(0)(\lambda e_1)U^{-1}a_2,(\xi^1(0)(\lambda e_1))^2U^{-1}a_2)\\
&=\lambda^3f^2
\end{align*}
and
\begin{align*}
1&=\omega_\mathfrak a(a_2,[\xi^1_{\tau_2,-\tau_1}(0)(e_1)a_2,a_2])\\
&=\omega_\mathfrak a(a_2,[(S,U)^*\xi^1(0)(e_1)a_2,a_2])\\
&=\lambda\omega_\mathfrak a(U^{-1}a_2,[\xi^1(0)(e_1)U^{-1}a_2,U^{-1}a_2])\\
&=\lambda f^3.
\end{align*}
We get $f=\lambda=1$ and thus obtain 
\begin{align*}
0&=\gamma(e_1,(\xi^1(0)(e_1))^2U^{-1}a_2,e_1)\\
&=\tilde\gamma(e_1,(\xi^1_{\tau_2,-\tau_1}(0)(e_1))^2a_2,e_1)\\
&=\tilde\gamma(e_1,\tau_2a_3+a_4,e_1)=2\tau_2\\
&=2\omega_\mathfrak a(a_2,(\xi^1_{\tau_2,-\tau_1}(0)(e_1))^2a_2)\\
&=2\omega_\mathfrak a(U^{-1}a_2,(\xi^1(0)(e_1))^2U^{-1}a_2)\\
&=x.
\end{align*}
Hence 
\begin{align*}
y_1=\tilde\gamma(e_1,\xi^1_{\tau_2,-\tau_1}(0)(e_1)a_2,e_1)=\gamma(e_1,\xi^1(0)(e_1)U^{-1}a_2,e_1)=x_1
\end{align*}
and consequently $[\sigma^1\otimes(x_1\alpha^1+\alpha^3)\otimes\sigma^1,0,\xi^1(0)]$ and $[\sigma^1\otimes(y_1\alpha^1+\alpha^3)\otimes\sigma^1,0,\xi^1(0)]$ lie in the same $G$-orbit, if and only if $x_1=y_1$ holds.
\end{proof}

Now, let us assume that $(\gamma,0,\xi)\in Z^2(\mathds R,\mathfrak h_3\oplus\mathds R,\omega_\mathfrak a)_0$ satisfies that the image of $\xi(e_1)$ is not  orthogonal to $[\mathfrak a,\mathfrak a]$.
This case is (as the previous cases) invariant under equivalence and $G$-action. Using the balanced condition, we get that $\xi(e_1)\neq 0$ on $\mathfrak a/\mathfrak z(\mathfrak a)$ and $\xi(e_1)=0$ on $\mathfrak z(\mathfrak a)$.

Since $\{a_1,a_2+xa_1,a_3,a_4-xa_3\}$, $\{a_1+xa_3,a_2+xa_4,a_3,a_4\}$ and $\{f^{-2}a_1,fa_2,f^{-1}a_3,f^2a_4\}$ for all $x\in\mathds R$ and $f\neq 0$ are bases of $\mathfrak a$, such that the commutator and the symplectic form is given unchanged, we can choose a basis of $\mathfrak a$ and a basis $\{L\}$ of $\mathfrak l$, such that $\xi(L)$ is given by 
\begin{align*}
\begin{pmatrix}
0&0&0&0\\
1&0&0&0\\
i&j&0&0\\
0&n&0&0
\end{pmatrix}
\end{align*}
for a suitable $n\in\{1,0,-1\}$ and $i,j\in\mathds R$ with respect to a suitable basis.
This shows that every nilpotent cocycle, whose image of $\xi(e_1)$ is not orthogonal to $[\mathfrak a,\mathfrak a]$, lies in the $G$-Orbit of 
\[(\sigma^1\otimes(x_1\alpha^1+x_2\alpha^2+j\alpha^3+y_4\alpha^4)\otimes\sigma^1,0,\xi^{1,n}_{i,j})\]
for some $x_1,\dots,x_3,i,j\in\mathds R$, $y_4\neq 0$ and $n\in\{1,0,-1\}$.
Here $\xi^{1,n}_{i,j}$ is defined by
\begin{align}\label{eq:brauch7}
\xi^{1,n}_{i,j}=\sigma^1\otimes\alpha^1\otimes(a_2+ia_3)+\sigma^1\otimes\alpha^2\otimes(ja_3+na_4).
\end{align}
Again, we write simply $\xi^{1,n}$ instead of $\xi^{1,n}_{0,0}$.

\begin{lemma}\label{lm:68}
The cocycles 
\[(\sigma^1\otimes(y_2\alpha^2+y_4\alpha^4)\otimes\sigma^1,0,\xi^{1,\pm 1}),\quad(\sigma^1\otimes(y_2\alpha^2\pm\alpha^4)\otimes\sigma^1,0,\xi^{1,0})\]
where $y_2\geq 0$, $y_4\neq 0$ form a system of representatives of all $[\gamma,0,\xi]\in H^2(\mathds R,\mathfrak h_3\oplus\mathds R,\omega_\mathfrak a)_0/G$, which satisfy that the image of $\xi(e_1)$ is not orthogonal to $[\mathfrak a,\mathfrak a]$.
\end{lemma}

\begin{proof}
Using Lemma \ref{lemma:Z2RH3xR0} yields directly that the given cocycles are nilpotent.

Now, we determine the equivalence classes of $(\sigma^1\otimes(y_2\alpha^2+y_4\alpha^4)\otimes\sigma^1,0,\xi^{1,n})$ for $n\in\{1,0,-1\}$.
Therefor, we use equations (\ref{eq:aequivH3xR_1})-(\ref{eq:aequivH3xR_3}) and
\begin{align*}
\tau^*(\xi(L)A)&=(\sigma^1\otimes(-\tau_3\alpha^1+\tau_1n\alpha^2)\otimes\sigma^1)(L)A\\
\beta_0(\tau L)A&=(\sigma^1\otimes((\tau_2n+\tau_3)\alpha^1-\tau_1n\alpha^2-\tau_1\alpha^3)\otimes\sigma^1)(L)A
\end{align*}
and obtain
\begin{align*}
(\sigma^1&\otimes(y_2\alpha^2+y_4\alpha^4)\otimes\sigma^1,0,\xi^{1,n})\tau=\\
&(\sigma^1\otimes((-2\tau_3+\tau_2^2-\tau_2n)\alpha^1+(y_2+2\tau_1n-\tau_1\tau_2)\alpha^2+\tau_1\alpha^3+y_4\alpha^4)\otimes\sigma^1,0,\xi^{1,n}_{\tau_2,-\tau_1}).
\end{align*}
Thus
\[[\sigma^1\otimes(x_1\alpha^1+x_2\alpha^2+j\alpha^3+y_4\alpha^4)\otimes\sigma^1,0,\xi^{1,n}_{i,j}]=[\sigma^1\otimes(y_2\alpha^2+y_4\alpha^4)\otimes\sigma^1,0,\xi^{1,n}],\]
since $x_2$ can be written as $x_2=y_2-2jn+ij$.
Moreover, we have 
\begin{align*}
(-1,\diag(1,-1,-1,1))^*(\sigma^1\otimes(y_2\alpha^2+y_4\alpha^4)\otimes\sigma^1,0,\xi^{1,n})=
(\sigma^1\otimes(-y_2\alpha^2+y_4\alpha^4)\otimes\sigma^1,0,\xi^{1,n}).
\end{align*}
Hence the cohomology classes 
$[\sigma^1\otimes(x_1\alpha^1+x_2\alpha^2+j\alpha^3+y_4\alpha^4)\otimes\sigma^1,0,\xi^{1,n}_{i,j}]$ 
and 
$[\sigma^1\otimes(|y_2|\alpha^2+y_4\alpha^4)\otimes\sigma^1,0,\xi^{1,n}]$ lie in the same $G$-orbit.
If in addition $n=0$, we set $S=f^{-3}$ and $U=\diag(f^{-2},f,f^{-1},f^2)$ where $f=\sqrt[8]{|y_4|}$.
Then 
\begin{align*}
(S,U)^*(\sigma^1\otimes(|y_2|\alpha^2+y_4\alpha^4)\otimes\sigma^1,0,\xi^{1,0})=
(\sigma^1\otimes(f^{-7}|y_2|\alpha^2+f^{-8}y_4\alpha^4)\otimes\sigma^1,0,\xi^{1,0}).
\end{align*}

At this point, we have shown that  every nilpotent cocycle, whose image of $\xi(e_1)$ is not orthogonal to $[\mathfrak a,\mathfrak a]$, lies in the $G$-orbit of one of the equivalence classes given in the lemma. Moreover, this cocycle is unique:
Therefor, let $(S,U)$ be an isomorphism of pairs satisfying \[(S,U)^*[\gamma,0,\xi^{1,n}]=[\sigma^1\otimes(\tilde y_2\alpha^2+\tilde y_4\alpha^4)\otimes\sigma^1,0,\xi^{1,\tilde n}]\] for the parameters given in the lemma, where $\gamma=\sigma^1\otimes(y_2\alpha^2+y_4\alpha^4)\otimes\sigma^1$ holds.
This means that $(S,U)^*\xi^{1,n}=\xi^{1,\tilde n}_{\tau_2,-\tau_1}$ and 
\[(S,U)^*\gamma=\sigma^1\otimes(z_1\alpha^1+(\tilde y_2+2\tau_1n-\tau_1\tau_2)\alpha^2+\tau_1\alpha^3+\tilde y_4\alpha^4)\otimes\sigma^1=:\tilde\gamma\]
for some $z_1,\tau_1,\tau_2\in\mathds{R}$.
Because of 
$0=\omega_\mathfrak a(a_1,a_3)=\omega_\mathfrak a(U^{-1}a_1,U^{-1}a_3)$ we have  $U^{-1}a_1=fa_1+z$ for some $f\neq 0$ and $z\in\mathfrak z(\mathfrak a)$.
Assume $Se_1=\lambda e_1$ for some $\lambda\neq 0$.
It holds
\begin{align*}
\tilde n=\omega_\mathfrak a(a_1,(\xi^{1,\tilde n}_{\tau_2,-\tau_1}(e_1))^2a_1)=\lambda^2\omega_\mathfrak a(U^{-1}a_1,(\xi^{1,n}(e_1))^2U^{-1}a_1)=\lambda^2f^2n.
\end{align*}
Since $n,\tilde n\in\{1,0,-1\}$, we have especially $n=\tilde n$.
(Here we mention that the computation shows that $\sign(\omega_\mathfrak a(a_1,(\xi(e_1))^2a_1))\in\{1,0,-1\}$ is an invariant of the $G$-orbits.)
Moreover, we have 
\begin{align*}
1&=\omega_\mathfrak a(\xi^{1,n}_{\tau_2,-\tau_1}(e_1)a_1,[a_1,\xi^{1,n}_{\tau_2,-\tau_1}(e_1)a_1])\\
&=\omega_\mathfrak a(\xi^{1,n}(e_1)U^{-1}a_1,[U^{-1}a_1,\xi^{1,n}(e_1)U^{-1}a_1])=\lambda^2f^3.
\end{align*}
Thus for $n=\tilde n=\pm 1$ it follows that $f=\lambda=1$ and for $n=0$ we obtain $\lambda^2=f^{-3}>0$.
Because of $1=\omega_\mathfrak a(a_1,a_4)=\omega_\mathfrak a(U^{-1}a_1,U^{-1}a_4)$ we get $U^{-1}a_4=f^{-1}a_4+\tilde z$ for $\tilde z\in[\mathfrak a,\mathfrak a]$ and we obtain 
\begin{align*}
\tilde y_4=\tilde\gamma(e_1,a_4,e_1)=((S,U)^*\gamma)(e_1,a_4,e_1)=\lambda^2\gamma(e_1,U^{-1}a_4,e_1)=\lambda^2f^{-1}y_4.
\end{align*}
Thus for $n=\tilde n=\pm 1$ the parameter $y_4=\tilde y_4\neq 0$. For $n=0$ and $y_4,\tilde y_4\in\{1,-1\}$ we also get $y_4=\tilde y_4$. Especially it also follows for $n=0$ that $f=\lambda=1$ and hence
\begin{align*}
y_2&=\gamma(e_1,\xi^{1,n}(e_1)U^{-1}a_1,e_1)+2n\omega_\mathfrak a(\xi^{1,n}(e_1)U^{-1}a_1,(\xi^{1,n}(e_1))^2U^{-1}a_1)\\
&=\tilde\gamma(e_1,\xi^{1,n}_{\tau_2,-\tau_1}(e_1)a_1,e_1)+2n\omega_\mathfrak a(\xi^{1,n}_{\tau_2,-\tau_1}(e_1)a_1,(\xi^{1,n}_{\tau_2,-\tau_1}(e_1))^2a_1)=\tilde y_2.
\end{align*}
\end{proof}

Now, we can summerize Lemma \ref{lm:66a} - \ref{lm:68} as follows:
\begin{lemma}
The cocycles
\begin{align*}
&(\sigma^1\otimes\alpha^3\otimes\sigma^1,0,\xi^1)\\
&(\sigma^1\otimes(l+1)\alpha^3\otimes\sigma^1,0,\xi^1(l)),\quad l\notin\{0,-1\}\\
&(\sigma^1\otimes(y_1\alpha^1+\alpha^3)\otimes\sigma^1,0,\xi^1(0)),\quad y_1\in\mathds R\\
&(\sigma^1\otimes(y_2\alpha^2+y_4\alpha^4)\otimes\sigma^1,0,\xi^{1,\pm 1})\\
&(\sigma^1\otimes(y_2\alpha^2\pm\alpha^4)\otimes\sigma^1,0,\xi^{1,0}),\quad y_2\geq 0,y_4\neq 0
\end{align*}
form a system of representatives of $H^2(\mathds R,\mathfrak h_3\oplus\mathds R,\alpha^1\wedge\alpha^4+\alpha^2\wedge\alpha^3)_0/G$.
\end{lemma}

\subsection{Case $\mathfrak l=\mathds R^2$, $\mathfrak a=\mathds R^2$}\label{NSLA:5}

Let $\{e_1,e_2\}$ denote the standard basis of $\mathfrak l=\mathds R^2$, $\{\sigma^1,\sigma^2\}$ its dual basis of $\mathfrak l^*$ and $\{a_1,a_2\}$ the standard basis of $\mathfrak a=\mathds R^2$ and $\omega_\mathfrak a=\alpha^1\wedge\alpha^2$, where $\{\alpha^1,\alpha^2\}$ denotes the corresponding dual basis of $a^*$.

At first, we consider cocycles with $\xi=0$.
For every linear $\gamma:\mathfrak l\rightarrow \Hom(\mathfrak a,\mathfrak l^*)$ and every $\epsilon\in C^2(\mathfrak l,\mathfrak l^*)$ the triple $(\gamma,\epsilon,0)$ satisfies the conditions (\ref{eq:lemma1}) - (\ref{eq:lemma5}), since $\beta=0$, $\mathfrak a=\mathds R^2$ is abelian and $\mathfrak l=\mathds R^2$ is two-dimensional.
Thus $(\gamma,\epsilon,0)\in Z^2(\mathfrak l,\mathfrak a,\omega_\mathfrak a)$.
\begin{lemma}
A cocycle $(\gamma,\epsilon,0)\in Z^2(\mathfrak l,\mathfrak a,\omega_\mathfrak a)$ is a nilpotent one if and only if 
\begin{itemize}
\item $\gamma=0$ and $\epsilon\neq 0$ or
\item $\ker\gamma_0=\{0\}$.
\end{itemize}
\end{lemma}

\begin{proof}
Because of \[\mathfrak z(\mathfrak a)\cap\ker\beta_0\cap\mathfrak a^\mathfrak l_\xi\cap\ker\gamma_0=\ker\gamma_0\]
condition (b) from page \pageref{eq:aundb} is satisfied if and only if $\gamma=0$ or $\ker\gamma_0=\{0\}$.
If $\gamma=0$, then condition (a) is equivalent to  $\epsilon\neq 0$.
Now, assume $\ker\gamma_0=\{0\}$. We show that condition (a) holds:
At first, we consider
$\ker\gamma_0(a_1)\cap\ker\gamma_0(a_2)=\{0\}$. Then (a.2) implies $L=0$ and (a) holds.
Now, assume $\ker\gamma_0(a_1)\cap\ker\gamma_0(a_2)\neq\{0\}$.
Then there is a $L_0\in\ker\gamma_0(a_1)\cap\ker\gamma_0(a_2)$ with $L_0\neq 0$. Using $\ker\gamma_0=\{0\}$ yields the existence of a $\tilde L_0\neq 0$, which satisfies $\gamma_0(A)\tilde L_0\neq 0$ for all $A\neq 0$.
Moreover, $\{L_0, \tilde L_0\}$ form a basis of $\mathfrak a$ and 
\begin{align}\label{eq:n1}
\{\gamma(\tilde L_0)A\mid A\in\mathfrak a\}=\mathfrak l^*
\end{align}
holds.
Now, suppose $L\in\mathfrak l$ and assume that $\gamma(L)=0$ (a.2) and $\alpha(L,\cdot)=0$ (a.3) hold.
Since $\alpha(L,\tilde L_0)=0$ is satisfied if and only if 
\begin{align*}
0&=\omega_\mathfrak a((\gamma(L))^*(\tilde L_0),A)-\omega_\mathfrak a((\gamma(\tilde L_0))^*(L),A)\\
&=\gamma(L,A,\tilde L_0)-\gamma(\tilde L_0,A,L)\\
&=-\gamma(\tilde L_0,A,L)
\end{align*}
for all $A\in\mathfrak a$, it follows $L=0$ by using  no. (\ref{eq:n1}). So, condition (a) holds for all $\ker\gamma_0=\{0\}$ and the assertion of the lemma follows.
\end{proof}

For $(0,\epsilon,0)\in Z^2(\mathfrak l,\mathfrak a,\omega_\mathfrak{a})_0$ we can easily find a basis $\{L_1,L_2\}$ of $\mathfrak l$ with $\epsilon(L_1,L_2)=Z_1$, where $\{Z_1,Z_2\}$ denotes the corresponding dual basis of $\mathfrak l^*$ for $\{L_1,L_2\}$. Thus $[0,\epsilon,0]$ lies in the $G$-orbit of $[0,\sigma^1\wedge\sigma^2\otimes\sigma^1,0]$.
\\\quad\\
Now, assume in the following that $(\gamma,\epsilon,0)\in Z^2(\mathfrak l,\mathfrak a,\omega_\mathfrak a)_0$ satisfies $\ker\gamma_0=\{0\}$.
Two cocycles $(\gamma,\epsilon,0)$ and $(\gamma',\epsilon',0)$ are equivalent if and only if $\gamma=\gamma'$ and 
\begin{align}\label{eq:epsequiv}
\epsilon'=\epsilon+\tau^*\alpha-\ev(\gamma\wedge\tau)
\end{align}
for some $\tau\in C^1(\mathfrak l,\mathfrak a)$.
Because of $\tau^*\alpha(e_1,e_2)=-(\gamma(e_1)\tau\cdot)(e_2)+(\gamma(e_2)\tau\cdot)(e_1)$ equation (\ref{eq:epsequiv}) is equivalent to
\begin{align}
\epsilon'(e_1,e_2)(e_1)&=\epsilon(e_1,e_2)(e_1)-\gamma(e_1,\tau(e_1),e_2)+2\gamma(e_2,\tau(e_1),e_1)-\gamma(e_1,\tau(e_2),e_1)\label{eq:12}\\
\epsilon'(e_1,e_2)(e_2)&=\epsilon(e_1,e_2)(e_2)-2\gamma(e_1,\tau(e_2),e_2)+\gamma(e_2,\tau(e_2),e_1)+\gamma(e_2,\tau(e_1),e_2).\label{eq:13}
\end{align}
For every $A\in\mathfrak a$ we can consider $\gamma_0(A)$ as a bilinear form on $\mathds R^2$.
The space of skewsymmetric maps from $\mathds R^2\times\mathds R^2$ to $\mathds R$ is one-dimensional.
Thus, there is an $A\in\mathfrak a$, $A\neq 0$ such that $\gamma_0(A)$ is symmetric as a bilinear form on $\mathds R^2$.

At first, let us consider that there is a $\tilde A\in\mathfrak a$ such that $\gamma_0(\tilde A)$ is not symmetric.
So, there are bases $\{A_1,A_2\}$ of $\mathfrak a$ and $\{L_1,L_2\}$ of $\mathfrak l$ such that $\gamma_0(A_1)=Z_1\otimes Z_1+\kappa Z_2\otimes Z_2$ for exactly one $\kappa\in\{1,0,-1\}$ and $\gamma_0(A_2)$ is not symmetric.
Here $\{Z_1,Z_2\}$ denotes the dual basis of $\mathfrak l^*$ for $\{L_1,L_2\}$.

Given the linear map $\gamma:\mathfrak l\rightarrow\Hom(\mathfrak a,\mathfrak l^*)$ we can describe $\gamma_0(a_1)$ and $\gamma_0(a_2)$ by a $2\times 2$-matrices $A$ and $B$ with respect to the standard basis $\{e_1,e_2\}$ in the domain $\mathfrak l$ and the corresponding dual basis $\{\sigma^1,\sigma^2\}$ in the codomain $\mathfrak l^*$. We obtain here
\begin{align}\label{eq:gamma0AB}
(A)_{ji}:=(\gamma_0(a_1)(e_i))e_j\quad (B)_{ji}:=(\gamma_0(a_2)(e_i))e_j
\end{align}
for $i,j=1,2$. 

If $\gamma_0(A)$ is a non-degenerate symmetric bilinear form on $\mathds R^2$, then we can consider conformal mappings with respect to $\gamma_0(A)$.

Let $\operatorname{CO}(2)$ denote the set of all linear maps $S:\mathds R^2\rightarrow \mathds R^2$, which are conformal with respect to $\gamma_0(a_1)=\sigma^1\otimes\sigma^1+\sigma^2\otimes\sigma^2$.
It holds $\operatorname{CO}(2)=\{S:\mathds R^2\rightarrow\mathds R^2\mid \lambda^2\gamma_0(a_1)=S^*\circ \gamma_0(a_1)\circ S, \lambda\neq 0\}$.

For $\gamma_0(a_1)=\sigma^1\otimes\sigma^1-\sigma^2\otimes\sigma^2$ we define 
$\tilde{\operatorname{CO}}(1,1)$ as the set of all linear maps $S:\mathds R^2\rightarrow\mathds R^2$ satisfying $\lambda\gamma_0(a_1)=S^*\circ \gamma_0(a_1)\circ S$ for some $\lambda\neq 0$.
The set $\tilde{\operatorname{CO}}(1,1)$ consists of all linear maps
\[S^{\pm}(s_1,s_2):=\begin{pmatrix}s_1&\pm s_2\\s_2&\pm s_1\end{pmatrix}\]
with $s_1^2-s_2^2\neq 0$.
Moreover, for some $S\in \tilde{\operatorname{CO}}(1,1)$ with $\lambda\gamma_0(a_1)=S^*\circ \gamma_0(a_1)\circ S$ let $\lambda(S)=\sign{\lambda}\in\{1,-1\}$ denote the sign of $\lambda$. 

We want to describe the action of $G$ with $A$ and $B$.

\begin{lemma}\label{lm:66}
Assume that $\gamma,\tilde\gamma:\mathfrak l\rightarrow\Hom(\mathfrak a,\mathfrak l^*)$ are given by $(A,B)$ and $(A,\tilde B)$, where $B$ and $\tilde B$ are not symmetric.
There is an somorphism of pairs $(S,U)$ satisfying $(S,U)^*\gamma=\tilde\gamma$
\begin{itemize}
\item for $A=\operatorname{id}$, if and only if there is an $S\in \operatorname{CO}(2)$ and a $\lambda\in\mathds R$ such that $\tilde B=\lambda A+S^TBS$,
\item for $A=\diag(1,-1)$, if and only if there is an $S\in\tilde{\operatorname{CO}}(1,1)$ and a $\lambda\in\mathds R$ such that $\tilde B=\lambda A+\lambda(S) S^TBS$,
\item for $A=\diag(1,0)$, if and only if there is a $\lambda\in\mathds R$ and an \[S=\begin{pmatrix}x&0\\y&z\end{pmatrix}\in GL(2,\mathds R)\] such that $\tilde B=\lambda A+ x^2 S^TBS$.
\end{itemize}
\end{lemma}

\begin{proof}
Since $\gamma_0(a_2)$ and $\tilde\gamma_0(a_2)$ are not symmetric, the isomorphism $U$ of $(\mathfrak a,\omega_\mathfrak a)$ satisfies $Ua_1\in\Span\{a_1\}$.
Thus, there is an $(S,U)$ satisfying $(S,U)^*\gamma=\tilde\gamma$, if and only if there are numbers $a\neq 0$, $b\in\mathds R$ and an $S\in GL(2,\mathds R)$ such that
\[\tilde\gamma_0(a_1)=aS^*\circ\gamma_0(a_1)\circ S,\quad \tilde\gamma_0(a_2)=S^*\circ\gamma_0(ba_1+a^{-1}a_2)\circ S.\]
Now, $a>0$ and $S_2=\frac{1}{\sqrt{a}}S\in\operatorname{CO}(2)$ is a conformal map for $\tilde\gamma_0(a_1)=\gamma_0(a_1)=\sigma^1\otimes\sigma^1+\sigma^2\otimes\sigma^2$  and we get 
\[\tilde\gamma_0(a_2)=\frac{b}{a}\gamma_0(a_1)+{S_2}^*\circ\gamma_0(a_2)\circ S_2.\]
For $\tilde\gamma_0(a_1)=\gamma_0(a_1)=\sigma^1\otimes\sigma^1-\sigma^2\otimes\sigma^2$ 
we have $S_2=\frac{1}{\sqrt{|a|}}S\in\tilde{\operatorname{CO}}(1,1)$
and
\[\tilde\gamma_0(a_2)=\pm b\gamma_0(a_1)+\lambda(S_2){S_2}^*\circ\gamma_0(a_2)\circ S_2.\]
Finally, for $\tilde\gamma_0(a_1)=\gamma_0(a_1)=\sigma^1\otimes\sigma^1$ we can easily see that $\tilde\gamma_0(a_1)=aS^*\circ\gamma_0(a_1)\circ S$ is equivalent to 
\[S=\begin{pmatrix}x&0\\y&z\end{pmatrix}\in GL(2,\mathds R)\] 
and $a=x^{-2}$. Thus
\[\tilde\gamma_0(a_2)=bx^{-2}\gamma_0(a_1)+x^2S^*\circ\gamma_0(a_2)\circ S\]
and the assertion yields.
\end{proof}

For a given $\gamma_0(a_2)$ let $\gamma_0(a_2)_s$ denote the symmetric part of $\gamma_0(a_2)$ and $\gamma_0(a_2)_a$ the skewsymmetric one.
We set 
\begin{align}
\gamma_a&:=\sigma^1\otimes\alpha^2\otimes\sigma^2-\sigma^2\otimes\alpha^2\otimes\sigma^1,\label{eq:brauch11}\\
\gamma^1_t&:=\sigma^1\otimes\alpha^1\otimes\sigma^1+\sigma^2\otimes\alpha^1\otimes\sigma^2+t(\sigma^1\otimes\alpha^2\otimes\sigma^1-\sigma^2\otimes\alpha^2\otimes\sigma^2)+\gamma_a.\label{eq:brauch12}
\end{align}

Using Lemma \ref{lm:66} it is easy to see that for every $\gamma$ satisfying $\gamma_0(a_1)=\sigma^1\otimes\sigma^1+\sigma^2\otimes\sigma^2$ and not symmetric $\gamma_0(a_2)$ there is an $(S,U)\in G$ and a $t\geq 0$ such that $(S,U)^*\gamma=\gamma^1_t$. This $t$ is uniquely determined:
For $\gamma_0(a_1)=\sigma^1\otimes\sigma^1+\sigma^2\otimes\sigma^2$ and not symmetric $\gamma_0(a_2)$ let $\lambda_1,\lambda_2$ denote the eigenvalues of the symmetric part of $\gamma_0(a_2)$ and let $\kappa\neq 0$ be the number such that $\kappa\gamma_a$ equals the skewsymmetric part of $\gamma_0(a_2)$. Then 
\begin{align}\label{eq:term}
\frac{1}{2}|(\lambda_1-\lambda_2)\kappa^{-1}|
\end{align}
is invariant under equivalence and all isomorphisms of pairs $(S,U)$, which satisfy
\[(S,U)^*\gamma_0(a_1)=\sigma^1\otimes\sigma^1+\sigma^2\otimes\sigma^2.\]
Since, moreover, the term of (\ref{eq:term}) has the value $t$ for  $\gamma=\gamma^1_t$, the value $t$ is an invariant of the $G$-orbits.
For $\tau=\sigma^1\otimes\tau_1a_1+\sigma^2\otimes\tau_2a_1$ where $\tau_1,\tau_2\in\mathds R$ we have
\[-\gamma^1_t(e_1,\tau(e_1),e_2)-2\gamma^1_t(e_2,\tau(e_1),e_1)-\gamma^1_t(e_1,\tau(e_2),e_1)=-\tau_2\]
and 
\[-2\gamma^1_t(e_1,\tau(e_2),e_2)+\gamma^1_t(e_2,\tau(e_2),e_1)+\gamma^1_t(e_2,\tau(e_1),e_2)=\tau_1.\]
Using no. (\ref{eq:12}) and (\ref{eq:13}) we obtain that every $[\gamma,\epsilon,0]$ with $\gamma_0(a_1)=\sigma^1\otimes\sigma^1+\sigma^2\otimes\sigma^2$ and not symmetric $\gamma_0(a_2)$ is in the $G$-orbit of $[\gamma^1_t,0,0]$
for an uniquely determined $t\geq 0$.
\\\quad\\
For $t\in\mathds R$ we define
\begin{align}
\gamma^0_t&=\sigma^1\otimes\alpha^1\otimes\sigma^1+t(\sigma^1\otimes\alpha^2\otimes\sigma^2+\sigma^2\otimes\alpha^2\otimes\sigma^1)+\gamma_a,\label{eq:brauch13}\\
\gamma^0_\pm&=\sigma^1\otimes\alpha^1\otimes\sigma^1\pm\sigma^2\otimes\alpha^2\otimes\sigma^2+\gamma_a.\label{eq:brauch14}
\end{align}

Now, assume that $\gamma_0(a_1)=\sigma^1\otimes\sigma^1$ and \[\gamma_0(a_2)=a\sigma^1\otimes\sigma^1+b\sigma^1\otimes\sigma^2+b\sigma^2\otimes\sigma^1+c\sigma^2\otimes\sigma^2+\kappa(\sigma^1\otimes\sigma^2-\sigma^2\otimes\sigma^1)\]
for $a,b,c\in\mathds R$ and $\kappa\neq 0$.

The cases $c=0$, $c>0$ and $c<0$ are invariant under equivalence and the action of all $(S,U)\in G$, which satisfy $(S,U)^*\gamma_0(a_1)=\sigma^1\otimes\sigma^1$.

For $c=0$ we set $S=\diag(1,\kappa^{-1})$ and $\lambda=-a$.
This shows that there is an $(S,U)$ satisfying $(S,U)^*\gamma=\gamma^0_t$ for some $t \in\mathds R$.
In addition, this $t\in\mathds R$ is uniquely defined, because on the one hand the value of 
\begin{align}\label{eq:oben3}
\frac{{\gamma_0}(a_2)_s(e_1,e_2)}{{\gamma_0}(a_2)_a(e_1,e_2)}
\end{align}
is invariant under equivalence and all isomorphisms of pairs $(S,U)$ which satisfy
\[(S,U)^*\gamma_0(a_1)=\sigma^1\otimes\sigma^1.\]
On the other hand the value of formula (\ref{eq:oben3}) equals the parameter $t$ for $\gamma=\gamma^0_t$.\\
Assume $\tau=\sigma^1\otimes(\tau_1a_1+\tau_2a_2)+\sigma^2\otimes(\tau_3a_1+\tau_4a_2)$ and $\gamma=\gamma^0_t$. Then
\[-\gamma(e_1,\tau(e_1),e_2)-2\gamma(e_2,\tau(e_1),e_1)-\gamma(e_1,\tau(e_2),e_1)=(t-3)\tau_2-\tau_3\]
and 
\[-2\gamma(e_1,\tau(e_2),e_2)+\gamma(e_2,\tau(e_2),e_1)+\gamma(e_2,\tau(e_1),e_2)=-(t+3)\tau_4.\]
Thus $(\gamma^0_t,\epsilon,0)$ and $(\gamma^0_t,0,0)$ for $t\neq-3$ and for all $\epsilon\in C^2(\mathfrak l,\mathfrak l^*)$ are equivalent (compare no. (\ref{eq:12}) and (\ref{eq:13})).
Moreover, the cocycles $(\gamma^0_{-3},\epsilon,0)$ and $(\gamma^0_{-3},\epsilon',0)$ are equivalent if and only if $\epsilon(e_1,e_2)(e_2)=\epsilon'(e_1,e_2)(e_2)$.
We consider $\lambda=0$ and $S=\diag(x,x^{-3})$. Then $(S,U)^*\gamma^0_{-3}=\gamma^0_{-3}$ and $(S,U)^*\epsilon=x^{-5}\epsilon$.
So, $[\gamma^0_{-3},\epsilon,0]$ lies in the $G$-orbit of either $[\gamma^0_{-3},0,0]$ or $[\gamma^0_{-3},\sigma^1\wedge\sigma^2\otimes\sigma^2,0]$.

For $c>0$ there is an $(S,U)\in G$ such that $(S,U)^*\gamma=\gamma^0_+$. 
This can be easily seen by choosing a suitable $\lambda$ and \begin{align}\label{eq:S1}
S=\begin{pmatrix}x&0\\-bxc^{-1}&\kappa^{-1}x^{-3}\end{pmatrix}
\end{align}
with $x=\sqrt[4]{c\kappa^{-2}}$.
Using the knowledge of the equivalence classes yields that $[\gamma,\epsilon,0]$ lies in the $G$-orbit of $[\gamma^0_+,0,0]$.
Analogous we obtain for a negative $c$ that $[\gamma,\epsilon,0]$ lies in the $G$-orbit of $[\gamma^0_-,0,0]$.
\\\quad\\
Now, let $(\gamma,\epsilon,0)$ be a nilpotent cocycle such that $\gamma_0(a_1)=\sigma^1\otimes\sigma^1-\sigma^2\otimes\sigma^2$ and $\gamma_0(a_2)$ is not symmetric.
We define
\begin{align}
\gamma_{s1}:=\sigma^1\otimes\alpha^2\otimes\sigma^1+\sigma^2\otimes\alpha^2\otimes\sigma^2&,\quad
\gamma_{s2}:=\sigma^1\otimes\alpha^2\otimes\sigma^2+\sigma^2\otimes\alpha^2\otimes\sigma^1,\label{eq:brauch15}\\
\gamma^{-1}:=\sigma^1\otimes\alpha^1\otimes&\sigma^1-\sigma^2\otimes\alpha^1\otimes\sigma^2,\label{eq:brauch16}\\
\end{align}
W. l. o. g. we can assume that $\gamma=\gamma^{-1}+a\gamma_{s1}+b\gamma_{s2}+\kappa\gamma_a$ for some $a,b\in\mathds R$, $\kappa\neq 0$.
For $S=S^\pm(0,\mu)$ we have
\[\lambda(S)S^*\circ\gamma_0(a_2)_s\circ S=-\mu^2\big(a(\sigma^1\otimes\sigma^1+\sigma^2\otimes\sigma^2)\pm b(\sigma^1\otimes\sigma^2+\sigma^2\otimes\sigma^1)\big)\]
and for $S=S^\pm(\mu,0)$ we obtain 
\[\lambda(S)S^*\circ\gamma_0(a_2)_s\circ S=\mu^2\big(a(\sigma^1\otimes\sigma^1+\sigma^2\otimes\sigma^2)\pm b(\sigma^1\otimes\sigma^2+\sigma^2\otimes\sigma^1)\big).\]
For the skewsymmetric tensor $\gamma_0(a_2)_a$ of $\gamma_0(a_2)$ and $S\in\tilde{\operatorname{CO}}(1,1)$ it holds
\[\lambda(S)S^*\circ\gamma_0(a_2)_a\circ S=\lambda(S)\det(S)\gamma_0(a_2)_a.\]
Since $S_1=S^+(\mu,0)\in \tilde{\operatorname{CO}}(1,1)$ and $S_2=S^-(0,\mu)\in\tilde{\operatorname{CO}}(1,1)$ satisfy 
$\lambda(S_1)\det(S_1)=\mu^2$ and $\lambda(S_2)\det(S_2)=-\mu^2$, there is an $(S,U)\in G$ for $a=b=0$ such that
\[(S,U)^*\gamma=\gamma^{-1}+\gamma_a.\]

Now, assume $a=\pm b$. Using $S=S^\pm(0,\mu)$ or $S=S^\pm(\mu,0)$ we can assume w. l. o. g. that $a=\pm b=1$.
For every $\mu>0$ there is a conformal map $S^+(s_1,s_2)$ with respect to  $\gamma_0(a_1)$ satisfying $(s_1+s_2)^2=1$, whose determinant equals $\mu$.
Thus, there is an $(S,U)\in G$ such that 
\[(S,U)^*\gamma=\gamma^{-1}+\gamma_{s1}+\gamma_{s2}\pm\gamma_a.\]
Now, assume $S\in \tilde{\operatorname{CO}}(1,1)$, $S=S^\pm(s_1,s_2)$ and $s_1^2-s_2^2\neq 0$.
Since \[\lambda(S)S^*\circ(\sigma^1+\sigma^2)\otimes(\sigma^1+\sigma^2)\circ S=\sign(s_1^2-s_2^2)(s_1+s_2)^2(\sigma^1\pm\sigma^2)\otimes(\sigma^1\pm\sigma^2)\]
the equation
\[\lambda(S)S^*\circ(\sigma^1+\sigma^2)\otimes(\sigma^1+\sigma^2)\circ S=(\sigma^1\pm\sigma^2)\otimes(\sigma^1\pm\sigma^2)\]
implies $S=S^+(s_1,s_2)$, $\det S>0$ and $\lambda(S)=1$.
Hence there is no $(S,U)\in G$ such that 
\[(S,U)^*(\gamma^{-1}+\gamma_{s1}+\gamma_{s2}+\gamma_a)=\gamma^{-1}+\gamma_{s1}+\gamma_{s2}-\gamma_a.\]

Now, assume $\gamma=\gamma^{-1}+a\gamma_{s1}+b\gamma_{s2}+\kappa\gamma_a$ for $a\neq\pm b$ and $\kappa\neq 0$.
At first, we choose $S=S^+(s_1,s_2)$ mit $s_1^2a+2bs_1s_2+as_2^2=0$ for $a^2-b^2>0$ and $bs_2^2+2as_1s_2+bs_1^2=0$ for $a^2-b^2>0$ respectively.
We set for instance $s_3=-\sqrt{s_1^2+1}$ and $s_1$ as the solution of $2s_1^2a+a-2bs_1\sqrt{s_1^2+1}=0$ for $a^2<b^2$. 
Then 
\[S^*\circ\gamma_0(a_2)_s\circ S=\mu(\sigma^1\otimes\sigma^2+\sigma^2\otimes\sigma^1)\] if $a^2<b^2$ and \[S^*\circ\gamma_0(a_2)_s\circ S=\mu(\sigma^1\otimes\sigma^1+\sigma^2\otimes\sigma^2)\] if $a^2>b^2$
for some suitable $\mu\neq 0$.
Using $S^\pm(\sqrt{|\mu|},0)$ or $S^\pm(0,\sqrt{|\mu|})$ respectively we obtain that there is an $(S,U)\in G$ and a $\tilde\kappa>0$ such that
\[(S,U)^*\gamma=\gamma^{-1}+\gamma_{s_2}+\tilde\kappa\gamma_a\] if $a^2<b^2$ and 
\[(S,U)^*\gamma=\gamma^{-1}+\gamma_{s_1}+\tilde\kappa\gamma_a\] if $a^2<b^2$.
Moreover, there is an $(S,U)\in G$ satisfying 
\[(S,U)^*(\gamma^{-1}+\gamma_{s_1}+\kappa_1\gamma_a)=\gamma^{-1}+\gamma_{s_1}+\kappa_2\gamma_a\]
for $\kappa_1,\kappa_2>0$, if and only if there is an $S$, which is conformal with respect to $\gamma_0(a_1)=\sigma^1\otimes\sigma^1-\sigma^2\otimes\sigma^2$ and orthogonal with respect to $\gamma_0(a_2)_s$, satisfying $\det(S)\kappa_1=\kappa_2$. From the orthogonality of $S$ we get $\det S=\pm 1$ and obtain $\kappa_1=\kappa_2>0$.
There is an $(S,U)\in G$ satisfying \[(S,U)^*(\gamma^{-1}+\gamma_{s_2}+\kappa_1\gamma_a)=\gamma^{-1}+\gamma_{s_2}+\kappa_2\gamma_a\]
for $\kappa_1,\kappa_2>0$, if and only if there is either a conformal map $S_1$ with respect to $\gamma_0(a_1)$, which is orthogonal with respect to $\sigma^1\otimes\sigma^2+\sigma^2\otimes\sigma^1$ and satisfies $\det(S_1)\kappa_1=\kappa_2$, or a map $S_2\in \tilde{\operatorname{CO}}(1,1)$, which is not conformal with respect to $\gamma_0(a_1)$ and satisfies \[\lambda(S_2)S_2^*\circ(\sigma^1\otimes\sigma^2+\sigma^2\otimes\sigma^1)\circ S_2=\sigma^1\otimes\sigma^2+\sigma^2\otimes\sigma^1\] and $-\det(S_2)\kappa_1=\kappa_2$.
Since $S_1$ respectively $(\sigma^1\otimes e_1-\sigma^2\otimes e_2)\circ S_2$  is orthogonal with respect to $\sigma^1\otimes\sigma^2+\sigma^2\otimes\sigma^1$ it follows $\det(S_1)=\pm 1$ or $\det(S_2)=\pm 1$ respectively. Thus we obtain $\kappa_1=\kappa_2>0$.

Using equations (\ref{eq:12}) and (\ref{eq:13}) we get 
$[\gamma,\epsilon,0]=[\gamma,0,0]$ for every $\epsilon\in C^2(\mathfrak l,\mathfrak l^*)$ and \[\gamma\in\{\gamma^{-1}+\gamma_a,\gamma^{-1}+\gamma_{s1}+\kappa\gamma_a,\gamma^{-1}+\gamma_{s2}+\kappa\gamma_a,\gamma^{-1}+\gamma_{s1}+\gamma_{s2}\pm\gamma_a\mid \kappa>0\}.\]

Now, we show that the $G$-orbits of 
$[\gamma^{-1}+\gamma_a,0,0]$,
$[\gamma^{-1}+\gamma_{s1}+\kappa\gamma_a,0,0]$,
$[\gamma^{-1}+\gamma_{s2}+\kappa\gamma_a,0,0]$ and 
$[\gamma^{-1}+\gamma_{s1}+\gamma_{s2}\pm\gamma_a,0,0]$
for $\kappa>0$ are pairwise distinct.
Therefor, it remains to show that the $G$-orbits of 
$[\gamma,0,0]$ and $[\tilde\gamma,0,0]$ for $\gamma_0(a_1)=\tilde\gamma_0(a_1)=\gamma^{-1}(a_1)$ and
$\gamma_0(a_2)_s,\tilde\gamma_0(a_2)\in\{\gamma_{s1},\gamma_{s2},\gamma_{s1}+\gamma_{s2},0\}$ where $\gamma_0(a_2)_s\neq\tilde\gamma_0(a_2)_s$ are distinct.

At first, we note that for $\gamma_0(a_1)=\sigma^1\otimes\sigma^1-\sigma^2\otimes\sigma^2$ and $\gamma_0(a_2)_s=\sigma^1\otimes\sigma^1+\sigma^2\otimes\sigma^2$ there is an $A\in\mathfrak a$, $A\neq 0$ such that $\gamma_0(A)_s$ is a positive definite bilinear form.
If $\gamma_0(a_2)_s=\sigma^1\otimes\sigma^2+\sigma^2\otimes\sigma^1$, then $\gamma_0(A)_s$ is non-degenerate for all $A\in\mathfrak a$, $A\neq 0$, but also not positive definite.
For $\gamma_0(a_2)_s=0$ we have that $\gamma_0(A)_s$ is not positive definite for all $A\in\mathfrak a$ and that there is an $A_2\in \mathfrak a$, $A_2\neq 0$ which satisfies $\gamma_0(A_2)_s=0$.
Finally, for $\gamma_0(a_2)_s=\gamma_{s1}(a_2)+\gamma_{s2}(a_2)$ we have $\gamma_0(A)_s\neq 0$ and not positive definite for all $A\in\mathfrak a$ and that there is an $A_2\in \mathfrak a$, $A_2\neq 0$ such that $\gamma_0(A_2)_s\neq 0$ is degenerate.
Hence the corresponding $G$-orbits are pairwise distinct.
\\\quad\\
Now, let $(\gamma,\epsilon,0)$ be a nilpotent cocycle and let $\gamma_0(A)$ be symmetric for all $A\in\mathfrak a$. 
Then there is a $\tilde A\in \mathfrak a$ such that $\gamma_0(\tilde A)$ has signature $(1,1)$.
This means that we can w. l. o. g. assume that
$\gamma_0(a_1)=\sigma^1\otimes\sigma^1-\sigma^2\otimes\sigma^2$ and $\gamma_0(a_2)$ is symmetric with $\ker\gamma_0=\{0\}$.

The basis transformations for $\gamma_0(a_1)=\sigma^1\otimes\sigma^1-\sigma^2\otimes\sigma^2$ and not symmetric $\gamma_0(a_2)$ are also valid for symmetric $\gamma_0(a_2)$.
Thus, we only have to consider $(\gamma,\epsilon,0)$ for \[\gamma\in\{\gamma^{-1}+\gamma_{s1},\gamma^{-1}+\gamma_{s_2},\gamma^{-1}+\gamma_{s1}+\gamma_{s_2}\}.\]
Here the case $\gamma_0(a_2)=0$ omits since $\ker\gamma_0=\{0\}$.

Using equations (\ref{eq:12}) and (\ref{eq:13}) we obtain 
$[\gamma,\epsilon,0]=[\gamma,0,0]$ for all $\epsilon\in C^2(\mathfrak l,\mathfrak l^*)$ and all 
$\gamma\in\{\gamma^{-1}+\gamma_{s1},\gamma^{-1}+\gamma_{s2},\gamma^{-1}+\gamma_{s1}+\gamma_{s2}\}.$

Moreover, the same argumentation as is the previous case, where $\gamma_0(a_1)=\sigma^1\otimes\sigma^1-\sigma^2\otimes\sigma^2$ and  $\gamma_0(a_2)$ is non-symmetric, yields that the $G$-orbits of 
$[\gamma^{-1}+\gamma_{s1},0,0]$,
$[\gamma^{-1}+\gamma_{s2},0,0]$ and
$[\gamma^{-1}+\gamma_{s1}+\gamma_{s2},0,0]$ are pairwise distinct.
So, we have completeley solved the case $\xi=0$.
\\\quad\\
Now, assume $\xi\neq 0$.
For a given $(\gamma,\epsilon,\xi)\in Z^2(\mathfrak l,\mathfrak a,\omega_\mathfrak a)_0$ there is always a vector $v_1\in\mathfrak l$ and a Darboux basis of $\mathfrak a$ satisfying \[\xi(v_1)=\begin{pmatrix}0&1\\0&0\end{pmatrix}\in \Hom(\mathfrak a).\]
Since $\xi(L)$ is nilpotent for all $L\in\mathfrak l$, the elements $\xi(L)$ and $\xi(v_1)$ are linearly dependent in $\Hom(\mathfrak a)$. Thus, we can choose a $v_2\neq 0\in\mathfrak l$ satisfying $\xi(v_2)=0\in\Hom(\mathfrak a)$.
Hence, there is a basis $\{v_1,v_2\}$ of $\mathfrak l$ and a Darboux basis of $\mathfrak a$ such that
\begin{align*}
\xi(v_1)=\begin{pmatrix}0&1\\0&0\end{pmatrix}\in \Hom(\mathfrak a)\quad \text{and}\quad \xi(v_2)=\begin{pmatrix}0&0\\0&0\end{pmatrix}\in \Hom(\mathfrak a).
\end{align*}

In the following, we consider $(\gamma,\epsilon,\xi^0)$ where 
\begin{align}\label{eq:brauch8}
\xi^0=\sigma^1\otimes\alpha^2\otimes a_1.
\end{align}
For this triple $(\gamma,\epsilon,\xi^0)$ we have $\beta=0$. Moreover, it is a cocycle if and only if $\gamma(e_2)(a_1)=0\in\mathfrak l^*$, since this equation is equivalent to the only non-trivial condition (\ref{eq:lemma3}).
\begin{lemma}
A triple $(\gamma,\epsilon,\xi^0)\in Z^2(\mathfrak l,\mathfrak a,\omega_\mathfrak{a})$ is balanced (and hence nilpotent) if and only if $\gamma(e_1)(a_1)\neq 0$ and one of the following three conditions holds:
\[\gamma(e_2)\neq 0\in\Hom(\mathfrak a,\mathfrak l^*),\quad \gamma(e_1,a_1,e_2)\neq 0\quad\text{ or }\quad\gamma(e_1,a_2,e_2)^2\neq \epsilon(e_1,e_2)(e_2).\]
\end{lemma}

\begin{proof}
Because of $\beta=0$ we have
\[\mathfrak z(\mathfrak a)\cap\mathfrak a^\mathfrak l_{\xi^0}\cap\ker\beta_0\cap\ker\gamma_0=\Span\{a_1\}\cap\ker\gamma_0.\]
Thus, condition (b) on page \pageref{eq:aundb} is satisfied, if and only if $\gamma(e_1)(a_1)\neq 0$.

Now, we show that the three given cases are the only ones, which satisfy condition (a) under the assumtion $\gamma(e_1)(a_1)\neq 0$.
Because of $\ad_\mathfrak a=0$ and $\beta=0$ we have 
\[\mathds L=\Span\{e_2\}\cap\ker\gamma\cap\{L\in\mathfrak l\mid \exists A\in\mathfrak a:\text {(a.3)-(a.4) holds}\}.\]
Thus, condition (a) is satisfied, if and only if $\gamma(e_2)\neq 0\in\Hom(\mathfrak a,\mathfrak l^*)$ or if there is no $A\in\mathfrak a$ for $e_2$ such that (a.3) and (a.4) hold.
The latter holds for $\gamma(e_2)=0$, if and only if 
\begin{align}\label{eq:AeqKlbla}
\alpha(e_2,e_1)=\xi^0(e_1)A\quad\text{ and }\quad \epsilon(e_2,e_1)=\gamma(e_1)A
\end{align}
is not satisfied simultaneously for any $A\in\mathfrak a$.
Using 
$\omega_\mathfrak a(\alpha(e_1,e_2),a_i)=\gamma(e_1,a_i,e_2)-\gamma(e_2,a_i,e_1)=\gamma(e_1,a_i,e_2)$
yields
\[\alpha(e_1,e_2)=\gamma(e_1,a_2,e_2)a_1-\gamma(e_1,a_1,e_2)a_2=:x_4a_1-x_3a_2.\]
Since, moreover, $\xi^0(e_1)A$ lies in $\Span\{a_1\}$, the first equation in (\ref{eq:AeqKlbla}) is equivalent to $x_3=0$ and $A=A_1a_1-x_4a_2$ for some $A_1\in\mathds R$. Hence the equations in (\ref{eq:AeqKlbla}) are satisfied, if and only if additionally $\epsilon(e_1,e_2)e_2=x_4^2$ holds and we obtain our assertion.
\end{proof}

Because of $\beta=0$ and $\ad_\mathfrak a=0$ the cocycles $(\gamma,\epsilon,\xi^0), (\gamma',\epsilon',\xi^0)\in Z^2(\mathfrak l,\mathfrak a,\omega_\mathfrak a)_0$ are equivalent, if and only if 
\begin{align}\label{eq:95a}
\gamma'_0(a_j)(e_i)=\gamma_0(a_j)(e_i)+\tau^*(\xi^0(e_i)a_j)
\end{align}
for $i,j=1,2$ and \[\epsilon'(e_1,e_2)=\epsilon(e_1,e_2)+\tau^*(\alpha(e_1,e_2))-\ev(\gamma\wedge\tau)(e_1,e_2)-\tau^*\circ \rd_{\xi^0}\tau(e_1,e_2)\]
for some $\tau\in C^1(\mathfrak l,\mathfrak a)$.
Using $\tau^*(\alpha(e_1,e_2))=-\gamma(e_1,\tau(\cdot),e_2)+\gamma(e_2,\tau(\cdot),e_1)$ yields 
\begin{align}\label{eq:95b}
\epsilon'(e_1,e_2)=\epsilon(e_1,e_2)&-\gamma(e_1,\tau(\cdot),e_2)+\gamma(e_2,\tau(\cdot),e_1)-\gamma(e_1,\tau(e_2),\cdot)\\
&+\gamma(e_2,\tau(e_1),\cdot)-\omega_\mathfrak a(\tau\cdot,\xi^0(e_1)\tau(e_2)).\nonumber
\end{align}
Especially, we have $\gamma'_0(A)(L)=\gamma_0(A)(L)$ for equivalent $(\gamma',\epsilon',\xi^0)$ and $(\gamma,\epsilon,\xi^0)$, if $A\in\ker\xi^0(L)$.

We define linear maps \[S(a,s_3,s_4):=\begin{pmatrix}\frac{1}{a^2}&0\\s_3&s_4\end{pmatrix}\text{ and }U(a,b):=\begin{pmatrix}a&b\\0&\frac{1}{a}\end{pmatrix}\]
with respect to the basis $\{e_1,e_2\}$ of $\mathfrak l$ and the Darboux basis $\{a_1,a_2\}$ of $\mathfrak a$.
It holds $(S,U)^*\xi^0=\xi^0$ if and only if $S=S(a,s_3,s_4)$ and $U=U(a,b)$ for $a,s_4\neq 0$ and $s_3\in\mathds{R}$.

In the following let us assume that $(\gamma,\epsilon,\xi^0)\in Z^2(\mathfrak l,\mathfrak a,\omega_\mathfrak a)_0$ satisfies $\gamma(e_1,a_1,e_2)\neq 0$.
Moreover, we can w. l. o. g. assume that $\gamma_0(a_1)=\sigma^1\otimes\sigma^2$ holds, since we can consider $\gamma_0(a_1)$ with respect to the basis $\{e_1+s_3e_2,s_4e_2\}$ of $\mathfrak l$ for some suitable $s_3\in\mathds R$ and $s_4\neq 0$, otherwise.
Furthermore, it is easy to see that 
\[(S(a,s_3,s_4),U(a,b))^*\gamma_0(a_1)=\gamma_0(a_1)\]
is satisfied if and only if $s_3=0$ and $s_4=a^3$ hold.
Assume \[\gamma_0(a_2)=b_1\sigma^1\otimes\sigma^1+b_2\sigma^2\otimes\sigma^1+b_3\sigma^1\otimes\sigma^2+b_4\sigma^2\otimes\sigma^2.\]

Motivated by the action of $C^1(\mathfrak l,\mathfrak a)$ on the cocycles, we define the linear maps $\gamma_{1,x,y}:\mathfrak l\rightarrow\Hom(\mathfrak a,\mathfrak l^*)$ for $x,y\in\mathds R$ by 
\begin{align}\label{eq:brauch9}
\gamma_{1,x,y}:=\sigma^1\otimes\alpha^1\otimes\sigma^2+x\sigma^2\otimes\alpha^2\otimes\sigma^1+y\sigma^2\otimes\alpha^2\otimes\sigma^2.
\end{align}
Using the equations (\ref{eq:95a}) and (\ref{eq:95b}) yields that the equivalence class of $(\gamma_{1,x,y},0,\xi^0)$ consists of all cocycles $(\gamma_{1,x,y}-\sigma^1\otimes\alpha^2\otimes(\tau_3\sigma^1+\tau_4\sigma^2),\epsilon,\xi^0)$
where $\epsilon\in C^2(\mathfrak l,\mathfrak l^*)$ and $\tau_3,\tau_4\in\mathds R$.
Moreover, for $(S,U)=(S(a,0,a^3),U(a,-ba^{-1}))$, $a\neq 0$ we have
\begin{align}\label{eq:SUgamma1}
(S,U)^*\gamma_0(a_2)=(b+b_1a^{-3})\sigma^1\otimes\sigma^1+b_2a^2\sigma^2\otimes\sigma^1+b_3a^2\sigma^1\otimes\sigma^2+b_4a^7\sigma^2\otimes\sigma^2.
\end{align}
Thus, we obtain that $[\gamma,\epsilon,\xi^0]$ lies in the $G$-orbit of exactly one of the cohomology classes
\[[\gamma_{1,\lambda,1},0,\xi^0],\quad [\gamma_{1,1,0},0,\xi^0],\quad [\gamma_{1,-1,0},0,\xi^0],\quad [\gamma_{1,0,0},0,\xi^0]\]
with $\lambda\in\mathds R$.
\\\quad\\
Now, assume $\gamma(e_1,a_1,e_2)=0$. Then $\gamma(e_1,a_1,e_1)\neq 0$ since $(\gamma,\epsilon,\xi^0)$ is balanced.
We can w. l. o. g. assume that $\gamma_0(a_1)=\sigma^1\otimes\sigma^1$. Otherwise, we consider $(S(a,0,a^{-2}),U(a,0))^*(\gamma,\epsilon,\xi^0)$ for some suitable $a\neq 0$.
It is easy to see that $(S(a,s_3,s_4),U(a,b))^*\gamma_0(a_1)=\sigma^1\otimes\sigma^1$ is equivalent to  $a=1$.

Motivated by the action of $C^1(\mathfrak l,\mathfrak a)$ on the cocycles, we define the linear maps $\gamma^1_{x,y}:\mathfrak l\rightarrow\Hom(\mathfrak a,\mathfrak l^*)$ for $x,y\in\mathds R$ by 
\begin{align}\label{eq:brauch10}
\gamma^1_{x,y}:=\sigma^1\otimes\alpha^1\otimes\sigma^1+x\sigma^2\otimes\alpha^2\otimes\sigma^1+y\sigma^2\otimes\alpha^2\otimes\sigma^2.
\end{align}
Analogous to $\gamma_{1,x,y}$ we obtain that $(\gamma^1_{x,y},\epsilon,\xi^0)$ and $(\gamma',\epsilon',\xi^0)$ are equivalent, if and only if
\[\gamma'=\gamma^1_{x,y}+\sigma^1\otimes\alpha^2\otimes(-\tau_3\sigma^1-\tau_4\sigma^2)\] and \[\epsilon'(e_1,e_2)(e_2)=\epsilon(e_1,e_2)(e_2)+x\tau_4+\tau_4^2+y\tau_3\]
for some suitable $\tau_3,\tau_4\in\mathds R$.

So, it follows that every $(\gamma,\epsilon,\xi^0)$ satisfying $\gamma_0(a_1)=\sigma^1\otimes\sigma^1$ is equivalent to a cocycle $(\gamma^1_{x,y},t\sigma^1\wedge\sigma^2\otimes\sigma^2,\xi^0)$ for a suitable $x,y,t\in \mathds R$, where $t\neq 0$ holds for $x=y=0$, since the cocycle is balanced.
It remains to investigate which cohomology classes $[\gamma^1_{x,y},t\sigma^1\wedge\sigma^2\otimes\sigma^2,\xi^0]$ lie in distinct $G$-orbits.
Therefor, we compute 
\begin{align}\label{eq:96}
(S(1,s_3,s_4),U(1,b))^*\gamma^1_{x,y}=&\,\sigma^1\otimes\alpha^1\otimes\sigma^1+
(s_4x+s_3s_4y)\sigma^2\otimes\alpha^2\otimes\sigma^1+s_4^2y\sigma^2\otimes\alpha^2\otimes\sigma^2\nonumber\\
&+(-b+s_3x+s_3^2y)\sigma^1\otimes\alpha^2\otimes\sigma^1+s_3s_4y\sigma^1\otimes\alpha^2\otimes\sigma^2
\end{align}
and $(S(1,s_3,s_4),U(1,b))^*\epsilon(e_1,e_2)(e_2)=s_4^2\epsilon(e_1,e_2)(e_2)$.

So, it follows that $[\gamma^1_{0,0},\epsilon,\xi^0]$ lies in the $G$-orbit of either $[\gamma^1_{0,0},\sigma^1\wedge\sigma^2\otimes\sigma^2,\xi^0]$ or $[\gamma^1_{0,0},-\sigma^1\wedge\sigma^2\otimes\sigma^2,\xi^0]$.
Moreover, for every $x\neq 0$ there is exactly one $t\in\mathds R$ such that
$[\gamma^1_{x,0},\epsilon,\xi^0]$ and $[\gamma^1_{1,0},t\sigma^1\wedge\sigma^2\otimes\sigma^2,\xi^0]$ lie in the same $G$-orbit.
Finally, the cohomology class $[\gamma^1_{x,y},\epsilon,\xi^0]$ for $x\in\mathds R$, $y\neq 0$ lies in the $G$-orbit of either
$[\gamma^1_{0,1},0,\xi^0]$ or $[\gamma^1_{0,-1},0,\xi^0]$.
Altogether, we obtain by using equation (\ref{eq:96}) that every $(\gamma,\epsilon,\xi^0)$ satisfying $((\gamma_0(a_1))(e_1))(e_2)=0$ lies in the $G$-orbit of exactly one of the cohomology classes
\[[\gamma^1_{0,0},\pm\sigma^1\wedge\sigma^2\otimes\sigma^2,\xi^0], [\gamma^1_{1,0},t\sigma^1\wedge\sigma^2\otimes\sigma^2,\xi^0], [\gamma^1_{0,\pm},0,\xi^0]\]
with $t\in\mathds R$.
\\\quad\\
%Theorem zu \mathfrak l=\mathds R^2,\mathfrak a=\mathds R^2
We summerize the results of this section in the following lemma. 

\begin{lemma}
The cocycles 
\begin{align*}
&[\gamma_{1,t,1},0,\xi^0],[\gamma_{1,1,0},0,\xi^0],[\gamma_{1,-1,0},0,\xi^0],[\gamma_{1,0,0},0,\xi^0],[\gamma^1_{0,\pm 1},0,\xi^0],\\
&[\gamma^1_{1,0},t\sigma^1\wedge\sigma^2\otimes\sigma^2,\xi^0],[\gamma^1_{0,0},\pm\sigma^1\wedge\sigma^2\otimes\sigma^2,\xi^0],[\gamma^1_{t_1},0,0],\\
&[\gamma^0_t,0,0], [\gamma^0_{-3},\sigma^1\wedge\sigma^2\otimes\sigma^2,0],[\gamma^0_\pm,0,0],[\gamma^{-1}+\gamma_a,0,0],\\
&[\gamma^{-1}+\gamma_{s1}+\kappa\gamma_a,0,0],[\gamma^{-1}+\gamma_{s2}+\kappa\gamma_a,0,0],[\gamma^{-1}+\gamma_{s1}+\gamma_{s2}\pm\gamma_a,0,0]
\end{align*}
where $t\in\mathds R$, $t_1\geq 0$ and $\kappa>0$
form a system of representatives of $H^2_{Q}(\mathds R^2,\mathds R^2,\alpha^1\wedge\alpha^2)_0/G$.
\end{lemma}

\subsection{Nilpotent symplectic Lie algebras of dimension six}

We conclude the calculations in section \ref{NSLA:1} - \ref{NSLA:5} in the following list.
For the definitions of the maps given in the theorem see the formulas
(\ref{eq:brauch1}) - (\ref{eq:brauch4}), (\ref{eq:brauch6}), (\ref{eq:brauch5}), (\ref{eq:brauch7}),
(\ref{eq:brauch11}) - (\ref{eq:brauch12}),
(\ref{eq:brauch13}) - (\ref{eq:brauch14}),
(\ref{eq:brauch15}) - (\ref{eq:brauch16}),
(\ref{eq:brauch8}), (\ref{eq:brauch9}), (\ref{eq:brauch10}).

\begin{theorem}
Every six-dimensional non-abelian nilpotent symplectic Lie algebra is isomorphic to $\mathfrak d_{\gamma,\epsilon,\xi}(\mathfrak l,\mathfrak a)$ for exactly one of the following triples:\\
$\mathfrak l=\mathds R^3$, $\mathfrak a=0$
\begin{itemize}
\item $\epsilon=-(1+b)\sigma^1\wedge\sigma^2\otimes\sigma^3-b\sigma^1\wedge\sigma^3\otimes\sigma^2+\sigma^2\wedge\sigma^3\otimes\sigma^1$, $\gamma=0$, $\xi=0$, $b\in[0,1]$,
\item $\epsilon=-2\sigma^1\wedge\sigma^2\otimes\sigma^3-\sigma^1\wedge\sigma^3\otimes(\sigma^1+\sigma^2)+\sigma^2\wedge\sigma^3\otimes\sigma^1$, $\gamma=0$, $\xi=0$,
\item $\epsilon=-2\sigma^1\wedge\sigma^2\otimes\sigma^3-\sigma^1\wedge\sigma^3\otimes(z\sigma^1+\sigma^2)+\sigma^2\wedge\sigma^3\otimes(\sigma^1-z\sigma^2)$, $\gamma=0$, $\xi=0$, $z>0$,
\item $\epsilon=-\sigma^1\wedge\sigma^3\otimes\sigma^2-\sigma^2\wedge\sigma^3\otimes\sigma^1$, $\gamma=0$, $\xi=0$,
\item $\epsilon=-\sigma^1\wedge\sigma^3\otimes\sigma^1+\sigma^1\wedge\sigma^2\otimes\sigma^2$, $\gamma=0$, $\xi=0$,
\end{itemize}
$\mathfrak l=\mathds R$, $\mathfrak a=\mathds R^4$ where $\omega_\mathfrak a=\alpha^1\wedge\alpha^4+\alpha^2\wedge\alpha^3$ and $\{\alpha^1,\dots,\alpha^4\}$ denotes the dual basis of the standard basis $\{a_1,\dots,a_4\}$ of $\mathfrak a$
\begin{itemize}
\item $(\gamma,\epsilon,\xi)=(\gamma^0,0,\xi_1)$,
\item $(\gamma,\epsilon,\xi)=(\gamma^0,0,\xi_\pm)$,
\item $(\gamma,\epsilon,\xi)=(\gamma^0,0,\xi^\kappa)$,\quad $\kappa\in\mathds R$
\end{itemize}
$\mathfrak l=\mathds R$, $\mathfrak a=\mathfrak h_3\oplus\mathds R=\{a_1,\dots,a_4 \mid [a_1,a_2]=a_3\}$ where $\omega_\mathfrak a=\alpha^1\wedge\alpha^4+\alpha^2\wedge\alpha^3$ and $\{\alpha^1,\dots,\alpha^4\}$ denotes the dual basis of the basis $\{a_1,\dots,a_4\}$ of $\mathfrak a$
\begin{itemize}
\item $(\gamma,\epsilon,\xi)=(\sigma^1\otimes\alpha^3\otimes\sigma^1,0,\xi^1)$,
\item 
$(\gamma,\epsilon,\xi)=(\sigma^1\otimes(l+1)\alpha^3\otimes\sigma^1,0,\xi^1(l))$,\quad $l\notin\{0,-1\}$,
\item $(\gamma,\epsilon,\xi)=(\sigma^1\otimes(y_1\alpha^1+\alpha^3)\otimes\sigma^1,0,\xi^1(0))$,\quad $y_1\in\mathds R$,
\item $(\gamma,\epsilon,\xi)=(\sigma^1\otimes(y_2\alpha^2+y_4\alpha^4)\otimes\sigma^1,0,\xi^{1,\pm 1})$,\quad $y_2\geq 0$, $y_4\neq 0$,
\item $(\gamma,\epsilon,\xi)=(\sigma^1\otimes(y_2\alpha^2\pm\alpha^4)\otimes\sigma^1,0,\xi^{1,0})$,\quad $y_2\geq 0$.
\end{itemize}
$\mathfrak l=\mathds R^2$, $\mathfrak a=\mathds R^2$, $\omega_\mathfrak a=\alpha^1\wedge\alpha^2$  where $\{\alpha^1,\alpha^2\}$ denotes the dual basis of the standard basis $\{a_1,a_2\}$ of $\mathfrak a$
\begin{itemize}
\item $(\gamma,\epsilon,\xi)=(\gamma_{1,t,1},0,\xi^0)$,\quad $t\in\mathds R$,
\item $(\gamma,\epsilon,\xi)=(\gamma_{1,\pm 1,0},0,\xi^0)$,
\item $(\gamma,\epsilon,\xi)=(\gamma_{1,0,0},0,\xi^0)$,
\item $(\gamma,\epsilon,\xi)=(\gamma^1_{0,\pm 1},0,\xi^0)$,
\item $(\gamma,\epsilon,\xi)=(\gamma^1_{1,0},t\sigma^1\wedge\sigma^2\otimes\sigma^2,\xi^0)$, \quad $t\in\mathds R$,
\item $(\gamma,\epsilon,\xi)=(\gamma^1_{0,0},\pm\sigma^1\wedge\sigma^2\otimes\sigma^2,\xi^0)$,
\item $(\gamma,\epsilon,\xi)=(\gamma^1_{t_1},0,0)$,\quad $t_1\geq 0$,
\item $(\gamma,\epsilon,\xi)=(\gamma^0_t,0,0)$,\quad $t\in\mathds R$,
\item $(\gamma,\epsilon,\xi)=(\gamma^0_{-3},\sigma^1\wedge\sigma^2\otimes\sigma^2,0)$,
\item $(\gamma,\epsilon,\xi)=(\gamma^0_\pm,0,0)$,
\item $(\gamma,\epsilon,\xi)=(\gamma^{-1}+\gamma_a,0,0)$,
\item $(\gamma,\epsilon,\xi)=(\gamma^{-1}+\gamma_{s1}+\kappa\gamma_a,0,0)$,\quad $\kappa>0$,
\item $(\gamma,\epsilon,\xi)=(\gamma^{-1}+\gamma_{s2}+\kappa\gamma_a,0,0)$,\quad $\kappa>0$,
\item $(\gamma,\epsilon,\xi)=(\gamma^{-1}+\gamma_{s1}+\gamma_{s2}\pm\gamma_a,0,0)$.
\end{itemize}
\end{theorem}

% %%%%%%%%%%%%%% B I B L I O G R A P H I E

\end{document}